\documentclass[a4paper,UKenglish,cleveref, autoref, thm-restate]{lipics-v2021}

\title{Solving one variable word equations in the free group in cubic time}

\author{Robert Ferens}{Institute of Computer Science, University of Wroc\l{}aw, Poland}{robert.ferens@cs.uni.wroc.pl}{0000-0002-0079-1936}{}
\author{Artur Je\.z}{Institute of Computer Science, University of Wroc\l{}aw, Poland \and \url{https://ii.uni.wroc.pl/~aje/}}{aje@cs.uni.wroc.pl}{0000-0003-4321-3105}{}

\authorrunning{R.~Ferens and A.~Je\.z}



\funding{This work was supported under National Science Centre (NCN), Poland project number 2017/26/E/ST6/00191.}

\nolinenumbers 

\hideLIPIcs  

\newcommand{\iref}[1]{\hyperref[i#1]{i#1}\xspace}
\newcommand{\jref}[1]{\hyperref[j#1]{j#1}\xspace}
\usepackage{xspace}
\usepackage{verbatim}
\usepackage{xcolor}
\usepackage{todonotes}

\newcommand{\NPclass}{{\sf NP}}
\newcommand{\PSPACE}{{\sf PSPACE}}

\newcommand{\pref}{{\ensuremath{\sqsubseteq} }}
\newcommand{\suff}{{\ensuremath{\sqsupseteq} }}
\DeclareMathOperator{\nf}{nf}
\DeclareMathOperator{\lce}{lce}

\DeclareMathOperator{\lcp}{lcp}

\newcommand{\ov}[1]{\ensuremath{\overline{#1}} }
\newcommand{\twodots}{\mathinner{\ldotp\ldotp}}
\newcommand{\shift}{\sim}
\newcommand{\eqg}{\approx}

\newcommand{\Ocomp}{\mathcal O}

\newcommand{\inv}{^{-1}}

\begin{document}
\maketitle
 
\begin{abstract}
A word equation with one variable in a free group is given as $U = V$,
where both $U$ and $V$ are words over the alphabet of generators of the free group and $X, X \inv$, for a fixed variable $X$.
An~element of the free group is a solution when substituting it for $X$ yields a true equality (interpreted in the free group) of left- and right-hand sides.
It is known that the set of all solutions of a given word equation with one variable is a finite union of sets of the form
$\{\alpha w^i \beta \: : \: i \in \mathbb Z \}$,
where $\alpha, w, \beta$ are reduced words over the alphabet of generators,
and a polynomial-time algorithm (of a high degree)
computing this set is known.
We provide a cubic time algorithm for this problem, which also shows that the
set of solutions consists of at most a quadratic number of the above-mentioned sets.
The algorithm uses only simple tools of word combinatorics and group theory
and is simple to state.
Its analysis is involved and focuses on the combinatorics of occurrences of powers of a word
within a larger word.
\end{abstract}


\section{Introduction}
\subparagraph{Word equations in the free group}
A word equation is a formal equation $U = V$ in which both $U, V$ contain letters from a fixed set (called alphabet) $\Sigma$ and variables;
a solution is a substitution of variables by words over $\Sigma$ such that
this formal equation is turned into an equality.
We consider such equations in a free group, so the aforementioned equality is
interpreted as the equality in the free group generated by $\Sigma$;
naturally, we allow the~usage of~inverses of variables and generators in the equations.
The satisfiability problem (of word equation over the free group)
is to decide, whether the input equation has a solution.
By~solving the equation we mean to return an (explicit or effective)
representation of all solutions.

The first algorithm for the satisfiability problem was given by Makanin~\cite{mak82}
and it is an involved generalization of Makanin's algorithm for the satisfiability of word equation in~the~free monoid~\cite{makanin};
Razborov generalized the~algorithm so that it solves word equations in~the~free group~\cite{raz87};
the description is infinite and is known as Makanin-Razborov diagrams.
Makanin's algorithm is very involved and known to be not primitively recursive~\cite{koscielskinotrecursive},
the same applies to Razborov's generalisation, 
which was the first step of solving Tarski's conjectures (on elementary equivalence and decidability of the theory of free groups)~\cite{tarskielementary,tarski_sela}.
A~different approach based on Plandowski's algorithm for the free monoid case~\cite{plandowskifocs}
was later proposed~\cite{diekertfreegroups},
and an even simpler approach, which gives also a finite description of~the~solution set, was~given by Diekert, Plandowski and Je\.z~\cite{wordeqgroups},
it extends Je\.z's algorithm for the free monoid case~\cite{wordequations}.

The problem of word equations in the free group was first investigated by Lyndon~\cite{lyndon_one_var_FG},
who considered the restricted variant of one-variable equations.
He showed that the solution set is a finite union of sets of the form
\begin{equation}
\label{eq:parametrizable_solution_set}
\{w_0w_1^{i_1}w_2w_3^{i_2}\cdots w_{2k-1}^{i_{k}}w_{2k} \: : \: i_1, \ldots, i_{k} \in \mathbb Z\} \enspace,
\end{equation}
where $w_0, \ldots, w_{2k}$ are words over the generators of the free group,
we call such sets $k$-parametric.
In fact, it was first shown using combinatorial arguments that a superset of all solutions is~of~this form,
and using algebraic methods the superset of all solutions is transformed into the actual set of all solutions.
As a result, $k$ depends on the equation and is a by-product of the algorithm rather than an explicitly given number.
By using a more refined, though purely combinatorial, argument Appel~\cite{appel_one_var_FG}
showed that there exists a superset of solutions that is a finite union of $1$-parametric sets
and that one can test for which values such words are indeed solutions.
In principle, the proof can be readily used as an algorithm, but no reasonable bounds can be derived from it.
Unfortunately, Appel's proof contains an error (see~\cite{chiswell_remeslennikov_one_var_FG} for a discussion).
A similar characterization was announced by Lorentz~\cite{lorents_one_var_FG_2},
but the proof was not supplied.
Chiswell and Remeslennikov~\cite{chiswell_remeslennikov_one_var_FG} used a different approach,
based on geometric group theory, to show that the solution is a finite union of $1$-parametric sets.
However, their argument does not give any algorithm for solving an equation.
Gilman and Myasnikov~\cite{gilman_myasnikov_one_var_FG}
gave a proof that the solution set is $4$-parametric;
their proof is based on formal language theory
and is considerably simpler and shorter than the other known ones,
however, it yields no algorithm.

A polynomial-time algorithm solving the one-variable word equations
(in the free group) was given by Bormotov, Gilman and Myasnikov~\cite{onevariablefreegroup}.
In principle, their argument is similar to Appel, though simpler (and without errors),
and extra care is taken to guarantee that testing takes polynomial time.
The running time is high, though little effort was made to~lower the exponent,
we believe that simple improvements and better analysis should yield $\Ocomp(n^5)$ running time of their algorithm.

It is known that already two-variable word equations (in the free group) do not always have
a parametrizable solution set~\cite{appel_2_var_FG},
here a parametrizable set is a~generalization of~parametric sets~\eqref{eq:parametrizable_solution_set}
in which the exponents using integer parameters can be nested
and one exponent may depend on different parameters.
Moreover, no polynomial-time algorithm for two-variable equations is known.
Other restricted cases were also investigated,
say the famous Lyndon-Sch\"utzenberger Theorem was originally shown for the free group~\cite{lyndonschutzenberger1962}
and satisfiability of~quadratic word equations is
known to be \NPclass-complete~\cite{quadraticfreegroups} in the case of free group.

\subparagraph{Our results and proof outline}
We present an $\Ocomp(n^2 m)$ algorithm for solving equations with one variable in a free group,
where $n$ is the length of the equation and $m$ the number of~occurrences of~the variable in it.

\begin{theorem}
\label{thm:main}
Given a word equation with one variable in a free group,
with length $n$ and $m$ occurrences of the variable,
we can compute the set of all its solutions in time $\Ocomp(n^2m)$.
The set of solutions is a union of $\Ocomp(n^2)$ sets of the form
$\{\alpha w^k \beta \: : \: k \in \mathbb Z \}$,
where $\alpha, w, \beta$ are words over the generators of the given free group.
\end{theorem}
The running time is achieved in the RAM model,
more specifically we require that operations on $\log n$-bits long integers (and byte-arrays)
can be performed in $\Ocomp(1)$ time.
If this is not the case, then the running time
increases by a multiplicative $\Ocomp(\log n)$ factor.
Note that in Theorem~\ref{thm:main} we allow $w = \varepsilon$,
i.e.\ the set $\{\alpha w^k \beta \: : \: k \in \mathbb Z \}$
from Theorem~\ref{thm:main} may consist of~a~single string.

The $\Ocomp(n^2m)$ running time seems hard to improve:
all known characterization of solution set include $\Omega(n^2)$ 
individual words that should be tested as solutions
and natural testing of a single solution is done in $\Theta(m)$ time,
note that this does not take into account the $1$-parametric sets
that do depend on the parameter,
which seem to be harder to be tested.

We use a previous characterization of the solution superset~\cite{onevariablefreegroup},
from which it follows that the main task is to compute, given words
$\alpha, u, v, \beta$, for which $i, j \in \mathbb Z$ the word $\alpha u^i v^j \beta$ is a~solution.
Roughly speaking, the previous approaches~\cite{appel_one_var_FG,onevariablefreegroup}
argued that if $\alpha u^i v^j \beta$ is a solution for a ``large enough'' $i$ then $\alpha u^{i'} v^j \beta$  is a~solution for each $i' \in \mathbb Z$;
thus one has to check some ``small'' $i$s and one ``large enough'';
for each fixed $i$ we substitute its value and similarly argue that if $j$ is ``large enough''
then each $j'$ yields a solution
(the actual argument is more subtle and symmetric in terms of $i$ and $j$).
We refine this approach: previously the tested values of $i$ and $j$ did not depend on the actual equation,
but only on its length.
We identify a~small set of candidate pairs $(i,j)$ based on the actual equation.
To this end, we substitute $\alpha u^I v^J \beta$ to the equation,
where $I, J$ are integer variables,
and intend to verify, for which values $(i, j)$ of variables $(I, J)$ it is a solution.
Such parametric candidates cannot be tested as solutions
(in particular because it could be that only for some values of $I$ and $J$ they indeed are solutions),
however, some operations can be performed on $u^I$ (or $v^J$),
regardless of the actual value substituted for $I$:
say $u^Iu^Iu^{-1}u^{-I}$ is equal to $u^{I-1}$ (in a free group).
After performing all such possible operations
we obtain a word with ``parametric powers'' of $u, v$,
i.e.\ powers, whose exponents depend on parameters $I, J$,
note that the parameters are the same for all powers in the parametric word,
but the actual exponents in different powers may be different.
If there are only powers of $u$ (or only powers of $v$)
then using known tools one can show that one of those exponents is (almost) $0$.
This yields a linear set of possible $i$s that should be tested.
Ideally, we would like to say that a similar claim holds also when parametric powers of both $u$ and $v$ are present.
However, those powers can interact and such an approach does not work directly.
Instead, if $I =i, J = j$ yields a solution,
then substituting $I = i$ (as a mental experiment)
either reduces the whole word to $\varepsilon$, in which case each $J = j$ yields a solution,
or leaves only powers of $v$, in which case we can reiterate the same approach,
this time for powers of $v$.
The former case gives a set of candidates for $I$, the latter for $J$,
technically those depend on the substituted $i$,
but this dependency can be removed by further analysis.
A similar analysis can be made for substitution $J = j$,
together yielding a superset of all possible solutions,
which are then individually tested.

Additional analysis is needed to bound the number of candidates that is obtained in this way.
To this end, we analyze the set of possible exponents of powers of $u$ and $v$.
In particular, we show that initially all such exponents are of the form $\pm I +c $ and $\pm J +c$,
which allows for much better estimations:
for the candidate solution to be different, the constants in those expressions need to be different and
to have a factor $u^{I+c}$ some $c |u|$ letters from the equation are ``consumed''
and easy calculations show that
there are only $\Ocomp(\sqrt n)$ different possible constants,
which leads to $\Ocomp(\sqrt n)$ different candidates.
One has to take special care of $\alpha, \beta$, as their introduction
can yield a quadratic-size equation.
To avoid this, we analyze how powers of $u$ in concatenations of words can be obtained.

In most cases, we reduce the problem in the free group to the problem in the free monoid (with involution)
and use standard tools of word combinatorics.
However, this requires some additional properties of words $\alpha, u, v, \beta$.
Those cannot be inferred from known characterizations,
and so known proofs are reproved and the additional claims are shown.

\subparagraph{Connection to word equation in the free monoid}
The connection between word equations in the free group and free monoid is not perfectly clear.
On one hand, the satisfiability of the former can be reduced to the satisfiability of word equations over the free monoid (with involution),
this was implicitly done by Makanin~\cite{mak82} and explicitly by Diekert et al.~\cite{diekertfreegroups}
and so generalizations of algorithms for the monoid case are used for the group case.
However, there is an intuition that the additional group structure should make the equations somehow easier.
This manifests for instance for quadratic equations (so the case when each variable is used at most twice),
for which an \NPclass{} algorithm was given for the free-group case~\cite{quadraticfreegroups}
and no such result is known for the free monoid case.
Furthermore, the whole first-order theory of equations over the free group is decidable~\cite{tarskielementary},
while already one alternation of the quantifiers make a similar theory for monoid undecidable
(see~\cite{word_equation_theories_undecidable} for an in-depth discussion of undecidable and decidable fragments).

On the other hand, such general reductions increase the number of variables
and so are not suitable in the bounded number of variables case.
In particular, a polynomial time algorithm for the satisfiability of two-variable equations for the free monoid is known~\cite{twovarnew},
in~contrast to the case of the free group
(the set of solutions is still not parametrisable~\cite{Hmelevskii_free_semigroup},
as~in~the~case of~the~free group.).

\subparagraph{Word equations in free monoid with restricted number of variables}
Word equations in~the~free monoid with restricted number variables were also considered.
For one variable a~cubic-time algorithm is trivial and can be easily improved to quadratic-running time~\cite{lothairediekert}.
Eyono Obono, Goralcik and Maksimenko gave a first non-trivial algorithm running in time $\Ocomp(n \log n)$~\cite{onevarfirst}.
This was improved by D\k{a}browski and Plandowski~\cite{onevarold} to $\Ocomp(n + m \log n)$,
where $m$ is the number of occurrences of the variable in the equation,
and to $\Ocomp(n)$ by Je\.z~\cite{onevarlinear};
the last two algorithms work in the RAM model,
i.e.\ they assume that operations on the $\log n$-bits long numbers can be performed in constant time.
The properties of the solution set were also investigated:
all above algorithms essentially use the fact that the solution set
consists of at most one $1$-parametric set
and $\Ocomp(\log n)$ other solutions~\cite{onevarfirst}.
Plandowski and Laine showed that the solution set is either
exactly a $1$-parametric set or of size $\Ocomp(\log m)$~\cite{onevarnew}
and conjectured that in the latter case there are at most $3$ solutions.
This conjecture was recently proved by Saarela and Nowotka~\cite{Saarela_Nowotka_3_solutions} using novel techniques.

Word equations in the free monoid with two variables were also investigated.
it was shown by Hmelevski\u{\i}~\cite{Hmelevskii_free_semigroup} that there are equations whose solution set is not
parametrizable.
The first polynomial-time algorithm (of a rather high degree)
for satisfiability of such equations was given by Charatonik and Pacholski~\cite{CharatonikPacholski},
this was improved to $\Ocomp(n^6)$ by Ille and Plandowski~\cite{Ille_Plandowski_two_var} 
and later to $\Ocomp(n^5)$ by D\k{a}browski and Plandowski~\cite{twovarnew},
the latter algorithm also returns a~description of all solutions.
The computational complexity of word equations with three variables is unknown,
similarly, the computational complexity of satisfiability in the general case of word equations in the free monoid remains unknown (it is \NPclass-hard and in \PSPACE).

\section{Definitions and preliminaries}\label{sec:definitions}

\subsection{Notions}\label{sec:notions}
\subparagraph{Monoids, monoids with involution}
By $\Sigma$ we denote an alphabet, which is endowed with involution $\ov{\cdot}: \Sigma \to \Sigma$, i.e.\ a function such that $\ov a \neq a  = \ov {\ov a}$.
The free monoid $\Sigma^*$ with involution consists of all finite words over $\Sigma$
and the involution uniquely extended from $\Sigma$ to $\Sigma^*$
by requiring that $\ov{(uv)} = \ov v \, \ov u$,
i.e.\ we think of it as of inverse in a group.
We denote the empty word by $\varepsilon$.
Given a word $uvw$: $u$ is its prefix, $w$ suffix and $v$ its subword; 
we also write $u \pref uvw \suff w$ to denote the prefixes and suffixes;
for a word $w$ often $w'$ and $w''$ will denote the prefix and suffix  of $w$,
this will be always written explicitly.
A~word $w = a_1\cdots a_k$, where $a_1, \ldots, a_k \in \Sigma$, has length $|w| = k$
and $w[i\twodots j]$ denotes a subword $a_i\cdots a_j$.
For $k \geq 0$ a word $u^k$ is a $k$-th \emph{power} of $u$ (or simply $u$-power),
by convention $u^{-k}$ denotes $\ov u ^k$.
A $u$-power prefix (suffix) of $v$ is the longest $u$-power
that is prefix (suffix, respectively) of $v$, note that this may b a positive or negative power, or $\varepsilon$.
A~single-step reduction replaces $wa \ov a v$ with $wv$,
a reduction is a sequence of single-step reductions.
A~word in a free monoid $\Sigma^*$ with involution
is \emph{reduced} if no reduction can be performed on it.
It~is folklore knowledge
(and a bit tedious to show)
that for $w$ there exists exactly one reduced $v$ such that $w$ reduces to~$v$;
we call such a $v$ the \emph{normal form} of $w$ and denote it by $\nf(w)$;
we write $w \eqg v$ when $\nf(w) = \nf(v)$.
A $t \in s^* \cup \ov{s}^*$ is an $s$-power (or power of $s$).
We write $u \shift v$ to denote that $u=u'v'$ and $v=v'u'$
or $\ov v = v'u'$ for some $u' ,v'$.
A reduced word $w$ is \emph{cyclically reduced} if it is not of the form $w=a v \ov a$ for any $a \in \Sigma$ and
$w$ is \emph{primitive} if there is no word $v$, such that $w=v^k$ for some natural number $k > 1$. 


\subparagraph{Free group}
Formally, the free group (over generators $\Sigma$) consists of all reduced words over $\Sigma$
with the operation $w \cdot v = \nf(wv)$.
We use all elements of $\Sigma^*$ to denote elements of the free group,
with $w$ simply denoting $\nf(w)$.
Note that in such a setting $\eqg$ corresponds to~equality in free group.
Note that the inverse $w \inv$ of $w$ is $\ov w$ and we will use this notation,
as~most of~the~arguments are given for the monoid and not the free group.

Any equation in the free group is equivalent to an equation in which the right-hand side is $\varepsilon$,
as $u \eqg v$ is equivalent to $u v^{-1} \eqg  \varepsilon$,
thus in the following we consider only equations in~such a form.
Moreover, $uv \eqg  \varepsilon$ is equivalent to $vu \eqg \varepsilon$,
which can be seen by multiplying by $v$ from the left and $v^{-1}$ from the right;
hence we can assume that the equation begins with a~variable.
Let us fix the equation 
\begin{equation}
\label{eq:main}
X^{p_1} u_1X^{p_2} u_2 \cdots u_{m-1}X^{p_m} u_m \eqg \varepsilon
\end{equation}
for the rest of the paper,
each $u_i$ is a reduced word in $\Sigma^*$,
every $p_i$ is $1$ or $-1$ and
there are no expressions $X \varepsilon \ov X$ nor $\ov X \varepsilon X$ in the equation.
Clearly, $m$ is the number of occurrences of the variable $X$ in the equation,
let $n = m + \sum_{i=1}^{m} |u_i|$ be the length of the equation.
A~reduced word $x \in \Sigma^*$ is a solution when
$x^{p_1} u_1x^{p_2} \cdots u_{m-1}x^{p_m} u_m \eqg \varepsilon$.

\subparagraph{Integer expressions, parametric words}
Let us fix two integer variables $I, J$ for the remainder of the paper.
An integer expression is of the form $n_I I+n_J J+n_c$,
where $n_I, n_J, n_c \in \mathbb Z$ are integers;
an expression is \emph{constant} when $n_I = n_J = 0$ and \emph{non-constant} otherwise.
We denote integer expressions with letters $\phi, \psi$,
note that all expressions that we consider are in the same two variables $I, J$.
A value $\phi(i,j)$ is defined in a natural way;
we also use this notation for substitutions of variables,
say $\phi(I, k - I)$, which is defined in a natural way.
The~integer expression $\phi$ depends on the variable $I$ ($J$)
if $n_I \neq 0$ ($n_J \neq 0$) and it depends on $I+J$ if $n_I = n_J \neq 0$.
If $\phi$ depends on exactly one variable then we write $\phi(i)$ to denote its value.

An $s$-parametric power is of the form $s^\phi$, where $\phi$ is an integer expression and $s$ a word;
then $s(i,j)$ denotes $s^{\phi(i,j)}$,
this can be interpreted both as an element in the monoid and in the free group.
Unless explicitly stated, we consider only non-constant expressions $\phi$ as exponents in parametric powers,
this should remove the ambiguity that an $s$-power is also an $s$-parametric power.
A parametric word is of the form $w = t_0s_1^{\phi_1}t_1\cdots t_{k-1}s_k^{\phi_k}t_{k}$
(all~arithmetic expressions $\phi_1, \ldots, \phi_k$ are in the same two variables $I, J$)
and $w(i, j)$ denotes $t_0s_1^{\phi_1(i, j)}t_1\cdots t_{k-1}s_k^{\phi_k(i, j)}t_{k}$.
In most cases, we consider very simple parametric words, where $k\leq2$ and both expressions depend on one variable only.
We sometimes talk about equality of parametric words (in a free group),
formally $w \eqg w'$ if for each $(i,j) \in \mathbb Z^2$ it holds that $w(i,j) \eqg w'(i,j)$.
We will use those only in very simple cases, say $u^{I+1}u^{-I+1} \eqg u^2$.

As we process sets of integer expressions (as well as parametric powers),
we will often represent them as sorted lists (with duplicates removed):
we can use any linear order, say for integer expressions the lexicographic order on triples $(n_I,n_J,n_c)$
and for parametric powers the lexicographic order on tuples $(s,n_I,n_J,n_c)$,
where tuple $(s,n_I,n_J,n_c)$ corresponds to a~parametric power
$u^{n_II+n_JJ+n_c}$.


\subsection{Pseudosolutions}
We want to specify some properties of reductions of solutions,
instead of usual reduction sequences it is a bit more convenient
to talk about pairings that they induce.
Given a word $w[1\twodots n] \in \Sigma^n$, $w \eqg \varepsilon$
its partial reduction pairing (or simply partial pairing).
is intuitively speaking,
a pairing of indices of $w$ corresponding to some reduction.
Formally, it is a partial function $f:[1\twodots n] \to [1\twodots n]$
such that if $f(i) \neq \bot$ then $f(f(i)) = i$ (it is a pairing), $w[i] = \ov{w[f[i]]}$
(it pairs inverse letters) and either $f(i) \in \{i-1, i+1\}$ or $f(i) = j \neq i$
and $f$ is defined on the whole interval $[\min(i,j)+1,\max(i,j)-1]$
and $f([\min(i,j)+1,\max(i,j)-1]) = [\min(i,j)+1,\max(i,j)-1]$
(so the pairing is well nested and corresponds to a sequence of reductions).
A partial pairing is a \emph{pairing} if it is a total function.
When needed, we will draw partial reduction pairings as on Fig.~\ref{fig:x1xx12}--\ref{fig:xxx6},
i.e.\ by connecting appropriate intervals of positions.
Note that the reduction pairing is not unique, say $a \ov a a \ov a$ has two different reduction pairings.

It is easy to see that a reduction pairing induces to reduction sequence (perhaps more than one) and vice-versa,
and so a word $w$ has a reduction pairing if and only if $w \eqg \varepsilon$.

\begin{lemma}
A word $w$ has a reduction pairing if and only if $w \eqg \varepsilon$.
\end{lemma}
\begin{proof}
We proceed by a simple induction: if $w$ is reducible then either $w = a \ov a$ for some $a \in \Sigma$ and then it clearly has a reduction paring
or $w = w_1 a \ov a w_2$ and $w_1w_2 \eqg \varepsilon$.
Create the pairing for $w$ by pairing those $a$ and $\ov a$ and otherwise
using the pairing for $w_1w_2$, which is known to exist by the induction assumption
(formally some renumbering of the indices is needed).

In the other direction, if $w$ has a reduction pairing,
consider $f(1)=i$.
Then $w = a w_1 \ov a w_2$, such that $w_1$ and $w_2$ are paired inside.
Thus by induction assumption both $w_1\eqg \varepsilon$ and $w_2 \eqg \varepsilon$.
Thus $w \eqg a \ov a \eqg \varepsilon$.
\end{proof}

Given a (not necessarily reducible) word $w = w_1w_2w_3 \in \Sigma^*$
we say that $w_2$ is a \emph{pseudo-solution} for a partial reduction pairing $f$
if $f$ is defined on whole $w_2$.
This is sometimes written as $w_1 \underline{w_2}w_3$ to make graphically clear,
which subword is a pseudosolution. 
Note that we do allow that $w_2 \eqg \varepsilon$, in which case it is a pseudo-solution,
and we do allow that $f$ pairs letters inside $w_2$.

The first fact to show is that for any pairing $f$ if we factorize a reducible word
then there is a pseudo-solution for some of its consecutive subwords.
A variant of this Lemma was used Lyndon~\cite[Proposition~1]{lyndon_one_var_FG},
Appel~\cite[Proposition~1]{appel_one_var_FG} and by~Bormotov, Gilman and Myasnikov~\cite[Lemma~3]{onevariablefreegroup}
and it is attributed already to Nielsen~\cite{Nielsen}.

\begin{lemma}[{cf.~\cite[Lemma~3]{onevariablefreegroup}}, full version of Lemma~\ref{lem:pseudosolution_simple}]
\label{lem:pseudosolution}
Let $\varepsilon \eqg s_0u_1s_1u_2\cdots s_{k-1}u_k s_k$ and $f$ be its pairing.
Then there is $u_i$ that is a pseudo-solution of $u_{i-1}s_{i-1}\underline{u_{i}}s_{i}u_{i+1}$ (for $f$).
\end{lemma}
\begin{proof}
If there is $u_i = \varepsilon$ then we are done. So consider the case that each $u_i \neq \varepsilon$.
We maintain an interval of position $I$ such that $f(I) \subseteq I$
and $I$ contains at least one word $u_i$.
Initially $I = [1\twodots |w|]$.

Take any $u_i$ within $I$ and let $i_{\min}, i_{\max}$ be positions within $u_i$
such that $f(i_{\min}) = \min f(u_i)$ and $f(i_{\max}) = \max f(u_i)$,
i.e.\ $[f(i_{\min}),f(i_{\max})]$ is the smallest interval of positions such that $u_i$ is a pseudosolution within it.
If $f(i_{\min}), f(i_{\max}) \in u_{i-1}s_{i-1}u_is_iu_{i+1}$ then we are done.
If not, then by symmetry consider $f(i_{\min}) \notin u_{i-1}s_{i-1}u_is_iu_{i+1}$.
If $f(i_{\min})$ it is to the right of $u_{i+1}$
then we take as the interval $[i_{\min}+1,f(i_{\min})-1]$:
clearly it contains whole $u_{i+1}$ and $f([i_{\min}+1,f(i_{\min})-1]) = [i_{\min}+1,f(i_{\min})-1]$ and it is smaller than $I$.
If $f(i_{\min})$ it is to the left of $u_{i-1}$ then we take
$[i_{\min}-1,f(i_{\min})+1]$ and analyze it symmetrically.
Thus at some point we will find the pseudo-solution.
\end{proof}


\section{Word combinatorics}

In this section we present various combinatorial properties of words,
treated as elements of free monoid or as elements of the free group.

Section~\ref{sec:period-runs-primitivity} deals
with standard notions of periodicity, primitivity and runs,
Section~\ref{sec:assorted-combinatorial-lemmata}
gives various combinatorial properties that are needed
for the proofs,
but their proofs are not needed in order to understand the general argument.
Section~\ref{sec:maximal-powers} is concerned with maximal powers within a string and the way they factorize into concatenations.
This is one of the main tools used in the restriction of the set of candidate solutions in Section~\ref{sec:restricting-the-superset-of-solutions}.


\subsection{Period, runs, primitivity}\label{sec:period-runs-primitivity}

The following fact follows straight from definitions of cyclic reductions and primitivity.
\begin{lemma}
	\label{lem:cyclic_shifts_reduced_primitive}
If $w$ is cyclically reduced and $v \shift w$ then $v$ is cyclically reduced.
If $w$ is primitive and $v \shift w$ then $v$ is primitive.
\end{lemma}

If $s[1 \ldots |s|-p] = s[p+1 \ldots |s|]$ then $p$ is a \emph{period} of $s$.
It is a classic fact that
\begin{lemma}[Periodicity Lemma, Fine-Willf Lemma]
Is $p, p'$ are periods of $s$ and $|s| \geq p+p' - \gcd(p,p')$
then $\gcd(p,p')$ is also a period of $s$.
\end{lemma}

For a cyclically reduced word $t$ we say that $w$ is a \emph{run} of $t$ if it is a subword of $t$ or 
$w = t'' t^k t'$ or $\ov w = t'' t^k t'$,
where $t' \pref t \suff t''$ and $k \geq 0$.
Note that it is often assumed that a $t$-run has length at least $|t|$;
also, the involution is often not considered.
If $r$ is a run of $t$ and $|r| \geq |t|$ then $|t|$ is a period of $r$.

\begin{lemma}
\label{lem:different_runs_overlap}
\label{lem:runs_of_different_words}
\label{lem:u-power_pref_and_suff_overlap}
Let $s, s'$ be both primitive and cyclically reduced.
Let $r \neq r'$ be an $s$-run and an $s'$-run, and subwords of $w$.
If $s \not \shift s'$ then the overlap of $r$ and $r'$ is of length smaller than $|s| + |s'|$.
If $s \shift s'$ then the length of overlap of $r$ and $r'$ is smaller than $|s|$
or there is an $s$-run containing both $r, r'$.
\end{lemma}
\begin{proof}
If $|s| \neq |s'|$ and the overlap of $r, r'$ has length at least $|s|+|s'|$
then by Periodicity Lemma this overlap has both period $|s|$ and $|s'|$,
so it has a period $\gcd(|s|,|s'|)$,
which contradicts the primitivity of $s, s'$.
If $|s| = |s'|$ then the assumption that the overlap
is of length at least $|s|$ implies (by primitivity) that $s \shift s'$.
But then $r, r'$ overlap at at least $|s|$ letters
and both have period $|s|$, which implies that they are subwords of the same run.
\end{proof}

\subsection{Maximal powers}\label{sec:maximal-powers}
We say that a word $s^p$ is a \emph{maximal power} in a word $t$,
if it is a subword of $t$ and there is no $s$ nor $\ov s$ to its left and right in $t$;
note that $t$ need not to be reduced.
For instance $a^3$, $a^2$ and $(ab)^2$ are maximal powers in $aaababaa$. To streamline the analysis, we assume that $s^0$
(called the \emph{trivial power}) is a maximal power in any word $t$, even the empty one.

If $s^p$ is a maximal power in a normal form of concatenation of several words
$\nf(w_1\cdots w_\ell)$, then clearly $s^p$ can be partitioned into $\ell$
subwords such that the $i$-th of them comes from $w_i$.
However, we show more: we can identify such a maximal power in each $w_i$,
that $s^p$ is (almost) the normal form of concatenation of those maximal powers.
This is beneficial: the~number of different maximal powers in a word is much
smaller than the number of different powers that are subwords. 

\begin{lemma}
\label{lem:sum_of_powers}
Let $w_1, w_2, \ldots, w_\ell$ be reduced and $s$ be cyclically reduced.
If $s^k$ is a maximal power in $\nf(w_1\cdots w_\ell)$
then for each $1 \leq h \leq \ell$ there exists such a maximal power $s^{k_h}$
in $w_h$ that $|\sum_{h=1}^\ell k_h - k| < \ell$.
Moreover, if $s^k$ is the $s$-power prefix (suffix) of $\nf(w_1\cdots w_\ell)$
then we can choose $s^{k_1}$ as the $s$-power prefix of $w_1$ or a trivial power
($s^{k_\ell}$ as the $s$-power suffix of $w_\ell$ or a trivial power, respectively);
if $s^k = \nf(w_1\cdots w_\ell)$ then both conditions hold simultaneously.
\end{lemma}
The proof of Lemma~\ref{lem:sum_of_powers} in case of $\ell \leq 2$ is a simple case distinction.
For larger $\ell$, we let $w_{1,2} = \nf(w_1w_2)$
and apply the induction assumption
to $w_{1,2}w_3, \ldots, w_\ell$,
the proof again follows by simple combinatorics on words.

\begin{proof}
To prove the statement it is enough to indicate appropriate powers $s^{k_1}, s^{k_2}, \ldots, s^{k_\ell}$ for any given maximal power $s^k$ in $\nf(w_1\cdots w_\ell)$.
We construct them by induction.
The base case $\ell = 1$ is trivial, since
$w_1 = \nf(w_1)$ and we simply take the same power $s^{k_1}=s^k$, and then obviously $0=|k_1-k|<l=1$.

Assume that the inductive hypothesis holds for $\ell -1$.
Consider a maximal power $s^k$ in $\nf(w_1\cdots w_\ell)$;
the same power is maximal in $\nf(w_{1,2}w_3\cdots w_\ell)$, where $w_{1,2}=\nf(w_1w_2)$.
Moreover, by the assumption we can construct the appropriate powers $s^{k_{1,2}}, s^{k_3}, \ldots, s^{k_{\ell}}$,
such that $s^{k_h}$ is a maximal power in $w_h$ and that they satisfy the claim of the lemma.
It is enough to construct maximal powers  $s^{k_1},s^{k_2}$ of $w_1,w_2$,
such that $|(k_1+k_2) - k| \leq 1$,
as then 
\begin{align*}
	\left|\sum_{j=1}^\ell k_j - k\right|
		&=
	\left |(k_1 + k_2 - k_{1,2}) + k_{1,2} +\sum_{j=3}^\ell k_j - k\right|\\	
		&\leq
	\left |(k_1 + k_2 - k_{1,2}) \right| + \left|k_{1,2} +\sum_{j=3}^\ell k_j - k\right|\\	
		&\leq
	1 + (\ell - 1)\\
	&= \ell	\enspace .
\end{align*}
Moreover, we should guarantee that if $s^k$ is an $s$-power prefix then $s^{k_1}$ is also
an $s$-power prefix of $w_1$ or $\varepsilon$ (the trivial power)
and if $h = 2$ then
if $s^k$ is an $s$-power suffix of $w$ then $s^{k_h}$ is an $s$-power suffix of $w_h$ or $\varepsilon$.

\begin{figure}
\centering
\includegraphics[scale=1.2]{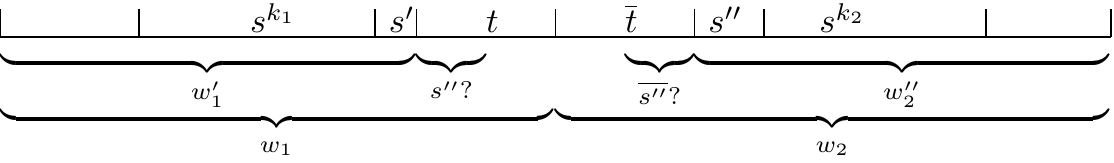}
\caption{The split of power, first case: $s^{k_{1,2}}$ is split between $w_1$ and $w_2$.}
\label{fig:max_power_split_1}
\end{figure}

Let $t$ be the maximal reduction of $w_1$ and $w_2$,
i.e.\ $w_1 = w_1' t$, $w_2 = \ov t w_{2}''$ and $\nf(w_1w_2)=w_1'w_{2}''$;
it could be that $t = \varepsilon$ or that one of $w_1', w_{2}''$ is $\varepsilon$,
see Fig.~\ref{fig:max_power_split_1}.
We consider only the case $k_{1,2} \neq 0$,
the other case is trivial as we choose $k_1=k_2=0$. 
If $s^{k_{1,2}}$ is also a maximal power in $w_1$ or $w_2$,
so as in $w_1'w_2''$,
then we take this power (and the trivial power in the other word) and we are done;
note that if $s^{k_{1,2}}$ was an $s$-power prefix then by our choice we also choose
$s$-power prefix of $w_1$.
There are two remaining cases:
the maximal power $k_{1,2}$ in $w_1'w_2''$ is a subword of both words 
or it is wholly inside $w_1'$ (or $w_2''$), but not maximal in $w_1$ ($w_2$ respectively).

\begin{figure}
\centering
\includegraphics[scale=1.2]{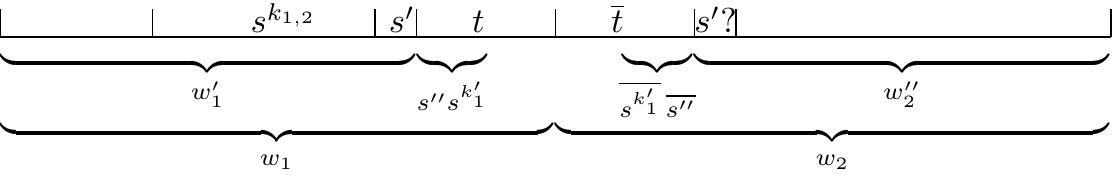}
\caption{The split of power, second case: $s^{k_{1,2}}$ is within $w_1$.}
\label{fig:max_power_split_2}
\end{figure}

Consider the fist case, i.e.\ the maximal power $s^{k_{1,2}}$ is split between $w_1'$ and $w_2''$. see Fig.~\ref{fig:max_power_split_2}.
We consider the subcase when $k_{1,2}>0$, the other one ($k_{1,2}<0$) is similar;
alternatively, we can consider $\ov{w_{1,2}}$ and the corresponding maximal power
$s^{-k_{1,2}}$, proceed with the proof and then at the end take the inverse again.
Let $w_1'$ have a suffix $s^{k_1}s'$ and $w_2''$ have a prefix $s''s^{k_2}$,
where the maximal power in $w_{1,2}$ is of the form $s^{k_{1,2}} = s^{k_1}s's''s^{k_2}$
and either $s'= s'' = \varepsilon$ or $s = s's''$ and $s' \neq \varepsilon \neq s''$.
We claim that $s^{k_1}$ and $s^{k_2}$ are maximal powers in $w_1$ and $w_2$.
By symmetry, suppose that $s^{k_1}$ is not a maximal power in $w_1$.
There cannot be $s$ or $\ov s$ to the left of it,
as this would contradict the fact that $s^{k_1}s's''s^{k_2}$ is a maximal $s$-power in $w_{1,2}$, then consider extending power to the right.
When $s' \neq \varepsilon \neq s''$ then $s''$ is a prefix of $t$
and so $\ov{s''}$ is a suffix of $\ov{t}$,
this contradicts the fact that $w_2 = \ov{t}w_2''$ is reduced.
When $s' = s'' = \varepsilon$
a similar argument shows that $s$ is a prefix of $t$
and so $\ov{s}$ is a suffix of $\ov t$.
By maximality of $t$ we know that $k_1$ and $k_2$ have the same sign.
As $\ov t w_2''$ is reduced, we conclude that $k_2 = 0$,
which means that the power is in fact not split between $w_1', w_2''$, so we get a contradiction.
Hence we take $s^{k_1}$ and $s^{k_2}$ as the maximal powers in $w_1$ and $w_2$, respectively.
Then $k_{1,2} -1 \leq k_1 + k_2 \leq k_{1,2} +1$;
note that here in fact $k_1 + k_2 \leq k_{1,2}$,
but for $k_{1,2} < 0$ we can only prove the claimed $k_1 + k_2 \leq k_{1,2}+1$.
For the second claim, we should show that if $s^k$ is a prefix of $w$ then
$s^{k_1}$ is a prefix of $w_1$ or $k_1 = 0$
and if $s^k$ is a suffix of $w$ then $s^{k_\ell}$ is a suffix of $w_\ell$ or $k_\ell = 0$.
By induction, the claim about the prefix holds for $w_{1,2}$
and thus $s^{k_{1,2}}$ is a prefix or $k_{1,2} = 0$,
the latter case was already considered.
In the former when $s^{k_{1,2}}$ is a prefix of $w_{1,2}$ then the construction guarantees
that $s^{k_1}$ is a prefix of $w_1$.
For the claim for $k_\ell$: it holds by induction assumption when $\ell > 2$
and if $\ell = 2$ (so $w_{1,2} = w$) it is shown symmetrically to the case of $k_1$.

In the second case the non-trivial power $s^{k_{1,2}}$ is fully within one word, say $w_1'$,
but in $w_1 = w_1' t$ it is not a maximal power.
Again, we consider the case when $k_{1,2} > 0$, the other one ($k_{1,2}<0$)
is shown in the same way.
Then $w_1'$ has a suffix $s^{k_{1,2}}s'$ and $t$ a prefix $s''s^{k_1'}$,
where either $s's'' = s$ and $s' \neq \varepsilon \neq s''$ or $s' = s'' = \varepsilon$.
This $s^{k_1'}$ is a maximal power in $t$
and trivially $\ov{s^{k_1'}}$ is a maximal power in $\ov t$;
we claim that the mentioned $\ov{s^{k_1'}}$ is also a maximal power in $\ov{t}w_2'' = w_2$.
Clearly the left side of that power cannot be extended.
If it can be extended to the right, then $\ov{s'}$ (when $s' \neq \varepsilon$, and $\ov s$ when $s'  = \varepsilon$)
would be a prefix of $w_2''$,
in both cases contradicting the maximality of $t$.
Hence we choose this maximal power
$s^{k_1} = s^{k_{1,2}}s's''s^{k_1'}$ in $w_1$
(here $k_1 = k_{1,2} + k_1' +1 $ or $k_1 = k_{1,2} + k_1'$, when $s' = s'' = \varepsilon$)
and $\ov{s^{k_1'}}$ in $w_2$, i.e.\ $k_2 = - k_1'$.
Then  $k_{1,2} -1 \leq k_1 + k_2 \leq k_{1,2} +1 $.
It is left to show the second claim:
if $s^k$ is prefix of $w$ then $s_1$ is a prefix of $w$ or $k_1 = 0$
and if $s^k$ is suffix of $w$ then $s^{k_\ell}$ is a suffix of $w$ or $k_\ell = 0$.
Consider first the prefix. By induction, this holds for $s^{k_{1,2}}$
and the case of $k_{1,2} = 0$ was already considered, so it is enough to consider the one when
$s^{k_{1,2}}$ is a prefix of $w_{1,2}$.
If $s^{k_{1,2}}$ is within $w_2$ then this means that $w_1' = \varepsilon$,
as otherwise $s^{k_{1,2}}$ is not a prefix. And in this case we take the prefix of $t$ as $s^{k_1}$.
If $s^{k_{1,2}}$ is within $w_1$ then $s^{k_{1,2}}$ is a prefix of $w_1'$
and we choose $s^{k_1}$ as the prefix of $w_1$.
For the suffix, for $\ell > 2$ this holds by induction assumption,
for $\ell = 2$ the analysis is symmetric to the one for $k_1$.
\end{proof}

There cannot be too many different maximal powers of the same word $s$ in a given word~$w$:
different maximal powers $s^{k_1}, \ldots, s^{k_p}$
use together $|s|k_1 +  \cdots + |s|k_p$ letters in $w$
and when $k_1, \ldots, k_p$ are pairwise different
then this sum is $\Omega(p^2|s|)$ and so $p = \Ocomp(\sqrt{|w|/|s|})$;
this can be naturally generalized to a set of words $W$ instead of a single word $w$.

\begin{lemma}
\label{lem:different_powers_in_a word}
Let $s$ be cyclically reduced word.
Let $W$ be a set of words and $k = \sum_{w \in W} |w|$.
Suppose that $s^{k_1}, \ldots, s^{k_p}$ are pairwise disjoint subwords
of words in $W$
and that $k_1, \ldots, k_p$ are pairwise different integers.
Then $p \leq \sqrt{4k/|s| + 1}$ and if additionally $k \geq |s|$ then $p \leq \sqrt{5k/|s|}$.
\end{lemma}
\begin{proof}
The bound trivially holds when $k < |s|$, as then the only maximal power in those words is the $s^0$,
i.e.\ there is at most one different power.

Observe that it is enough to show a bound
$p \leq \sqrt{4k/|s| + 1} - 1$ under the additional assumption that
each $k_h$ is non-zero.

Clearly $p$ different non-empty powers (for even $p$)
have smallest sum of lengths if the exponents are
$-\frac{p}{2}, \ldots, -1, 1, \ldots , \frac{p}{2}$.
Then sum of their lengths is
\begin{equation*}
   2 \cdot \frac{\frac{p}{2}\cdot (\frac{p}{2}+1)}{2} |s|
	=
\frac{p^2+2p}{4}|s|
\end{equation*}
letters.
Then $\frac{p^2+2p}{4}|s| \leq k$ implies
$p \leq \sqrt{4k/|s| + 1} - 1$.

For odd $p$ the estimation is similar, as then
the ``optimal'' exponents are
$-\frac {p-1} 2, \ldots, -1, 1, \ldots, \frac{p+1} 2$
(or $-\frac {p+1} 2, -\frac {p-1} 2, \ldots, -1, 1, \ldots, \frac{p-1} 2$),
in which case the estimation on the sum of lengths is the same)
and the sum of lengths is then
\begin{equation*}
\frac{\frac{p-1}{2}\cdot \frac{p+1}{2}}{2}|s| + \frac{\frac{p+1}{2}\cdot \frac{p+3}{2}}{2} |s|
	= \frac{(p+1)^2}{4}|s| \enspace ,
\end{equation*}
which yields a smaller upper-bound on $p$.
\end{proof}

\subsection{Assorted combinatorial lemmata}\label{sec:assorted-combinatorial-lemmata}

\begin{lemma}
\label{lem:inverse_no_overlap}
If $u_1wu_2 = v_1 \ov w v_2$ and $w$ is reduced then either $w = 1$ or $v_1 \ov w \pref u_1$ or $u_1 w \pref v_1$, i.e.\ $w$ and $\ov w$ cannot overlap.
\end{lemma}
\begin{proof}
Suppose that $w \neq \varepsilon$ and that $w$ and $\ov w$ overlap
so there is $w'$ that is a prefix of one of $w, \ov w$ and a suffix of the other.
Words $w'$ and $\ov{w'}$ are equal, which implies $w' \eqg \varepsilon$.
It contradicts with fact, that $w$ is reduced.
\end{proof}

\begin{lemma}
\label{lem:u_prefix_beta_w_ovbeta}
Let $t$, $p$ be reduced words and $s$ be cyclically reduced word. Let $t$ does not begin with $s$ nor $\ov s$.
Then:
\begin{itemize}
\item The $s$-power prefix of $\nf(t w \ov t)$
is of length at most $2|s| + |w|$ for any $w \in \Sigma^*$.
\item There are at most 16 possible $s$-power prefixes of $\nf(t w p)$
   over all $w$ such that $|w| < |s|$. 
\item If $\nf(t w \ov t)$ is an $s$-power for $w \not \eqg \varepsilon$,
then $|s| \leq |\nf(t w \ov t)| \leq |w|$.
\end{itemize}\end{lemma}
\begin{proof}
Consider the first claim.
If $|t|<|s|$ then we are done,
as $|tw\ov t| < 2|s|+|w|$.
In the other case, when $|t| \geq |s|$,
as $t$ does not begin with $s$ nor $\ov s$,
we need to reduce more than $|t|-|s|$ letters from $t$ (as less than $|s|$ can remain),
so we are left with less than $|t| + |w| + |\ov t| - 2(|t|-|s|) = 2|s| + |w|$.

For the second claim, let $\nf(t w p) = t' p''$,
where $t' \pref \nf(tw)$ and $p \suff p''$.
Let also $p^\bullet \pref p''$ be the prefix of $p''$ used in the
$s$-power prefix of $t'p''$ (it could be that $p^\bullet = \varepsilon$).
We assume that $s$-power of $t'p'$ is not trivial.
Note that $|t'| < 2|s|$:
$t'$ can use less than $|s|$ letters from $t$ (as it does not begin with $s$ nor $\ov s$) 
and less than $|s|$ from $w$ (as $|w|<|s|$).
If $|p^\bullet| < 4|s|$ then the $s$-power prefix ($= t'p^\bullet$)
has length less than $6|s|$, i.e.\ there are $11$ possibilities for it including trivial power.
So in the following we assume that $|p^\bullet| \geq 4|s|$.

Consider, how many letters in $p$ can be reduced in $\nf(t w p)$.
As $t$ does not begin with $s$ nor $\ov s$,
more than $|t| - |s|$ letters are reduced in $t$
and at most $|w|$ of those letters are not reduced with $p$.
So $t$ reduces at least $|t| - |s| - |w|$
and at most $t$ letters from $p$;
clearly $w$ can reduce at most $|w|$ letters (and at least $0$).
So we reduce between $|t| - |s| - |w|$ and $|t| + |w|$ letters from $p$.
Let $w_1$ and $w_2$ be different words of length less than $|s|$,
define $t_i',p_i'',p_i^\bullet$ for $\nf(tw_ip)$ as above, for $i = 1,2$.
Then $||p_1''| - |p_2''|| \leq (|t| + |w|) - (|t| - |s| - |w|) = 2|w|+|s| < 3|s|$.
As $|p_1^\bullet|, |p_2^\bullet| \geq 4|s|$ and they begin within less than $3|s|$ in positions in $p$,
so they overlap at at least $|s|$ positions.

As they are both $s$-runs of $p$, by Lemma~\ref{lem:different_runs_overlap} they are part of the same maximal $s$-run in $p$.
Both has $s$ (or $\ov s$) as a suffix and both, as maximal $s$-prefixes, cannot be extended to the right by next $s$ (or $\ov s$). 
There is only one suffix of a maximal $s$-run for primitive $s$ with such properties, 
so both ends in the same place in $p$ and one is a suffix of another.

Consider $t_1'p_1^\bullet$ and $t_2'p_2^\bullet$,
as they have common last $|s|$ positions and are reduced, they are either both positive
or both negative powers of $s$.
Moreover, $||p_1^\bullet| - |p_2^\bullet|| < 3|s|$ (as prefixes of $p_1''$,$p_2''$ with the same end) and then:
\begin{equation*}
||t_1'p_1^\bullet| -|t_2'p_2^\bullet||  <
 |t_1' - t_2'| + 3|s| \leq
 5 |s|
\end{equation*}
Hence there are at most $5$ such $s$-power prefixes, since each pair of them has the same sign and a difference in length less than $5|s|$.
Adding the previous $11$, this yields at most $16$ possible $s$-power prefixes.

For the last claim note that $\nf(t w \ov t)$ is cyclically-reduced,
as it is a power of cyclically reduced $s$.
Hence at least one of $t$ or $\ov t$ has to fully reduce within $t w \ov t$,
as otherwise the obtained word is not cyclically-reduced.
So the length of $\nf(t w \ov t)$
is at most
\begin{equation*}
|\nf(t w \ov t)| \leq |t w \ov t| - 2 |t| = |w| \enspace .
\end{equation*}
Note also that it cannot be that $t w \ov t \eqg \varepsilon$,
as then also $w \ov t t\eqg \varepsilon$ which implies that $w \eqg \varepsilon$,
which is excluded by the assumption.
\end{proof}

\begin{lemma}
\label{lem:no_long_self_reduction_after_reduction}
Let $s$ be cyclically reduced and primitive, $t \neq \varepsilon$ reduced and $s,\ov s \not \! \! \pref t$.
Then a reduction for $s^k$ in $s^k t s^{k'}$
is of length less than $2|s|$.
\end{lemma}
\begin{proof}
Suppose that the reduction is of length at least $2|s|$.
The case $k \in \{-1, 0, 1\}$ is trivial,
suppose $k\geq 2$, the other one ($k \leq -2$) is reduced by taking $\ov s$ and $-k, -k'$.

By the assumption, less than $|s|$ letters form $s^{k}$ can be reduced with $t$.
Hence the second from the right $s$ in $s^k$
reduces with a subword $\ov s$ of $s^{k'}$.
If $k' > 0$ then we get a contradiction by Lemma~\ref{lem:inverse_no_overlap},
as $\ov s$ and $s$ cannot overlap.
Thus $k'<0$.
As $\ov s$ is primitive, inside $s^{k'}$ to the left of this
reducing $\ov s$ there is $s^{k''}$ for some $k' < k'' \leq 0$:
otherwise we would have $s = s'' s'$ for some $s' \pref s \suff s''$, which cannot be.
But then this means that the word between this $s$ and the reducing $\ov s$ also reduces,
and it is of the form $s t s^{k''} \eqg \varepsilon$.
Then $t \eqg s^{-k''-1}$.
As $t$ is reduced and not $1$, this means that $t = s^{-k''-1} \neq \varepsilon $ and this contradicts with $s,\ov s \not \! \! \pref  t$.
\end{proof}

\begin{lemma}
\label{lem:parametric_pseudosolution_almost_0}
Let $s$ be a cyclically reduced word and  
let $s^{k}$ be a pseudosolution in $s^{k'}t's^{k}t''s^{k''}$,
where $k, k',k''$ are integers, $s$ is cyclically reduced and primitive,
and $t',t''$ are reduced and do not begin nor end with $s$ nor $\ov s$.
Then $|k| \leq 3$.
\end{lemma}
\begin{proof}
Suppose that $|k| \geq 4$, say $k > 0$, the other case is shown in the same way.
Consider equivalent representation $s^{k'}t's^2 s^{k-4} s^2 t''s^{k''}$.
Then by Lemma~\ref{lem:no_long_self_reduction_after_reduction}
the reductions in $s^{k'}t's^2$ (presented in inverted form for direct lemma application) and $s^2 t''s^{k''}$
are both  of length less than $2|s|$, thus not the whole $s^2 s^{k-4} s^2$ is reduced, contradiction.
\end{proof}

The following Lemma is a simple variant of the classic Nielsen–Schreier theorem 
that each subgroup of a free group is also free.
We use it only for two generators satisfying some additional properties,
so we supply the simplified proof for completeness and for the readers not familiar with it.
\begin{lemma}
\label{lem:concatenation_of_powers_is_not_a_power}
If $s, t$ are cyclically reduced or $st$ is cyclically reduced
and $k_1, k_2, \ldots, k_{2\ell}, k_{2\ell+1}$, for $\ell \geq 1$, are non-zero
integers such that
\begin{equation}
\label{eq:powers_to_powers}
s^{k_1} t^{k_2} \cdots s^{k_{2\ell-1}} t^{k_{2\ell}} \eqg \varepsilon \quad \text{or} \quad 
s^{k_1} t^{k_2} \cdots s^{k_{2\ell-1}} t^{k_{2\ell}} s^{k_{2\ell+1}} \eqg  \varepsilon
\end{equation}
then $s, t$ are powers of the same word.

In particular, the mapping $S \mapsto s$, $T \mapsto t$ defines an isomorphism between
the subgroup generated by $s, t$ and a free group generated by $S, T$.
\end{lemma}
\begin{proof}
First of all observe that the case of odd number of powers can be reduced to the case of even number of powers:
if
$s^{k_1} t^{k_2} \cdots s^{k_{2\ell-1}} t^{k_{2\ell}} s^{k_{2\ell+1}} \eqg  \varepsilon$
then also
$s^{k_1 + k_{2\ell+1}} t^{k_2} \cdots s^{k_{2\ell-1}} t^{k_{2\ell}}  \eqg  \varepsilon$;
if $k_1 + k_{2\ell+1} \neq 0$ then we made the reduction,
otherwise 
we remove $s^{k_1 + k_{2\ell+1}}$ and continue the procedure.
In the end either we obtain an even number of powers, or that a non-zero power
of $s$ or $t$ is $\varepsilon$, which also shows the claim.

Consider first the case when one of $s, t$, say $s$, is not cyclically reduced..
Then $s, \ov s$ begin and end with the same letter,
thus $st$ being cyclically reduced,
which holds by Lemma assumption,
implies that $\ov s t$ is cyclically reduced.
Then by Lemma~\ref{lem:cyclic_shifts_reduced_primitive}
also $ts$ and $t \ov s$ are cyclically reduced.
Taking the inverse yields that also
$\ov t \ov s$, $\ov s \ov t$, $s \ov t$ and $\ov t s$ are also
cyclically reduced.
Hence, there is no reduction between powers of $s$ and $t$,
so the claim of the Lemma holds.

From now on we consider the case when $s, t$ are cyclically reduced and $st$ is cyclically reduced.
Suppose that the claim does not hold and let $s, t$ be words such that
\eqref{eq:powers_to_powers} holds for some non-zero 
$k_1, \ldots, k_{2\ell}$ and $|s|+|t|$ is minimal possible.
Clearly it is not true that $|s| = |t| = 1$:
any reduction in the equation shows
that they are the same letter (or the inverse),
and so the claim holds.

Consider first the case,
when the reduction in $s t$ is of length $\min(|s|,|t|)$,
say by symmetry $|t| \leq |s|$.
Then $s = s' \ov t$; if $s' = \varepsilon$ then we are done, so assume in the following that $s' \neq \varepsilon$.
Substitute $s = s' \ov t$ into the equation~\eqref{eq:powers_to_powers},
and group and reduce the powers of $s'$ and $t$.
More formally,
consider the word
\begin{equation}
\label{eq:new_powers}
(S'\ov T)^{k_1} T^{k_2} \cdots (S'\ov T)^{k_{2\ell-1}} T^{k_{2\ell}}
\end{equation}
and treat it as a word in group freely generated by $S', T$,
denote by $\equiv$ the equality of elements in this group.
Let $W'$ be the normal form of~\eqref{eq:new_powers},
we then substitute $s', t$ for $S', T$,
obtaining $w'$.
Clearly, $w' \eqg \varepsilon$ (in the original free group).
We need to show that $W' \not \equiv 1$ (so that the equation is non-trivial)
and that $s'$ is cyclically reduced.
This will give the contradiction, as $|s'| + |t| < |s|+|t|$
and they also satisfy the condition of the lemma.

We claim that no reduction of $S'$ and $\ov {S'}$ occurs:
suppose that it does, consider the first one that happens in some arbitrary order of reductions.
Suppose that those are $S'$ and $\ov{S'}$ (in this order)
so in between them is a word $\ov T P T \equiv 1$ for some $P$,
as each $S'$ is followed by $\ov T$ and each $\ov {S'}$ is preceded by $T$.
Note that $\ov T P T \equiv 1$ implies that $P \equiv 1$.
By the choice of $P$, there is no reduction of $S'$ and $\ov {S'}$ inside $P$,
so $P$ is generated by $T$ alone, i.e.\ $P = T^k$ and so $k = 0$.
Hence $S' \ov T$ and $T \ov{S'}$ are next to each other before any reduction,
which cannot be by the form~\eqref{eq:new_powers}
(as all $k_1, \ldots, k_{2\ell}$ are non-zero).
The analysis when the order of powers is $\ov {S'}$, $S'$ is similar:
the word $P'$ between them is generated by $T$ on one hand and $P' \equiv 1$ on the other.
So $W' \not \equiv 1$.

Suppose for the sake of contradiction that $s'$ is not cyclically reduced.
Observe that $s'\ov{t} = s$ is cyclically reduced and we have shown already
that in this case the claim of the lemma holds, which contradicts our assumption that $\varepsilon \neq w' \eqg \varepsilon$.
Hence, $s'$ is cyclically reduced and so $w' \eqg \varepsilon$ for $s', t$
is a smaller counterexample, which cannot be.

When the reduction in $s \ov t$ is of length $|t|$
then the argument is the same, we just exchange $t$ for $\ov t$.
For $t s$ we can take the inverse so that the reduction
in $\ov s \, \ov t$ is of length $|t|$, here we exchange $\ov s$ with $s$ and $\ov t$ with $t$.
For $\ov t s$ after the inverse we have $\ov s t$, so we exchange $\ov s$ and $s$.
For the cases when the reduction is of length $|s|$ observe that by symmetry we can swap $s$ and $t$, 
as $s^{k_1} t^{k_2} \cdots s^{k_{2\ell-1}} t^{k_{2\ell}} \eqg \varepsilon$
is equivalent to $t^{k_{2\ell}}s^{k_1} t^{k_2} \cdots s^{k_{2\ell-1}} \eqg \varepsilon$.

In the following we assume that the reduction between $s^{p_s}t^{p_t}$ and $t^{p_t}s^{p_s}$
for $p_s, p_t \in \{-1, 1\}$,
are all of length smaller than $\min(|s|,|t|)$.
We use Lemma~\ref{lem:pseudosolution} for~\eqref{eq:powers_to_powers}:
let us choose each power of $s$ and $t$.
Then by Lemma~\ref{lem:pseudosolution} there is $k_{h-1},k_h,k_{h+1}$ such that $t^{k_h}$ reduces out in

\begin{equation}
\label{eq:no_power_sum_of_powers_local}
s^{k_{h-1}}t^{k_{h}}s^{k_{h+1}}
\end{equation}
or 
$s^{k_h}$ reduces out in
$t^{k_{h-1}}s^{k_{h}}t^{k_{h+1}}$;
the cases are symmetric (exchange of $s$ and $t$),
so we consider only the first one.
The case when the reduction of left or right side is of length at least
$\min(|s|,|t|)$ were already handled.
So we consider the case when both reductions are of length
smaller than $\min(|s|,|t|)$.
Hence $|k_h| = 1$, by symmetry we consider $k_h = 1$,
and $t$ reduces with both $s^{k_{h-1}}$ and $s^{k_{h+1}}$
in $s^{k_{h-1}}ts^{k_{h+1}}$.
Then $k_{h-1}$ and $k_{h+1}$ are of the same sign,
as otherwise we would get that $t$ is not cyclically reduced.
Consider the case when $k_{h-1}, k_{h+1}>0$, the other ones
are symmetric
(exchange of $s$ and $\ov s$ and/or $t$ and $\ov t$).
Then $t = \ov{s''} \, \ov{s'}$, where $s''$ and $s'$
are non-empty suffix and prefix of $s$.
If $|s| = |s'| + |s''|$ then $s = s's'' = \ov{t}$
and we are done.
If $|s| > |s'|+|s''|$ then let $s = s' s_1 s''$ and $t_1 = \ov {s'}\, \ov{s''}$.
Clearly $|s_1| + |t_1| < |s|+|t|$.
We will construct a nontrivial equation of the form~\eqref{eq:powers_to_powers}
and show that $s_1, t_1$ are cyclically reduced,
thus showing the claim.
Note that in the other case, when $|s| < |s'| + |s''|$,
we can choose a prefix $t'$ and suffix $t''$ of $t$ such that
$s = \ov{t''} \, \ov{ t'}$ and $t = t' t_1 t''$; the analysis is symmetric, so we skip it.

To get the equation of the form~\eqref{eq:powers_to_powers}
for $s_1, t_1$, we reorganize the one for $s,t$:
observe that
\begin{align*}
s^k &= (s' s_1 s'')^k\\
	&= s' (s_1 \underbrace{s'' s'}_{\ov {t_1}})^{k-1} s_1 s''\\
	&\eqg
	s' (s_1 \ov {t_1})^{k-1} s_1 \underbrace{s'' s'}_{\ov {t_1}} \ov{s'}\\
	&=	s' (s_1 \ov {t_1})^k \ov{s'}\\
t^k &= (\ov {s''} \, \ov{s'})^k\\
	&= \ov {s''} t_1^{k-1} \ov{s'}\\
	&\eqg s' t_1^k \ov {s'}
\end{align*}
As powers of $s$ and $t$ alternate in~\eqref{eq:powers_to_powers},
the leading $s'$ and ending $\ov s'$ from power of $s$ and the
leading $\ov{s'}$ and ending $s'$ from powers of $t$ reduce,
except for the first and last one in the whole word.
In the end we obtain a word $s' w' \ov{s'} \eqg \varepsilon$,
so also $w' \eqg \varepsilon$ and clearly $w'$
is a concatenation of powers of $s_1$ and $t_1$.
We still need to show that $w'$ is not trivial,
more formally, consider the free group generated by $S_1, T_1$,
let $\equiv$ be the equivalence relation in this group
and consider a word:
\begin{equation*}
(S_1\ov{T_1})^{k_1} \, T_1^{k_2} \cdots (S_1 \ov {T_1})^{k_{2\ell-1}}  \, T_1^{k_{2\ell}} \enspace ,
\end{equation*}
which is exactly the same equation as~\eqref{eq:new_powers},
so we already know that it is non-trivial.

It is left to show that $s_1, t_1$ are cyclically reduced.
As $t = \ov{s''} \, \ov {s'}$ is cyclically reduced,
by Lemma~\ref{lem:cyclic_shifts_reduced_primitive}
also $t_1 = \ov {s'} \, \ov{s''} \shift t$  is cyclically reduced.
Suppose for the sake of contradiction that $s_1$ is not cyclically reduced.
Observe that $s_1 \ov {t_1} = s_1 s'' s' \shift s$ is cyclically reduced.
Then we have already shown that in this case a nontrivial concatenations of powers
of $s_1$ and $\ov{t_1}$ cannot be equivalent to $\varepsilon$,
which contradicts the constructed example of $w'$.
So $s_1$ is cyclically reduced.
Then $w'$ is a smaller counterexample, contradiction.

Concerning the isomorphism defined by a mapping,
clearly the mapping uniquely extends to a homomorphism (as it is from a free group)
and it is surjective.
Its kernel consists of words
$S^{k_1}T^{k_2}\cdots S^{k_{2\ell-1}}T^{k_{2\ell}}$
such that $s^{k_1}t^{k_2}\cdots s^{k_{2\ell-1}}t^{k_{2\ell}} \eqg \varepsilon$,
so it is trivial.
\end{proof}

\section{Data structure}\label{sec:data-structure}
Words appearing naturally in our proofs and algorithms are concatenations of a constant number of subwords
(or their involutions) of the input equation.
We say that a word $w$ is $k$-\emph{represented}, if $w$ is given as $w = (U\ov U)[b_1\twodots e_1] \cdots (U\ov U)[b_k\twodots e_k]$,
where $U = u_1\cdots u_m$ is the concatenation of all words from the equation~\eqref{eq:main}.
A parametric word $s_0 t_1^{\phi_1} s_1 \cdots s_{\ell-1} t_\ell^{\phi_\ell}s_{\ell}$
is $k$-represented, when $s_0, t_1, s_1, \ldots, t_\ell, s_\ell$
are  $k_0, \ldots, k_{2\ell}$ represented and $k = \sum_{i=0}^{2\ell} k_i$.

Intuitively, all basic operations that we perform on (parametric) words
that are $k$ and $\ell$ represented can be performed in $\Ocomp(k + \ell)$ time.
As $k, \ell$ are usually small constants this amounts to $\Ocomp(1)$ time.

We use standard data structures, like suffix arrays~\cite{suffixarrays} and structures for answering longest common prefix queries on them~\cite{lcpsuffixarrays}.
As a result, we can answer all basic queries
(like normal form, longest common prefix, power prefix, 
etc.)
about words in the equation in $\Ocomp(1)$ time;
note that this is the place in which we essentially use that we can perform operations on $\Ocomp(\log n)$-size numbers in $\Ocomp(1)$ time.
As an example of usage, we can test whether a word is a solution in $\Ocomp(m)$ time:

\begin{lemma}
\label{lem:testing_a_single_solution}
Given a word $\alpha u^i v^j \beta$,
where $\alpha, \beta, u, v$ are $\Ocomp(1)$-represented,
$\alpha, \beta$ are reduced and $u, v$ are cyclically reduced and primitive
and $i, j$ are a pair of integer numbers,
we can test whether $\alpha u^i v^j \beta$ is a solution of~\eqref{eq:main}
in $\Ocomp(m)$ time.
\end{lemma}

In the following, we give the appropriate construction.

\begin{lemma}
\label{lem:solution_testing}
For the equation~\eqref{eq:main} we can construct a data structure,
which given two words $s$, $t$ that are $k$ and $\ell$
represented, we can:
\begin{itemize}
\item compute the longest prefix of $s$
that has period $p$ in $\Ocomp(k)$ time,
\item compute the $s$-power prefix and suffix of $t$
in $\Ocomp(k +\ell)$ time,
\item compute $\nf(st)$ in $\Ocomp(k+\ell)$ time.
\end{itemize}
\end{lemma}
\begin{proof}
Let $U$ be a word $U = u_1\cdots u_m$,
where $u_i$s are the words from the equation~\eqref{eq:main},
For the word $U \ov U$ in linear time we can construct a data structure
which answers in constant time the longest common extension query ($\lce$),
i.e.\ given indices $b$, $b'$ return the largest $k$
such that $U\ov U[b\twodots b+k] = U\ov U[b'\twodots b'+k]$.
There are several standard data structures for this query,
say a suffix array~\cite{suffixarrays} plus a structure for answering \emph{longest common prefix query}
(lcp)~\cite{lcpsuffixarrays} on which we use range minimum queries~\cite{rmq}.
The last structure needs the flexibility of the RAM model to run in $\Ocomp(1)$ time per query
and suffix array construction~\cite{suffixarrays} assumes that the alphabet can be associated with
a set of consecutive natural numbers;
if this is not the case then we can sort the letter in total $\Ocomp(n \log n)$ time
and then assign them consecutive numbers, say starting from $1$.
There are also structures based on suffix tree with lowest common ancestor data structure.

The data structure clearly supports $\lcp$ query in $\Ocomp(1)$ time 
for words that are $1$-represented:
if the words $s, t$ are $1$ represented
then we known the positions $b$ and $b'$ at which they begin in $U \ov U$
and ask $\lce(b,b')$.
If the answer is longer then the minimum of the lengths of the words suffixes
then we cap it at the minimum of those lengths.
It is easy to generalize the query to the case when one word is $k$ represented
and the other $\ell$-represented:
if $s=s_1\cdots s_k$ and $t = t_1\cdots t_\ell$
then we first check the longest common prefix of $s_1$ and $t_1$:
if $|\lcp(s_1,t_1)| < \min(|s_1|,|t_1|)$ then it is also the longest common extension of $s$ and $t$.
Otherwise, when $\lcp(s_1,t_1) = |t_1|$
then $\lcp(s,t) = |t_1| + \lcp(s_1[1+|t_1|\twodots |s_1|],t_2\cdots t_\ell)$
(the case when $\lce(s_1,t_1) = |s_1|$ is done symmetrically).
Each step removes one subword from $k + \ell$ ones, so it takes $\Ocomp(k + \ell)$ time.
In fact, the argument above shows a slightly more refined bound:
if the $\lcp(s,t)$ is contained within
$s_1\cdots s_{k'}$ and $t_1\cdots t_{\ell'}$ then the running time is $\Ocomp(1 + k'+\ell')$.

Before we describe how to use longest common extension query to compute the normal form
observe that if $s = s_1 \cdots s_k$ where each $s_i$ is $1$-represented,
then $\ov s = \ov {s_k} \cdots \ov{s_1}$ and each $\ov{s_i}$ is also $1$-represented,
moreover if $s_i$ is represented effectively,
say we know $b, e$ such that $s_i = U \ov U[b\twodots e]$,
then $\ov{s_i}$ is also $1$-represented, i.e.\ $\ov {s_i} = U \ov U[2|U| - e\twodots 2|U| - b]$.

We will describe how to compute the $\nf(s)$ in $\Ocomp(k)$ time, the computation for $\nf(st)$ is similar.
We will iteratively compute $\nf(s_1 \cdots s_i)$ for consecutive $i$s,
the $\nf(s_1 \cdots s_i)$ is represented as a reduced word $s_1' \cdots s_{i'}'$,
where each $s_{j'}'$ is $1$-represented
(it is also a subword of some $s_{j}$, but this is not important for computation).
For $i=1$ we simply take $s_1' = s_1$, as $s_1$ is reduced by definition.
When we add $s_{i+1}$ we compute in $\Ocomp(1)$ time the reduction between
$s_{i'}'$ and $s_{i+1}$, this is exactly the longest common prefix of  $s_{i'}'$ and $\ov{s_{i+1}}$,
and we shorten the words appropriately.
If after the shortening $s_{i+1} = s_{i'}'= \varepsilon$ then we remove $s_{i'}'$ from the representation
and we are done.
If after the shortening $s_{i+1} = \varepsilon$ then we drop it and we are done (we replace $s_{i'}'$ in the representation with its shortened variant).
If after the shortening $s_{i'}' = \varepsilon$ then we remove it from the representation
and continue with the current representation and shortened $s_{i+1}$.
Clearly in each step either we remove one word from the representation or add one,
so the whole running time is $\Ocomp(k)$ and the obtained word is also $k$-represented.

Concerning the longest prefix of $t$ that has a period $p$ observe that this is
$p + \lcp(t,t[1+p\twodots|t|])$, which can be computed in $\Ocomp(\ell)$ time,
as $t[1+p\twodots|t|]$ is also $\ell$-represented.
The computation of the longest suffix of $t$ has has period $p$ is done in a similar way.

In order to compute the $s$-power prefix of $t$
we check, whether the $t[1\twodots|s|] \in \{s,\ov s\}$,
by computing $\lcp(t,s)$ and $\lcp(t,\ov s)$.,
this can be done in $\Ocomp(k + \ell)$ time.
This determines whether the $s$-power prefix is $\varepsilon$
and whether it is a power of $s$ or $\ov s$.
Then we compute the longest prefix of $t$ has period $|s|$,
which can be done in $\Ocomp(k)$ time.

The computation of $s$-power suffix is done in a similar way.
\end{proof}

When we want to verify, whether $\alpha u^i v^j \beta$ is solution,
we naturally arrive at a situation in which we need to manipulate words that are
represented as concatenations of $1$-represented words \emph{and} runs of $u, v$.
It turns out that operations as in Lemma~\ref{lem:solution_testing}
can still be performed effectively, at least for $u, v$
(assuming that they satisfy some mild conditions).

\begin{lemma}
\label{lem:solution_testing_2}
Given a data structure from Lemma~\ref{lem:solution_testing}
and two words $u,v$ that are primitive, cyclically reduced and $k$-represented
then for a sequence $s_1, \ldots s_\ell$ of words,
such that each $s_i$ is either a run of $u$ or $v$ or $1$-represented word,
we can compute $\nf(s_1\cdots s_\ell)$ in $\Ocomp(k\ell)$ time.
\end{lemma}
\begin{proof}
We proceed as in Lemma~\ref{lem:solution_testing},
i.e.\ we read the words one by one,
after reading $s_1, \ldots s_{i-1}$
we keep $s'_1, \ldots, s'_{i'-1}$
such that $\nf (s_1 \cdots s_{i-1}) = s'_1 \cdots s'_{i'-1}$
and each $s'_j$ is either $1$-represented or a $u$ or $v$-run.
It is enough to show that processing $s_i$
is proportional to $k$ times the number of removed $s'_{j}$s
plus $1$.

When we read $s_i$ we compute $\nf(s'_{i'-1}s_i)$.
If $\nf(s'_{i'-1}s_i) = \varepsilon$ then we simply remove $s'_{i'-1}$ and finish.
If whole $s'_{i'-1}$ is reduced in $\nf(s'_{i'-1}s_i)$ then we remove it from the stack and continue
(with the new read symbol $\nf(s'_{i'-1}s_i)$ and next topmost symbol).
Otherwise, if $s_i$ was wholly reduced in $\nf(s'_{i'-1}s_i)$
then we put $\nf(s'_{i'-1}s_i)$ on the stack (instead of $s'_{i'-1}$).
Each step either puts or removes one symbol on the stack,
so it is enough to show that we can compute the reduction
in $s'_{i'-1}s_i$ in $\Ocomp(k)$ time;
note that this is exactly $\lcp(\ov {s'_{i'-1}}, s_i)$.

If $s'_{i'-1}, s_i$ are both not runs then this follows directly from Lemma~\ref{lem:solution_testing} (and the running time is $\Ocomp(1)$).
If one (say $s'_{i'-1}$, but the situation is symmetric) is a run
(say of period $p$) and the other is a $1$-represented word
then we compute the longest prefix of $s_i$ that has period $p$ 
and check whether the first $p$ letters of $s_i$ and $\ov {s'_{i'-1}}$ are the same.
As $s'_{i'-1}$ is $k$-represented and $s_i$ is $1$-represented,
this is done in $\Ocomp(k)$ time, by Lemma~\ref{lem:solution_testing}.
This gives the length of the reduction.
The case when $s_i$ is a run and $s'_{i'-1}$ not is done in a symmetric way.

If both $s'_{i'-1}, s_i$ are runs, then there is a distinction, whether they are runs of words of the same period or not,
they can have different periods only when one is a run of $u$, the other of $v$ and $|u| \neq |v|$.
Let $\ov {s'_{i'-1}}$ and $s_i$ have different periods $p' \neq p$, say $p' < p$,
the other case is symmetric.
Then their longest common prefix has length less than $p + p'< 2p$,
see Lemma~\ref{lem:different_runs_overlap}.
So it is enough to compute the longest common prefix
for $\ov{s'_{i'-1}}$ and $s_i[1\twodots 2p]$,
and the latter is $\Ocomp(k)$ represented.
This was considered in the previous case
and can be done in $\Ocomp(k)$ time.
If $\ov s'_{i'-1}, s_i$ have the same period $p$
then their longest common prefix is either shorter than $p$,
or of length $\min(|s'_{i'-1}|,|s_i|)$.
So it is enough to compute the $\lcp(s'_{i'-1}[1\twodots p], s_i[1\twodots p]])$
which can be done in $\Ocomp(k)$ time, as both words are $\Ocomp(k)$ represented.
\end{proof}

\begin{proof}[Proof of lemma~\ref{lem:testing_a_single_solution}]
We evaluate~\eqref{eq:main} under the substitution $x$ for $X$.
Observe that under such a substitution the obtained word is a concatenation of $\Ocomp(m)$
words as required by Lemma~\ref{lem:solution_testing_2}: 
i.e.\ each is either $1$-represented or a $u$ or $v$-run.
So its normal form can be computed in $\Ocomp(m)$ time, as required.
\end{proof}

\section{Superset of solutions}\label{sec:superset-of-solutions}

The previous characterization~\cite{onevariablefreegroup}
essentially showed that a solution is either a $\Ocomp(1)$-represented word
or of the form $u^iu'v''v^j$ for some $i,j \in \mathbb Z$
and $u' \pref u, v \suff v''$ for some well defined $u, v$.
As we intend to analyze those solutions using word combinatorics,
it~is~useful to assume that $u, v$ are cyclically reduced and
primitive.
Unfortunately, this~cannot be extracted directly 
from the previous characterization,
so we repeat the previous arguments taking some extra care.

\begin{lemma}[{cf.~\cite[Lemma~15]{gilman_myasnikov_one_var_FG}}]
\label{lem:solutions_superset}
For a given equation~\eqref{eq:main}, in $\Ocomp(n^2)$ time one can compute
a superset of solutions of the form
\begin{equation*}
S \cup \bigcup_{(\alpha u^Iv^J \beta) \in W} \bigcup_{i,j \in \mathbb Z}  \{\alpha u^{I(i)}v^{J(j)} \beta\}
\end{equation*}
where
$S$ is a set of 
$\Ocomp(1)$-represented words with $|S| = \Ocomp(n^2)$
and 
for each $0 \leq i \leq m-1$ there are numbers $\ell_{i}, \ell_{i}' \leq  |u_i|+|u_{i+1}|$ 
such that $W$ contains exactly
$\ell_{i} \cdot \ell_{i}'$ parametric words satisfying
\begin{itemize}
\item $\alpha$, $\beta$, are $\Ocomp(1)$-represented,
reduced and $|\alpha|,|\beta| \leq |u_i|+|u_{i+1}|$;
\item $u$, $v$ are $2$-represented, cyclically reduced,
primitive and $|u| = \ell_i$ and $|v| = \ell_{i}'$.
\end{itemize}
\end{lemma}

The main principle of the proof of Lemma~\ref{lem:solutions_superset}
is that when $x$ is a solution of an equation~\eqref{eq:main},
then after the substitution the obtained word is reducible
and thus by Lemma~\ref{lem:pseudosolution},
one of substituted $x$ or $\ov x$ is a pseudosolution in
$x^{p_h} u_h x^{p_{h+1}} u_{h+1} x^{p_{h+2}}$, where $p_h, p_{h+1}, p_{h+2} \in \{-1,1\}$.
Thus we analyze each possible triple $p_h, p_{h+1}, p_{h+2}$
and show the possible form of the pseudosolution in corresponding case.
Note that by symmetry we can consider $p_{h+1} = 1$.

We begin with some preliminary Lemmata.

\begin{lemma}
\label{lem:x-1uxform}
Let $x_u \pref x$ be a pseudo-solution of $\ov x \alpha u \ov \alpha \underline{x_u}$,
where $x, \alpha u \ov \alpha $ are reduced
and $u$ is cyclically reduced.
Then 
\begin{itemize}
\item $x_u \pref \alpha \ov u \,\ov \alpha$ or
\item $x_u \eqg \alpha u^i u'$ where $u' \pref u$ and $i \in \mathbb{Z}$
\end{itemize}
\end{lemma}

\begin{proof}
\begin{figure}
\centering
\includegraphics[scale=1.4]{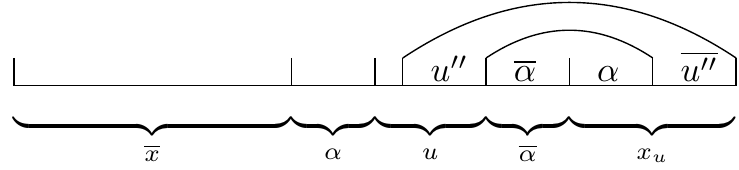}
\caption{Pseudo-solution for the equation $\ov x \alpha u \ov \alpha \underline{x_u}$. Case when whole $x_u$ is reduced  within $\alpha u \ov \alpha $.}
\label{fig:x1xx12}
\end{figure}

If $f(x_u) \subseteq \alpha u \ov \alpha$, see Fig.~\ref{fig:x1xx12},
then $x_u$ is an inverse of some suffix of  $\alpha u \ov \alpha$,
so 
$x_u \pref \alpha \ov u \, \ov \alpha$, as claimed.

In the remaining case observe that we may assume that $\alpha \pref x_u$:
consider the reduction pairing $f$, observe that either $\ov x \alpha$ reduces the whole $\alpha$
or $\ov \alpha x_u$ the whole $\ov \alpha$: if none of this happens then $x_u$ reduces within $\ov \alpha x_u$,
which was considered.
But then in either case $\alpha \pref x$ and so $\alpha \pref x_u$
or $x_u \pref \alpha$, the latter was already considered.

Let $x' \pref x$ be the minimal prefix of $x$ such that $\ov{x'} \alpha u \ov \alpha x_u \eqg \varepsilon$.
If $x' \pref \alpha$ then again we end in the case such that $f(x_u) \subseteq \alpha u \ov \alpha$.
So $\alpha \pref x'$.
As $\ov{x'} \alpha u \ov \alpha x_u \eqg \varepsilon$ we can modify the pairing by first pairing
the suffix $\ov \alpha$ of $x'$ with $\alpha$ and the $\ov \alpha$ with the prefix $\alpha$
of $x_u$, see Fig.~\ref{fig:x1xx13},
and then extend to the rest of $\ov{x'} \alpha u \ov \alpha x_u$.
Note that $x$ is still a pseudosolution for such a modified pairing.

\begin{figure}
\centering
\includegraphics[scale=1.4]{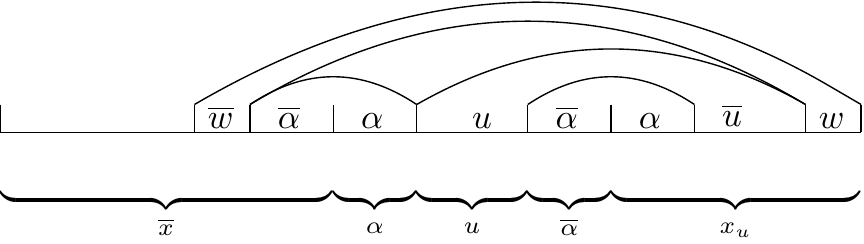}
\caption{Pseudo-solution for the equation $\ov x \alpha u \ov \alpha \underline{x_u}$. Case when $x_u$ is not reduced within $\alpha u \ov \alpha$ and $u$ is paired with $x_u$.}
\label{fig:x1xx13}
\end{figure}

Observe now that it cannot be that letters in $u$ are paired with both $\ov x$ and $x_u$,
as this would imply that the first and last letter of $u$ are paired
with a corresponding letter of $\ov x$ and $x$, respectively.
But then $u$ would not be cyclically reduced.

Consider first the case when (some) letters of $u$ are paired with $x_u$.
Then whole $u$ is paired with $x_u$ and the rest of $x_u$ is paired with $\ov x$
(as otherwise $f(x_u) \subseteq u \ov \alpha$, which was already considered),
see Fig.~\ref{fig:x1xx13}.
Thus $x_u = \alpha \ov u w$ for some $w$ and $\ov x \suff \ov w \, \ov \alpha$, which implies $\alpha w \pref x$.
Comparing $x_u = \alpha \ov u w \pref x$ with $\alpha w \pref x$ we get that $w \pref \ov u$
or $\ov u$ is period of $w$ and so
$x_u = \alpha \ov u ^ k \ov {u''}$, where $\ov {u''} \pref \ov u$ and $k \geq 0$;
this can be alternatively represented as
$x_u \eqg \alpha u ^ {-k-1} u'$, where $u = u'u''$.

\begin{figure}
\centering
\includegraphics[scale=1.4]{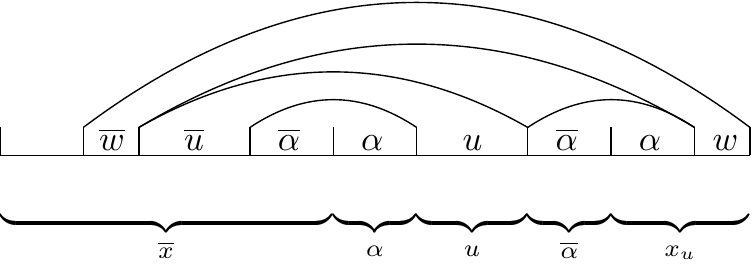}
\caption{Pseudo-solution for the equation $\ov x \alpha u \ov \alpha \underline{x_u}$.
Case when $x_u$ is not reduced within $\alpha u \ov \alpha$ and $u$ is not paired with $x_u$.}
\label{fig:x1xx14}
\end{figure}

The analysis for the case when some letter of $u$ is paired with $\ov x$
is symmetric:
let $x_u = \alpha w$.
As some letter of $x_u$ is paired with $\ov x$ then all letters in $u$ are paired
and so all of them are paired with $\ov x$, see Fig.~\ref{fig:x1xx14},
then $\ov x \suff \ov w \, \ov u \, \ov \alpha $,
which implies $\alpha u w \pref x$, and $x_u = \alpha w \pref x$,
thus $w \pref u$ or $u$ is a period of $w$ and so $x_u = \alpha u^k u'$
for some $u' \pref u$ and $k \geq 0$.
\end{proof}

\begin{lemma}
\label{lem:xvxform}
Let $x_v$ be a pseudo-solution (for some partial pairing $f$) of
$\underline{x_v} v x_ux_v$ but not in $\underline{x_v} v x_u$
(for the restriction of $f$).
Then $x_u x_ v = \ov{v''} \, \ov{v'}$ for some prefix $v' \pref v \suff v''$.
\end{lemma}
\begin{proof}
\begin{figure}
\centering
\includegraphics[scale=1.4]{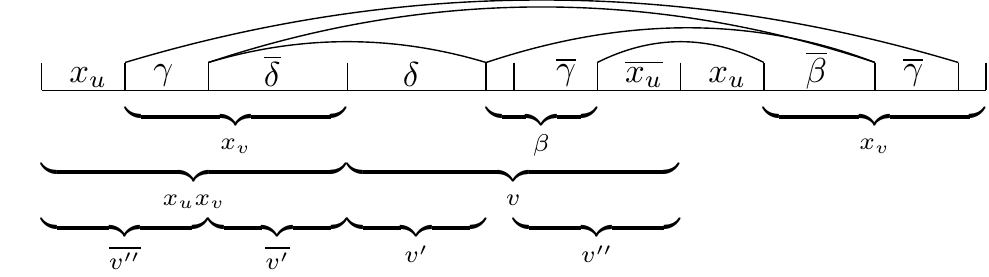}
\caption{Pseudosolution for the equation $\underline{x_v} v x$.
The case in which $x_v$ reduces within $x_v v x_u x_v$ but not inside $x_v v x_u$. 
We take into account that not the whole $v$ is reduced in $\nf(x_v v x_u)$.}
\label{fig:x1xx1}
\end{figure}
First, the whole $x_u$ is reduced within $\nf(x_v v x_u)$:
if not then there would be no further reduction in $\nf(x_v v x_u)x_v$,
as $x_u x_v$ is reduced, and so whole $x_v$ reduces within $x_v v x_u$,
which is forbidden by Lemma assumption.
If also whole $v$ is reduced within $x_v v x_u$ then we are left with a prefix $x_v'$ of $x_v$
that should reduce with $x_v$, i.e.\ $x_v'$ should reduce with $x_v'$,
which cannot happen.
So not the whole $v$ is reduced in $\nf(x_v v x_u)$,
see Fig.~\ref{fig:x1xx1}.
So $v = \delta \beta \ov {x_u}$, where $\delta$ is the maximal prefix that reduces with the preceding $x_v$ and $\ov{x_u}$ reduces with the following $x_u$.
Then $\beta$ reduces with the following $x_v$.
Also, in the first $x_v$ the remaining part
(i.e.\ after reduction of $\delta$), call it $\gamma$,
reduces with the remaining part of the second $x_v$.
So $x_v  = \gamma \ov \delta$ and $\ov \beta \ov \gamma \pref x_v$
and from Lemma~\ref{lem:inverse_no_overlap} we get that $\gamma \pref \ov \beta$ (or $\gamma = \varepsilon$, but this is covered by $\gamma \pref \ov \beta$).
It is left to observe that
\begin{equation*}
x_ux_v = x_u \gamma \ov \delta
\end{equation*}
and $\ov \gamma \, \ov {x_u}$ is a suffix of $v$
(as $\gamma \pref \ov \beta$ implies $\beta \suff \ov \gamma $)
and $\delta$ is a prefix of $v$,
so $x_ux_v = \ov{v''} \, \ov{v'}$ for some prefix $v'$ ($=\gamma$) of $v$ and suffix $v''$ ($=\ov \delta \, \ov{x_u}$), as claimed.
\end{proof}

\begin{lemma}
\label{lem:x-1uxvx-1form}
Let $x$ be a pseudo-solution of $\ov x\alpha u \ov \alpha \underline{x}\beta v \ov \beta \ov x$,
where $x, \alpha u \ov \alpha ,\beta v \ov \beta$ are reduced
and $u, v$ are cyclically reduced.
Then $x = x_u x_v$, where
\begin{itemize}
\item $x_u \pref \alpha \ov u \,\ov \alpha$ or
\item $x_u \eqg \alpha u^i u'$ where $u' \pref u$ and $i \in \mathbb{Z}$
\end{itemize}
similarly
\begin{itemize}
\item $ \beta \ov v \,  \ov \beta \suff x_v$ or
\item $x_v \eqg v'' v^j \ov \beta $ where $v \suff v''$ and $j \in \mathbb{Z}$.
\end{itemize}
\end{lemma}
\begin{proof}

Fix a (partial) reduction pairing $f$ such that whole middle $x$ is paired.
Define $x_u, x_v$ such that $x = x_ux_v$, $f(x_u) \subseteq x^{-1}\alpha u \ov \alpha$
and $f(x_v) \subseteq \beta v \ov \beta x^{-1}$,
i.e.\ as the prefix of $x$ that is reduced to the left and suffix that is reduced to the right.
Note that this is correct, as $x$ is reduced and so no pairing is done inside it.
Then Lemma~\ref{lem:x-1uxform}
applied to $x \alpha u \ov \alpha x_u$ 
and $x_v \beta v \ov \beta x$ yields the claim
(note that $\ov {u'} u^k \eqg u'' u^{k-1}$).
\end{proof}

\begin{lemma}
\label{lem:x-1uxvxform}
Let $x$ be a pseudo-solution of $\ov x\alpha u \ov \alpha \underline{x} v x$,
where $x, \alpha u \ov \alpha, v$ are reduced and $u$ is cyclically reduced.
Then either
\begin{itemize}
\item $x = \ov{v''} \, \ov{v'}$, where $v' \pref v \suff v''$ or
\item $x = x_u x_v$ where
\begin{itemize}
\item $x_u \pref \alpha \ov u \, \ov \alpha$ or
\item $x_u \eqg \alpha u^i u'$ where $u' \pref u$ and $i \in \mathbb{Z}$
\end{itemize}
and
\begin{itemize}
\item $\nf (\alpha u \ov \alpha \, \ov v) \suff x_v$ or
\item $x_v \eqg  u''  u^j \ov \alpha \, \ov v $,
where $u \suff u''$ and $j \in \mathbb{Z}$.
\end{itemize}
\end{itemize}\end{lemma}

\begin{proof}
Fix a (partial) reduction pairing $f$ such that whole middle $x$ is paired.
Let $x_u, x_v$ be such that $x = x_ux_v$ and $f(x_u) \subseteq x^{-1} u$ and $f(x_v) \subseteq v x$.
Consider first the case in which $f(x_v) \not \subseteq v x_u$.
Then Lemma~\ref{lem:xvxform} yields that $x = \ov{v''} \, \ov{v'}$, where $v' \pref v \suff v''$,
as claimed.
Thus we are left with the case when $f(x_v) \subseteq v x_u$, i.e.\ $\nf(\ov x_u \ov v) \suff x_v$.
Applying Lemma~\ref{lem:x-1uxform}
to the $x \alpha u \ov \alpha x_u$ yields that the form of
$x_u$ is as claimed.
Substituting the form of $x_u$ to $\nf(\ov x_u \ov v) \suff x_v$
yields the form of $x_v$.
\end{proof}

\begin{lemma}
\label{lem:xux-1vxform}
Let $x$ be a pseudo-solution of $x u \underline{x} v x$,
where $x, u, v$ are reduced then either
\begin{itemize}
\item 
$x = \ov{v''} \, \ov {v'}$ or
$x = \ov{u'} \, \ov{u''}$ or
$x =  \ov{u''} \, \ov{v'}$ or
$x \eqg \ov{u''} \, u^{\bullet \bullet} \ov v$ or
$x \eqg \ov u v^{\bullet} \ov {v'}$
where $v^\bullet \pref v' \pref v \suff v''$,
$u' \pref u \suff u'' \suff u^{\bullet \bullet}$;
\item $x = x_u x_v$,
where
\begin{itemize}
\item $x_u \pref \alpha$ or
$x_u = \alpha r_u^i r_u'$ for some $i\in \mathbb N$,
where $r_u' \pref r_u$ and $\ov u v = \alpha \ov r_u \ov \alpha$
and $r_u$ is cyclically reduced;
\item $\beta \suff x_v$ or $x_ v = r_v'' r_v^j \beta$
for some $j \in \mathbb N$ where
$u \ov v = \ov \beta r_v \beta$ and
$r_v \suff r_v''$ and $w_v$ is cyclically reduced;
\end{itemize}\end{itemize}
\end{lemma}

\begin{proof}
\begin{figure}
\centering
\includegraphics[scale=1.4]{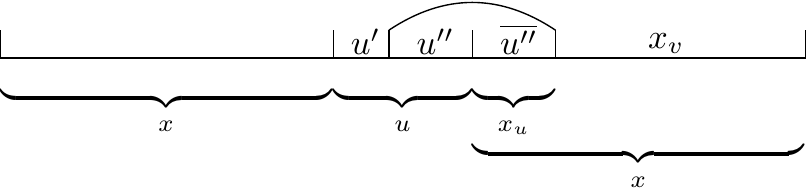}
\caption{Pseudosolution for the equation $x u \underline{x_u}$;
word $u''$ is maximal reducing with $x_u$.
The case in which $x_u = \ov{u''}$.}
\label{fig:xxx1}
\end{figure}

Fix a partial reduction pairing $f$ such that whole middle $x$ is paired.
Define $x_u$ and $x_v$ such that $x = x_ux_v$
and $f(x_u) \subseteq x u$ and $f(x_v) \subseteq v x$.
Consider $f(x_u) \subseteq x u = x_u x_v u$;
if $f(x_u) \not \subseteq x_v u$ then by Lemma~\ref{lem:xvxform}
we get that $x = \ov{u'} \, \ov{u''}$, where $u' \pref u \suff u''$.
Similarly if $f(x_v) \not \subseteq v x_u$
then 
by Lemma~\ref{lem:xvxform}
we get that $x = \ov{v''} \, \ov{v'}$, where $v' \pref v \suff v''$.
So in the following we may assume that
$f(x_u) \subseteq x_v u$ and
$f(x_v) \subseteq v x_u$.

Let $u = u' u''$ where $f(u'') \subseteq x_u$ is maximal with this property.
Either $x_u = \ov {u''}$, see Fig.~\ref{fig:xxx1},
or $\ov {u''} \pref x_u$, see Fig.~\ref{fig:xxx2}.
In the latter case, as not whole $x_u$ is paired with $u''$,
some of its letters need to be paired with the preceding $x_v$ and so the whole $u'$
is paired with this $x_v$ as well, see Fig.~\ref{fig:xxx2},
in particular, $x_v \suff \ov{u'}$.
Then $x_u = \ov{u''}w_u$ for some reduced $w_u \neq \varepsilon$ and $x_v \suff \ov {w_u} \ov {u'}$.

\begin{figure}
\centering
\includegraphics[scale=1.4]{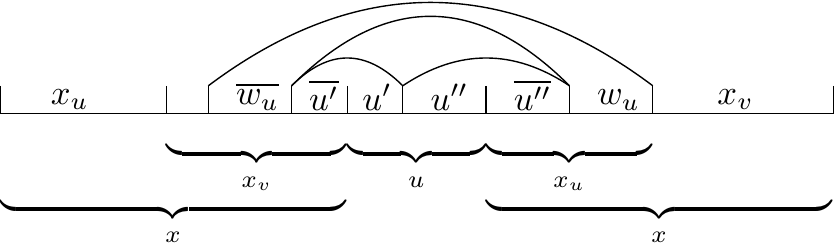}
\caption{Pseudosolution for the equation $x_v u \underline{x_u}$;
word $u''$ is maximal reducing with $x_u$.
The case in which $x_u \neq \ov{u''}$ and so $\ov{u''} \pref x_u$.}
\label{fig:xxx2}
\end{figure}

Similarly, define $v = v'v''$ where $f(v') \subseteq x_v$ is maximal of this property.
Either $ x_v =  \ov{v'}$ or 
$x_v = w_v \ov{v'}$ and $\ov{v''} \ov {w_v} \pref x_u$
for some reduced $w_v \neq \varepsilon$, see Fig.~\ref{fig:xxx3}.

\begin{figure}
\centering
\includegraphics[scale=1.4]{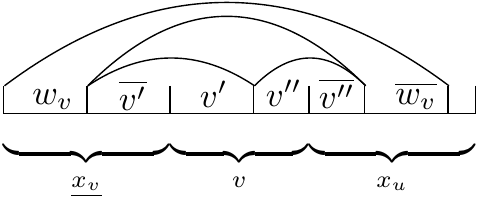}
\caption{Pseudosolution for the equation $ \underline{x_v} v x_u$;
word $v'$ is maximal reducing with $x_v$.
The case in which $x_v \neq \ov{v'}$ and so $x_v \suff \ov{v'}$.}
\label{fig:xxx3}
\end{figure}

Taking into account the form of $x_u$ and $x_v$, there are in total four subcases:

First, if $x_u = \ov{u''}$ and $x_v = \ov{v'}$ then this is 
of the desired form in the first.

Consider the case when $x_u = \ov{u''}$ and $x_v = w_v \ov{v'}$
(and $\ov{v''} \ov {w_v} \pref x_u = \ov{u''}$);
in particular, $x =  \ov{u''} w_v \ov{v'}$,
see Fig.~\ref{fig:xxx4}.
Let $\ov{u^{\bullet \bullet}} = \ov{v''} \ov{w_v}$;
note that $u'' \suff u^{\bullet \bullet}$.
Observe that
\begin{align*}
v \ov{u^{\bullet \bullet}}
	&=
v' v'' \ov{v''} \ov{w_v}\\	
	&\eqg
v' \ov{w_v}\\
	&=
\ov{x_v}
\end{align*}
And so $x = x_u x_v \eqg \ov{u''} u^{\bullet \bullet} \ov v$ for some $u \suff u'' \suff u^{\bullet \bullet}$,
as listed in the first point.
\begin{figure}
\centering
\includegraphics[width=\textwidth]{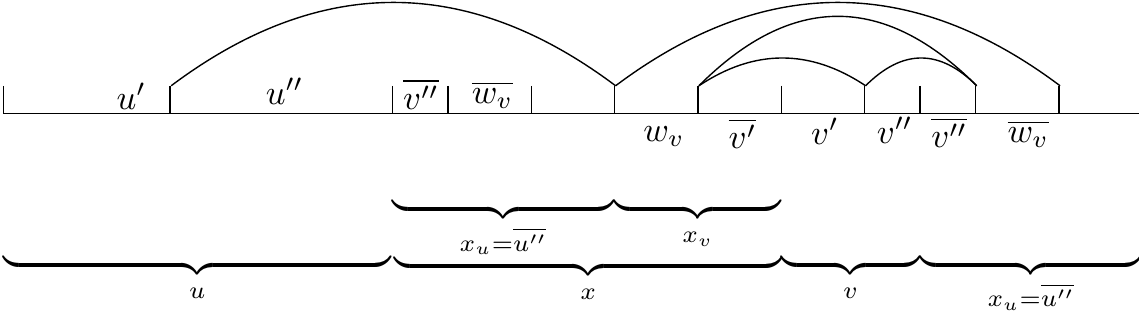}
\caption{Pseudosolution for the equation $x_v u \underline{x} v x_u$.
The case in which  $x_u = \ov{u''}$, $x_v = w_v \ov{v'}$
and $\ov{v''} \ov{w_v} \pref \ov{u''}$.}
\label{fig:xxx4}
\end{figure}

The case of $x_u = \ov{u''} w_u$ and $x_v = \ov{v'}$ is symmetric to the above.

\begin{figure}
\centering
\includegraphics[width=\textwidth]{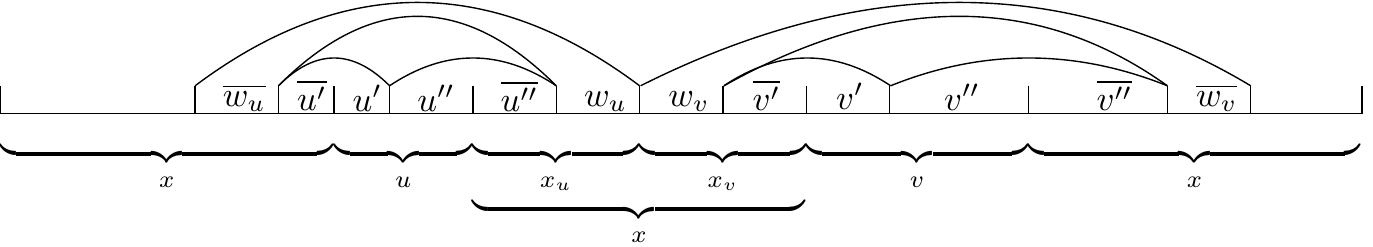}
\caption{Pseudosolution for the equation $x u \underline{x} v x$.
The main case: $x_u = \ov{u''} w_u$ and $x_v = w_v \ov{v'}$.}
\label{fig:xxx6}
\end{figure}

Now let us consider the main case: when $x_u = \ov{u''} w_u$
and $x_v = w_v \ov{v'}$,
see Fig.~\ref{fig:xxx6}.
In particular, $x = \ov{u''} w_u w_v \ov{v'}$.
Note that by case assumption also $x_v \suff \ov{w_u} \, \ov{u'}$
and $\ov{v''} \, \ov{w_v} \pref x_u$,
see Fig.~\ref{fig:xxx7}:
\begin{equation*}
\ov{v''} \ov {w_v} \pref \ov{u''} w_u \quad w_v \ov{v'} \suff \ov{w_u} \ov{u'} 
\end{equation*}

Let us analyze $x_u = \ov{u''} w_u$.
In the following, it is convenient to consider not only the prefix relation on normal forms.
Note that in general $\alpha \pref \beta$ does not imply
$\nf(s) \pref \nf(t)$ and
$\nf(s) \pref \nf(t)$ does not imply $\nf(s's) \pref \nf(s't)$.
However, if $\nf(s) \pref\nf(t)$ and the reductions in $us$ leading to $\nf(us)$ do not reduce whole $s$
then $\nf(us) \pref \nf(ut)$:
first without loss of generality we may assume that $s, t, u$ are reduced
(if they are not then we can reduce them, without affecting any claim).
Let $s = s's''$ and $u = u'\ov{s'}$ where $\nf(us) = u's''$, let also $t = s t''$.
Then $\nf(us) = u's''$
and $\nf(ut) = u's''t''$, as clearly there are no reduction in $s'u''t''$.
Which shows the claim.

Note that $ w_v \ov{v'} \suff \ov{w_u} \ov{u'} \implies u' w_u  \pref v' \ov{w_v}$:
\begin{align*}
\ov{v''} \ov {w_v}
	&\pref
\ov{u''} w_u\\
v' v'' \ov{v''} \ov {w_v} 
	&\pref
v \ov{u''} w_u &\text{Multiply by } v = v'v''\\
\nf(v' \ov {w_v})
	&\pref
\nf(v \ov{u''} w_u) &\text{Reduce left-hand side}\\
v' \ov {w_v}& \pref \nf(v \ov{u''} w_u)
&\nf(v' \ov {w_v}) = v' \ov {w_v}\\
u' w_u &\pref \nf(v \ov{u''} w_u)
&\text{Transitivity}\\
\nf(\ov {u''} \, \ov {u'} u' w_u)
	&\pref
\nf(\ov uv \ov{u''} w_u) &\text{Multiply by } \ov u = \ov {u''} \, \ov {u'}\\
\nf(\ov {u''} w_u) &\pref \nf(\ov u v \ov{u''} w_u) &\text{Reduce left-hand side}\\
x_u &\pref \nf(\ov u v x_u) & x_u = \ov{u''} w_u = \nf(\ov{u''} w_u)
\end{align*}

Multiplying by $v$ is legal: $\ov{v''}\ov{w_v}$ is irreducible
and after the multiplication the remaining
$v' \ov{w_v}$ is also irreducible.
Similarly, after the multiplication by $\ov u$ the $\ov{u''} w_u$ is irreducible, so the multiplication is legal.

\begin{figure}
\centering
\includegraphics[scale=1.4]{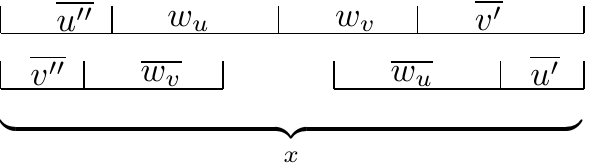}
\caption{Pseudosolution for the equation $x u \underline{x} v x$.
The main case: $x_u = \ov{u''} w_u$ and $x_v = w_v \ov{v'}$.
The depiction of prefixes and suffixes of $x$.}
\label{fig:xxx7}
\end{figure}

Let $\nf(\ov u v) = \alpha r_u \ov \alpha$, where 
$r_u$ is cyclically reduced.
Observe that in $\alpha r_u \ov \alpha x_u$
we can reduce at most half of letters in $\alpha r_u \ov \alpha$,
as otherwise $\nf(\alpha r_u \ov \alpha x_u)$ has less letters than $x_u$, which is its prefix.
Hence $\nf(\alpha r_u \ov \alpha x_u)$ begins with $\alpha a$, where $a$ is the first letter of $r_u$.
If $x_u \pref \alpha$ then we are done.
Otherwise $x_u = \alpha a x_u''$ and so 
$\alpha r_u \ov \alpha x_u \eqg \alpha r_u a x_u''$
and $a x_u'' \pref \nf (r_u a x_u'')$.
Note that $r_u a x_u''$ is reduced:
if there is a reduction in $r_u a$
then the last letter of $r_u$ is $\ov a$ and $a$ is the first letter of $r_u$,
contradiction, as $r_u$ is cyclically reduced.
Hence $a x_u'' \pref r_u a x_u''$
and so $a x_u'' = r_u^i r_u'$ for some $i \geq 0$ and $r_u' \pref r_u$, and so $x_u = \alpha r_u^ir_u'$, as claimed.

A similar analysis applies to $x_v$.
\end{proof}

\begin{proof}[proof of Lemma~\ref{lem:solutions_superset}]
Consider an equation~\eqref{eq:main} and suppose that $x$ is its solution.
By Lemma~\ref{lem:pseudosolution},
$x$ is a pseudosolution in
$x^{p_h} u_h x^{p_{h+1}} u_{h+1} x^{p_{h+2}}$, for some $h$.
Depending on $p_h, p_{h+1}$ and $p_{h+1}$
Lemmata~\ref{lem:x-1uxvx-1form}, \ref{lem:x-1uxvxform}, \ref{lem:xux-1vxform}
yield $(|u_h|+|u_{h+1}|+2)^2$ candidates that are $\Ocomp(1)$-represented
or that $x$ is of the form $\alpha u^i u' v'' v^j \beta$,
where $\alpha, u, v, \beta$ are uniquely defined by 
$x^{p_h} u_h x^{p_{h+1}} u_{h+1} x^{p_{h+2}}$
and $|\alpha u|, |v \beta| \leq |u_hu_{h+1}|$
and $u' \pref u$ and $v \suff v''$ are some prefix and suffix.

In the first case we can compute the candidates in $\Ocomp(1)$ per candidate
(representing them as offsets to words in equation)
and there are (we set $u_0 = \epsilon$ to streamline the calculations)
\begin{align*}
\sum_{h=0}^{m-1} (|u_h|+|u_{h+1}|+2)^2
	&\leq 
\left(\sum_{h=0}^{m-1} (|u_h|+|u_{h+1}|+2)\right)^2\\
	&\leq
\left(\left(2\sum_{h=1}^{m} |u_h|\right)+2m)\right)^2\\
	&=	
(2n)^2\\
	&=
4n^2
\end{align*}
such candidate solutions, as claimed.

In the second case first consider, whether
$u$ is primitive, which can be verified in $\Ocomp(|u|)$ time
(say, using the data structure from Lemma~\ref{lem:solution_testing}
we check whether $p$ is a period
for $p | |u|$).
If $u = w^k$ then we can represent
$u^i u'$ as $w^{ki + i'}w'$ for some $w' \pref w$
and $i'$ such that $w^{i'}w' = u'$;
let also $w''$ be such that $w = w' w''$.
Then $\alpha w^{ki+i'} w' = \alpha w' (w''w')^{ki+i'}$
which is of the promised form.
The same transformation can be carried out for $v'' v^j \beta$.

Clearly $|\alpha w'| < |\alpha w| \leq |\alpha u|\leq |u_hu_{h+1}|$, as required.
Also, there are $|w| \leq |u| \leq |u_hu_{h+1}|$ candidates for
the parametric word $\alpha w'' w^I$.
A similar analysis holds for $v''v^j\beta$,
which yields the claim.
\end{proof}

\section{Restricting the superset of solutions}\label{sec:restricting-the-superset-of-solutions}
By Lemma~\ref{lem:solutions_superset}, we know the form of possible solutions,
and by Lemma~\ref{lem:testing_a_single_solution}
we can test a single candidate solution in $\Ocomp(m)$ time.
In particular, all solutions from the set $S$ in Lemma~\ref{lem:solutions_superset}
can~be~tested in $\Ocomp(n^2m)$ time, as desired.
The other solutions are instances of parametric words the form $\alpha u^Iv^J \beta$
for well-defined $\alpha, u, v, \beta$.
The next step is to bound, for fixed $\alpha, u, v, \beta$,
the set of values $(i, j)$ such that $\alpha u^I v^J \beta(i,j)$ could be a solution;
this is the main result of the paper.

\subparagraph{Idea}
Suppose we want to find out which words of the form $u^i$ are a solution of~\eqref{eq:main}.
We~substitute $u^I$ to the equation and treat its left-hand side as a parametric word $w$
depending on $I$.
If substituting $I = i$ leads to a trivial word,
then it is known that some $u$-power cancels within the neighboring $u$-powers
(actually, a variant of this fact was used to characterize the superset of solutions~\cite{lyndon_one_var_FG,appel_one_var_FG,onevariablefreegroup},
and it is attributed already to Nielsen~\cite{Nielsen}),
more formally:

\begin{lemma}[{cf.~\cite[Lemma~3]{onevariablefreegroup}}]
\label{lem:pseudosolution_simple}
Let $\varepsilon \eqg s_0u_1s_1u_2\cdots s_{k-1}u_k s_k$.
Then there is $u_i$ which reduces within $u_{i-1}s_{i-1}u_{i}s_{i}u_{i+1}$.
\end{lemma}

We want to use Lemma~\ref{lem:pseudosolution_simple}
to claim that some $u$-parametric powers
need to reduce, however, as there can be powers of $u$ as constants,
this makes the analysis problematic:
as an example, consider an equation $a u^I u^\ell \ov a \eqg \varepsilon$,
if $I = i$ is a solution and we set $s_0 = a, u_1 = u^i, s_1 = u^\ell \ov a$
(so that $u_1$ corresponds to $u^I$)
then Lemma~\ref{lem:pseudosolution_simple}
guarantees that
$u^i$ cancels within $u^\ell$, i.e.\ $0 \geq i \ge - \ell$,
even though $I = - \ell$ is the only solution.
This is caused by $u$-powers next to $u$-parametric power,
which makes our application of the Lemma~\ref{lem:pseudosolution_simple} nearly useless.
To fix this, in $a u^i u^\ell \ov a$ we set $s_0 = a, u_1= u^{i + \ell}, s_1 = \ov a$,
and then Lemma~\ref{lem:pseudosolution_simple}
yields $i = -\ell$.
On the level of the parametric word this corresponds to considering
$a u^{I+\ell} \ov a\eqg a u^I u^\ell \ov a$,
i.e.\ we include $u$-powers into the $u$-parametric power next to them.

This is formalized as follows:
A parametric word $w$ is \emph{$u$-reduced} when $u$ is cyclically reduced, primitive and $w$ does not have a subword of the form:
\begin{itemize}
	\item $u^\phi$ for a constant integer expression $\phi$;
	\item $a \ov a$ for some letter $a$ (so $w$ is reduced);
	\item $u^\phi u^\psi$ for some (non-constant) integer expressions $\phi, \psi$;
	\item $u u^\phi, \ov u u^\phi, u^\phi u, u^\phi \ov u$ for some (non-constant) integer expression $\phi$.
\end{itemize}
Note that we do not forbid subwords that are powers of $u$,
we forbid parametric subwords that are in fact subwords, i.e.\ have constant exponents.

Given a parametric word $w$ we can $u$-reduce it to obtain a parametric word that is equal (in the free group) and $u$-reduced
by a simple greedy procedure, i.e.\ replacing a parametric power with a constant integer expression as exponent with a power or
reduction or joining two $u$-powers into one
(the running time for specific applications is analyzed separately at~appropriate places).
When we replace, say $u u^\phi$ with $u^{\phi+1}$, then we say that letters in $u$
were $u$-reduced to $u^{\phi+1}$.
Note that there are different $u$-reduced equivalent parametric words,
so the output of $u$-reduction is not unique, this has no effect on the algorithm, though.

If a parametric word $w$ (with all exponents depending on one variable)
is $u$-reduced then from Lemma~\ref{lem:pseudosolution_simple}
we infer that $w(i) \eqg \varepsilon$
implies $|\phi(i)| \leq 3$ for some parametric power $u^\phi$ in $w$:

\begin{lemma}
\label{lem:parametric_solution_almost_0}
Let $w = w_0 u^{\phi_1} w_1\cdots u^{\phi_k} w_k$ be a $u$-reduced parametric word,
where $w_0, \ldots, w_k$ are words
and $\phi_1,\ldots, \phi_k$ are integer expressions,
all depending on exactly one and same variable.
If $w(i) \eqg \varepsilon$ then there is $\phi_\ell$ such that $|\phi_\ell(i)| \leq 3$.
In particular, $w(i) \eqg \varepsilon$ for each $i$ if and only if $w = \varepsilon$.
\end{lemma}

\begin{proof}
We use Lemma~\ref{lem:pseudosolution} for a factorization with $s_j = w_j$ and $u_j = u^{\phi_j(i)}$.
Then for some $\ell$ we have that $u^{\phi_\ell(i)}$ reduces
within
$u^{\phi_{\ell-1}(i)}w_{\ell-1}u^{\phi_\ell(i)}w_{\ell}u^{\phi_{\ell+1}(i)}$.
Suppose that $|\phi_\ell(i)| \geq 4$, say $\phi_\ell(i) > 0$, the other case is shown in the same way.
Consider $u^{\phi_{\ell-1}(i)}w_{\ell-1}u^2 u^{\phi_{\ell}(i)-4} u^2 w_{\ell}u^{\phi_{\ell+1}(i)}$.
Then by Lemma~\ref{lem:no_long_self_reduction_after_reduction}
the reductions in $u^{\phi_{\ell-1}(i)}w_{\ell-1}u^2$ and $u^2 w_{\ell}u^{\phi_{\ell+1}(i)}$
are both  of length less than $2|u|$, thus not the whole $u^2 u^{w_{\ell}-4} u^2$ is reduced, contradiction.

For the last claim, if $w$ contains $u$-parametric powers,
then clearly there is a finite set of $i$s such that $w(i) \eqg \varepsilon$.
If it does not, then as it is $u$-reduced, it is also reduced,
and so $w \eqg \varepsilon$ implies $w = \varepsilon$.
\end{proof}

As $\phi_\ell$ in Lemma~\ref{lem:parametric_solution_almost_0}
is a non-constant integer expression
then there are at most $7$ values of $i$ such that
$|\phi_\ell(i)| \leq 3$.
Hence it is enough to find appropriate $i$ values.
Clearly, there are at most $m$ integer expressions in $w$
(as this is the number of variables).
We can give better estimations, though:
if the expression is not of the form $kI$ then it ``used'' at least $|u|$ letters from the equation.
So there are $n/|u|$ different expressions and the ones of the form $kI$;
as $|ki| \leq 3$ implies $|i| \leq 3$,
there are $7(1 + n/|u|)$ candidates for $i$ in total.
Lastly, when the solution depends on two variables,
it can be shown that all obtained parametric powers have coefficient $\pm 1$,
which allow even better estimations:
a parametric power $I + c$ uses at~least $c|u|$ letters from the equation
and so it can be shown that at most $\Ocomp(\sqrt{n/{|u|}})$ different integer expressions can be formed in such a case.

The actual solution is of the form $\alpha u^I v^J \beta$.
Firstly, the presence of $\alpha, \beta$ make estimations harder, as their letters can also be used in the $u$- and $v$-reductions.
Secondly, there are two parameters,
which makes a simple usage of Lemma~\ref{lem:parametric_solution_almost_0} impossible.
However, if $w(i, j) \eqg \varepsilon$ then $w(I,j) \eqg \varepsilon$ depends on one variable,
so Lemma~\ref{lem:parametric_solution_almost_0} is applicable to it.
The analysis yields that we can restrict the possible value of $i$ or $j$ or $(i,j)$;
note that this is non-obvious, as~there are infinitely many $w(I,j)$s.
A similar analysis can be made for $w(i, J)$,
and combining those two yields a set of pairs to be tested
as well as $\Ocomp(1)$ individual $i$s and $j$s that should be tested separately.
But for a fixed $i$ ($j$) we can substitute it to the equation
and use Lemma~\ref{lem:parametric_solution_almost_0} for $J$ ($I$, respectively).

\subsection{Restricting the set of $(i,j)$}
Fix some $0 \leq i_0 \leq m-1$ and the corresponding $u_{i_0}, u_{i_0+1}$ in the equation \eqref{eq:main}.
Using Lemma~\ref{lem:solutions_superset}
we construct a parametric word
$\alpha u^Iv^J \beta$, with $\alpha, u, v, \beta$
depending on $u_{i_0}, u_{i_0+1}$ as~well as exponents
$p_{i_0}, p_{i_0+1},p_{i_0+2}$.
We substitute $X = \alpha u^I v^J \beta$ to the equation~\eqref{eq:main},
obtaining a~parametric word on the left-hand side.
We are to find values $(i, j) \in \mathbb Z^2$
for which the value of the obtained parametric word is equivalent to $\varepsilon$,
thus we call such an $(i, j)$ a~solution.
We want to find a~suitable set of pairs $(i,j)$ and test each one individually,
using Lemma~\ref{lem:testing_a_single_solution}.

The analysis depends on the relation between $u$ and $v$:
i.e.\ whether $u \in \{v ,\ov v\}$,
$u \not \shift v$ or~$u \shift v$.
We analyze particulate cases in
Sections~\ref{sec:u-not-sim-v}--\ref{sec:u-shift-v}.
The idea is the same in each case, but technical details differ.

\subsubsection{$u \not \shift v$}\label{sec:u-not-sim-v}
Due to symmetry, we consider the case when $|v| \geq |u|$, note that it could be that $|u| = |v|$.
We~rotate the left-hand side of the equation
so that it begins and ends with a parametric power:
we rotate $\alpha u^I v^J \beta w = \varepsilon$
to $v^J \beta w \alpha u^I = \varepsilon$
or $\ov \beta \ov v^J \ov u ^I \ov \alpha w = \varepsilon$ to
$\ov u ^I \ov \alpha w \ov \beta \ov v^J = \varepsilon$,
depending on the form of the equation.
The equation after the rotation is equisatisfiable to the previous one.

We call each parametric word beginning with $v^J$ or $\ov u^I$ and ending with $u^I$ or $\ov v^J$
and no parametric power inside a \emph{fragment}.
The parametric word after the rotation is a concatenation of $m$ fragments.
We use the name $h$-th fragment to refer to the one corresponding to $u_h$
(so~$h$-th from the left);
let $f_h$ denote the word that is left from $h$-th fragment after removing
the leading and ending parametric power;
note that $f_h$ is of one of the forms
$\beta u_h \alpha $, $\beta u_h \ov \beta$, $\ov \alpha u_h \alpha $, $\ov \alpha u_h \ov \beta$.
For $u^I$ we call the preceding $\alpha$ the associated word,
the same name is used to $\beta$ succeeding $v^J$, $\ov \alpha$ succeeding $\ov u ^I$ and $\ov \beta$ preceding $\ov v^J$.
To simplify, we will call it a word associated with the parametric power.

We now preprocess the equation, by replacing the left-hand side with an equivalent parametric word
(i.e.\ equal according to $\eqg$).
As a first step, we replace each $f_h$ with $\nf(f_h)$.
Next, observe that if $w$ is the power of $u$
then $\ov u^I w u^I \eqg w$ and similarly $v^J w' \ov v ^J \eqg w'$ for $w'$ being a power of $v$.
In the second step we check each fragment separately, and if possible, replace it as described above.
For fragments that remained unchanged in the second step, we~use previous names,
i.e.\ if $h$-th fragment $v^J \nf(f_h) u^I$ was not replaced then
we still write it as $v^J \nf(f_h) u^I$ and call it $h$-th fragment.
A \emph{trivial fragment} is a maximal subword obtained as concatenations of words obtained due to replacements in the second step.

\begin{lemma}
\label{lem:preprocessing}
The preprocessing can be performed in $\Ocomp(m)$ time.
Afterwards the parametric word is $\Ocomp(m)$-represented
and it is a concatenation of fragments and trivial fragments.

Each trivial fragment is obtained by replacing some $h$-th, $h+1$-st, \ldots , $h+k$-th fragments by $\nf(f_h \cdots f_{k+h})$
moreover $|\nf(f_h \cdots f_{k+h})| \leq \sum_{i=h}^{h+k} |u_i|$;
if $k>0$ then such a trivial factor is not a power of $u$ nor $v$.
\end{lemma}
\begin{proof}
Initially, we represent each $f_h$ as $\alpha, u_h, \beta$ (or similarly),
which are all $\Ocomp(1)$-represented, so the parametric word before the preprocessing
is $\Ocomp(m)$-represented.
Computing the normal form can be done in total $\Ocomp(m)$ time,
see Lemma~\ref{lem:solution_testing}.
Hence testing, whether a fragment becomes trivial
(so is a power of $u$ or $v$) can be done in $\Ocomp(m)$ time
for all fragments.
When we replace fragment, say $\ov u^I f_h u^I$,
with a power of $u$ then from Lemma~\ref{lem:u_prefix_beta_w_ovbeta}, third item,
we get that $|u| \leq |\nf(f_h)| \leq |u_h|$.
Clearly the representation size does not increase.

If several consecutive fragments are replaced with trivial fragments,
then we compute the normal form for each concatenation forming a trivial fragment.
As the total representation size is $\Ocomp(m)$,
by Lemma~\ref{lem:solution_testing} we do it in total $\Ocomp(m)$ time.
The representation size does not increase.

Concerning the trivial fragment size:
\begin{equation*}
|\nf(f_h \cdots f_{k+h})| \leq \sum_{i=h}^{h+k} |\nf(f_i)| \leq \sum_{i=h}^{h+k} |u_i| \enspace .
\end{equation*}

Lastly, we want to show that if the trivial fragment
was obtained from more than one fragment then
it is not a power of $u$ nor $v$.
By symmetry, suppose that some trivial fragment is a power of $v$.
Hence there are some non-zero $k_1,\ldots k_{\ell}$
such that
\begin{equation*}
u^{k_1}v^{k_2}\cdots u^{k_{\ell-1}} \eqg v^{k_\ell}
\end{equation*}
(note that we can assume that the first and last power on the left-hand side is not a power of $v$, as in such a case we can move it to the right-hand side).
This is equivalent to
\begin{equation*}
u^{k_1}v^{k_2}\cdots u^{k_{\ell-1}}v^{-k_\ell} \eqg \varepsilon
\end{equation*}
From Lemma~\ref{lem:concatenation_of_powers_is_not_a_power} we conclude
that $u, v$ are powers of the same word, which contradicts the assumption on $u, v$ (that they are primitive, different and not inverses of each other).
\end{proof}

We now perform the $u$-reduction
(note that the $v^J$ is not touched)
and afterwards the~$v$-reduction.
Let the obtained equation be of the form
\begin{equation}
\label{eq:parametric_equation}
W \eqg \varepsilon \enspace ,
\end{equation}
where $W$ is a parametric word.
In the following,
we are looking for $(i,j)$s such that
$w(i, j) \eqg \varepsilon$, and so we simply call $(i,j)$ a solution (of~\eqref{eq:parametric_equation}).

\begin{lemma}
\label{lem:different_u_v_reduction}
For $u \not \shift v$ we can perform the $u$-reduction and  $v$-reduction after the preprocessing in $\Ocomp(m)$ time;
the obtained parametric word is $u$-reduced.
No two parametric powers are replaced by one during the $u$-reduction and $v$-reduction,
in particular,
for a given parametric power $u^\phi$ ($v^\psi$) in~\eqref{eq:parametric_equation}
the $\phi$ ($\psi$) has a coefficient of the variable equal to $\pm 1$
and the only letters that are $u$-reduced ($v$-reduced) to this
power come either from the associated fragment of $u^I$ or $\ov u^I$
($v^J$ or $\ov v^J$)
and the letters from the adjacent trivial fragment (assuming that there is an adjacent trivial fragment).
\end{lemma}

Note that the claim that no two parametric powers are replaced by one
is not obvious---in principle, it could be that after the preprocessing
a trivial fragment is a power of $u$ (or $v$) and then it is wholly $u$-reduced,
which can lead to two adjacent parametric powers of $u$,
which are then replaced with one.
However, this cannot happen, as such a trivial fragment
is~of~the~form $u^{k_1}v^{k_2}\cdots$ for some $0 < |k_1|, |k_2|, \ldots$
and such a word cannot be a power of $u$ nor $v$ when $u \not \shift v$,
as the subgroup generated by $u, v$ is a free group.

\begin{proof}
Consider a $u$ parametric power before the $u$-reduction.
To one side it has its fragment,
as the preprocessing finished, either the word in this fragment is not a power $u$
or it is a power and then the parametric power at the other end 
is a $v$-parametric power.
To the other side it either has a $v$-parametric power or a trivial fragment,
but the trivial fragment is not w power of $u$,
see Lemma~\ref{lem:preprocessing}.
So we will not join two $u$-parametric powers, as desired.

Hence to compute the $u$-reduction it is enough to consider each word in a fragment separately and compute its
$u$-power prefix/suffix.
As the whole parametric word is $\Ocomp(m)$ represented,
this takes $\Ocomp(m)$ time.

The same applies to $v$-reduction.
Note that the $v$-reduction does not remove any $v$-parametric powers, so afterwards the parametric word is still $u$-reduced.
\end{proof}

We now estimate, how many different $u$-parametric expressions are there after the reductions.
When we want to distinguish between occurrences of parametric powers with the same exponent
(say, two occurrences of $u^{I+1}$ counted separately) then we write about parametric powers
and when we want to treat it as one, then we talk about exponents.
We~provide two estimations, one focuses on parametric powers and the other on exponents.

\begin{lemma}
\label{lem:how_many_different_powers_simple}
There is a set $S$ of $\Ocomp(1)$ size of integer expressions
such that there are $\Ocomp(n/|u|)$ occurrences of $u$-parametric powers
in~$W$ from~\eqref{eq:parametric_equation} whose exponents are not in $S$
and $\Ocomp(n/|v|)$ occurrences of $v$-parametric powers whose exponents are not in $S$.
The set $S$ can be computed and the parametric powers identified in $\Ocomp(m + n/|u|)$ time.
\end{lemma}

The Lemma considers, whether the parametric power used some letters
from the trivial fragment or its associated fragment had $u_h$ of length
at least $|u|$.
If so, then it is in the $\Ocomp(n/|u|)$ parametric powers,
as one such power uses at least $|u|$ letters of the input equation
(this requires some argument for the trivial fragments)
and otherwise is can be shown that there are only $\Ocomp(1)$ possible exponents: say, when we consider the longest suffix of $\nf(\beta u_h \alpha)$ that is a $u$-power,
where $|u_h| < |u|$, then there is a constant number of possibilities how this suffix is formed
(fully within $\alpha$, within $\nf(u_h \alpha)$, uses some letters of $\beta$)
and in each case the fact that $|u_h| < |u|$
means that there are only $\Ocomp(1)$ different $u_h$s that can be used;
note that we need the primitivity of $u$ here.
Concerning the algorithm, note that we can distinguish between these two cases
during the preprocessing and mark the appropriate powers.

\begin{proof}
We consider the (more difficult) case of $v^\psi$, the one for $u^\phi$ is done similarly.

Consider, how $v^\psi$ was created.
If during the $u$-reduction it included some letters
from a trivial fragment, then there are at most $n/|v|$
such parametric powers $v^\psi$:
it $v$-reduced at least $|v|$ letters from a trivial fragment and the sum of lengths of all trivial fragments is at most $n$,
see~Lemma~\ref{lem:preprocessing}.
Similarly, if its associated fragment, say $\beta u_k \alpha$, the $|u_k|\geq |v|$,
then there are at most $n/|v|$ such parametric powers.
Clearly those parametric powers can be identified during the $u$-reduction.

If none of the above holds then we claim that $\psi$ is of the form $\pm J + c$ and there is a constant number of possible $c$s.
Hence the set $S$ from the statement can be computed,
as it is enough to list the powers outside of those $2n/|u|$
occurrences and take their union.
As the set is of constant size, this can be done in $\Ocomp(1)$ time per exponent.

Consider, how $\psi$ was created, note that it does not include letters from trivial fragments.
So $\psi$ is created from $\beta u_k \alpha$ or $\ov \alpha u_k \ov \beta$ or $\beta u_k \ov \beta$:
by computing the normal form removing the $u$-power prefix (second case)
or suffix (third case) and 
then computing the $v$-suffix or prefix (first and third, second and third).
By case assumption $|u_k| < |v|$.
The difference between the $v$-power prefix of $\beta u_k \alpha$
and $v$-power prefix of $\beta u_k \alpha$ after the removal of the $u$-power suffix has length at most $2|v|$:
the part of this difference that is outside the $u$-power suffix has length less than $|v|$,
as it is not included in the $v$-power prefix.
The part inside the $u$-power prefix is both a $u$-run and $v$-run,
so by Lemma~\ref{lem:different_runs_overlap} has length less
than $|u| + |v| \leq 2|v|$.
So the total length of this difference is less than $3|v|$ so at most $2|v|$.
This increases the number of possible expressions at most $5$ times
($-2|v|, - |v|, 0, |v|, 2|v|$).

Consider first the case of $\beta u_k \alpha$.
Then Lemma~\ref{lem:u_prefix_beta_w_ovbeta}, second point,
yields that there are only $\Ocomp(1)$ possible $v$-prefix lengths of $\nf(\beta u_k \alpha)$,
over all $u_k$ such that $|u_k| < |v|$.

In case of $\beta u_k \ov \beta$ from Lemma~\ref{lem:u_prefix_beta_w_ovbeta}, first point,
we have that the $v$-power prefix is of length at most
$2|v| + |u_k|<3|v|$, so again we are done.

The argument for $u$-powers is similar, it just omits the distinction between the $u$-power prefix
and the actual $u^c$ in the $u$-reduction, as they are the same.
\end{proof}

The next lemma provides a better estimation for the number of different exponents,
it essentially uses the fact that all exponents have coefficients at variables $\pm 1$:
as there are only two possible coefficients, we can focus on the constants.
Now, to have a constant $|c|$, we have to use a power $u^c$ from $W$
and to have $k$ \emph{different} constants
one has to use $k$ different powers and so from Lemma~\ref{lem:different_powers_in_a word}
we conclude that $k = \Ocomp(|W|/|u|)$.
In general, $W$ can be of quadratic length, as we introduce $m$ copies of $\alpha$ and $\beta$ into it;
the resulting bound is too weak for our purposes.
To improve the bound, we employ Lemma~\ref{lem:sum_of_powers}:
consider that when the $u$-power suffix of,
say, $\beta u_h \alpha$, is $u^k$ then by Lemma~\ref{lem:sum_of_powers}
there are $k_\alpha, k_u, k_\beta$ such that $|k - k_\alpha - k_u - k_\beta| \leq 2$
and $u^{k_u}$, $u^{k_\beta}$ are maximal $u$-powers in $u_h, \beta$
and $u^{k_\alpha}$ is the $u$-power suffix of $\alpha$.
Using Lemma~\ref{lem:different_powers_in_a word},
this yields that there are $\Ocomp(\sqrt{n/|u|})$
different possible values of $k_u$ (over all $u_h$),
$\Ocomp(\sqrt{|\beta|/|u|}) = \Ocomp(\sqrt{|u_{i_0}u_{i_0+1}|/|u|})$ of $k_\beta$ and $k_\alpha$ is fixed,
so there are at~most $\Ocomp(\sqrt{n/|u|} \cdot \sqrt{|u_{i_0}u_{i_0+1}|/|u|}) = \Ocomp(\sqrt{n|u_{i_0}u_{i_0+1}|}/|u|) $
possible values of $k$.

The actual argument is more involved, as it is also possible
that the $u$-parametric power includes letters from the trivial fragments,
which requires some extra arguments, nevertheless the general approach is similar.

\begin{lemma}
\label{lem:how_many_different_powers}
After the $u$-reduction and $v$-reduction there are $\Ocomp(\sqrt{n|u_{i_0}u_{i_0+1}|}/|u|)$
different integer expressions as exponents in parametric powers of $u$
and $\Ocomp(\sqrt{n|u_{i_0}u_{i_0+1}|}/|v|)$
of $v$ in~the equation.
The (sorted) lists of such expressions can be computed in $\Ocomp(m + n /|u|)$ and~$\Ocomp(m + n /|v|)$ time, respectively.
\end{lemma}
\begin{proof}
We first consider the expression powers of $u$.

Consider first a fragment $v^J \beta u_h \alpha u^I$
from which an expression power $u^\phi$ is formed during the $u$-reduction.
Suppose first that to the right of this fragment there is another fragment.
Thus $u^\phi$ was formed only from $\beta u_h \alpha$.

Consider the word $\beta u_h' \alpha'$ obtained after the reduction of  $u_h \alpha$,
i.e.\ $u_h' \alpha' = \nf(u_h \alpha)$ and $u_h'$ was obtained from $u_h$ while $\alpha'$ from $\alpha$.
The $u$-reduction will result in $I + c$,
where $c$ is the $u$-power suffix of $\nf(\beta u_h' \alpha')$.
By Lemma~\ref{lem:sum_of_powers} we obtain that
$|c - (c_\alpha + c_u + c_\beta)|\leq 2$,
where $c_{\alpha}, c_u, c_{\beta}$ are lengths of some maximal powers in
$\alpha'$, $u_h'$ and $\beta$,
moreover, $u^{c_\alpha}$ is a suffix of $\alpha$ or $c_\alpha = 0$.
As $\alpha$ does not have a suffix $u$ nor $\ov u$,
we conclude that $c_\alpha = 0$.
As $|\beta| \leq |u_{i_0}u_{i_0+1}|$ by Lemma~\ref{lem:solutions_superset},
from Lemma~\ref{lem:different_powers_in_a word} there are at most 
$\sqrt{5|u_{i_0}u_{i_0+1}|/|u|}$ possible values of $c_\beta$.
Consider the possible values of $c_u$, over all $u_h'$.
Clearly, the sum of lengths of all $u_h'$ is at most $n$.
By Lemma~\ref{lem:different_powers_in_a word}
there are at most $\sqrt{5n/|u|}$ choices for $c_u$,
for all possible $u_h'$.
Taking into account the $5$ possible choices for the difference between
$c$ and $c_u + c_\alpha$
we get that there are at most
\begin{equation*}
5 \cdot \sqrt{5|u_{i_0}u_{i_0+1}|/|u|} \cdot \sqrt{5n/|u|}
= 
25 \frac{\sqrt{n|u_{i_0}u_{i_0+1}|}}{|u|}
\end{equation*}
possible values of $c$ and so this number of expressions $\phi$.

So suppose now that to the right of $v^J \beta u_h \alpha u^I$ there is a trivial fragment,
by Lemma~\ref{lem:preprocessing} it is equal to $\nf(f_{h+1}\cdots f_{h+k})$ for some $k >0$
and it is of length at most $|u_{h+1}\cdots u_{h+k}|$.
Observe after the $u$-reduction the expression is equal to
\begin{equation*}
I + c+ c_{\nf}' \enspace ,
\end{equation*}
where $c$ is the $u$-power suffix of $\nf(\beta u'_h \alpha)$
and $c_{\nf}'$ is almost the exponent of $u$-power prefix of $\nf(f_{h+1}\cdots f_{h+k})$:
the $u$-power prefix and $u$-power suffix of $\nf(\beta u'_h \alpha)$ could overlap,
but due to Lemma~\ref{lem:u-power_pref_and_suff_overlap} they overlap by less than $|u|$ letters
(as otherwise $\nf(\beta u'_h \alpha)$ would be a power of $u$,
which does not hold by Lemma~\ref{lem:preprocessing}).
We arbitrarily choose which $u$-parametric word to extend,
so $|c_{\nf}' - c_{\nf}|\leq 1$, where $c_{\nf}$ is the exponent of the $u$-power prefix of $\nf(f_{h+1}\cdots f_{h+k})$.
Note that $u^{c + c_{\nf}}$ is a maximal power in 
$\nf(\beta u'_h \alpha' f_{h+1}\cdots f_{h+k})$.
However, there is a slight problem with estimating the possibilities for $c + c_{\nf}$,
as now we cannot claim that the power coming from $\alpha'$ is trivial.
To deal with this technicality, we make a slight case distinction,
depending on $u'_h \alpha'$.

Consider the $u$-power suffix of $u'_h \alpha'$.
If it is $\varepsilon$ then by Lemma~\ref{lem:sum_of_powers}
if $c$ is the $u$-power suffix of $\beta u_h' \alpha'$
then $|c - c_\beta| \leq 1$ for some maximal power $u^{c_\beta}$ of $\beta$,
and there are at most $\sqrt{5|u_{i_0}u_{i_0+1}|/|u|}$
possible values of $c_\beta$.
Concerning $c_{\nf}'$,
there are at most  $\sqrt{5n/|u|}$ choices for $c_{\nf}$,
the $u$-power prefix of $\nf(f_{h+1}\cdots f_{h+k})$,
over all trivial fragments in total
(as the length of all trivial fragments is at most $n$).
Now
\begin{equation*}
c + c_{\nf}' = c_\beta + c_{\nf} + (c - c_\beta) + (c_{\nf}' - c_{\nf}) \enspace .
\end{equation*}
As $|(c - c_\beta) + (c_{\nf}' - c_{\nf})| \leq 2$,
the total number of different integer expressions is at most
\begin{equation}
\label{eq:different_exponenets_1}
\underbrace{5}_{(c - c_\beta) + (c_{\nf}' - c_{\nf})}
	\cdot
\underbrace{\sqrt{5|u_{i_0}u_{i_0+1}|/|u|}}_{c_{\beta}}
	\cdot
\underbrace{\sqrt{5n/|u|}}_{c_{\nf}} = 
25 \frac{\sqrt{|u_{i_0}u_{i_0+1}|n}}{|u|} \enspace .
\end{equation}

So consider the other case, when the $u$-power suffix $u^{c_{u,\alpha}}$ of $u_h' \alpha'$
is not $\varepsilon$,
we consider the case when $c_{u,\alpha} > 0$, the other one is analogous.
Note, that as $\alpha'$ does not end with $u$ nor $\ov u$,
the fact that $c_\alpha \neq 0$ implies that $|\alpha '|<|u|$.
Let $u_h'' = \nf(u_h'  \alpha' \ov u)$;
as $c_{u,\alpha} > 0$, this whole $\ov u$ reduces, so also the whole $\alpha '$
reduces as well, as $|\alpha'| < |u|$.
Thus $|u_h''| < |u_h'| \leq |u_h|$.
Let $c_{u,\alpha}'$ be the length of $u$-power suffix of $u_h''$,
then $|c_{u,\alpha} - c_{u,\alpha}'| = 1$.
Recall that $c$ is the $u$-power suffix of $\nf(\beta u_h' \alpha')$,
by Lemma~\ref{lem:sum_of_powers}
\begin{equation*}
|c - c_{u,\alpha} - c_\beta| \leq 2
\end{equation*}
for some maximal power $u^{c_\beta}$ in $\beta$.
Switching from $c_{u,\alpha}$ to $c_{u,\alpha}'$ we get
(note that the bound holds also for $c_{u, \alpha} < 0$):
\begin{equation*}
|c - c_{u,\alpha}' - c_\beta| \leq 3 \enspace.
\end{equation*}

As in the previous case, let $c_{\nf}$ denotes the $u$-power prefix of $\nf(f_{h+1}\cdots f_{h+k})$, then $|c_{\nf} - c_{\nf}'| \leq 1$.
Observe that $u^{c_{\nf} + c_{u,\alpha}'}$ is a maximal $u$-power in
$\nf(u_h''f_{h+1}\cdots f_{h+k})$.
Let us estimate, how many possible values of $c_{\nf} + c_{u,\alpha}'$
are there.
Observe that
\begin{equation*}
|\nf(u_h''f_{h+1}\cdots f_{h+k})| \leq |u_h| + \sum_{\ell = 1}^k |u_{h+\ell}| = \sum_{\ell = 0}^k |u_{h+\ell}| \enspace .
\end{equation*}
Thus, when summing over all such fragments $u_h$ and neighboring trivial fragment,
the sum of all lengths is at most $n$.
Thus there are at most
$
\sqrt{5n/|u|}
$
possible values of $c_{\nf} + c_{u,\alpha}'$.
Also, there are 
$
\sqrt{5|u_{i_0}u_{i_0+1}|/|u|}
$
possible values of $c_\beta$.
So there are at most
\begin{equation*}
5 \frac{\sqrt{n|u_{i_0}u_{i_0+1}|}}{|u|}
\end{equation*}
possible values of $c_{\nf} + c_{u,\alpha'} + c$.
Recalling that $|c-c_{u, \alpha}'-c_\beta| \leq 3$ and $|c_{\nf} - c_{\nf}'|\leq 1$
we obtain that there are at most

\begin{equation}
\label{eq:different_exponenets_2}
\underbrace{9}_{|c-c_h-c_\beta| + |c_{\nf} - c_{\nf}'|}
	\cdot
5 \frac{\sqrt{n|u_{i_0}u_{i_0+1}|}}{|u|} = 45\frac{\sqrt{|u_{i_0}u_{i_0+1}|n}}{|u|}
\end{equation}
different possible values of $c + c_{\nf}'$, and so also different possible integer expressions.
Note that the first case, when to the other side of $u^I$ there is a $v$-parametric power
yields the same exponent when we choose a trivial power from a trivial fragment.
And this is accounted for in~\eqref{eq:different_exponenets_1}--\eqref{eq:different_exponenets_2}. 
So it is enough to sum~\eqref{eq:different_exponenets_1}--\eqref{eq:different_exponenets_2}
yielding that there are at most
\begin{equation*}
70 \frac{\sqrt{|u_{i_0}u_{i_0+1}|n}}{|u|}
\end{equation*}
different expression powers for this type of fragment.

The analysis for the fragment $\ov u ^I \ov \alpha u_h \ov \beta \ov v ^J$
is the same, and so are the bounds.

For a fragment $\ov u ^I \ov \alpha u_h \alpha u ^I$
observe that $\nf(\ov \alpha u_h \alpha)$ is not a power of $u$,
as otherwise we would have replaced the fragment by a~trivial one in the preprocessing.
The analysis is otherwise similar, with the role of $\beta$ taken by $\alpha$ for one of the parametric powers and by $\ov \alpha$ for the other.
Note however, that the $u$-power prefix and $u$-power suffix of $\ov u ^I \ov \alpha u_h \alpha u ^I$
could overlap,
so we need additionally take into account that there is additional difference $1$ between
the true $u$-power prefix/suffix and the one used in the $u$-reduction.
Hence the number of different integer expressions in this case is at most
\begin{equation*}
7 \cdot \sqrt{5|u_{i_0}u_{i_0+1}|/|u|} \cdot \sqrt{5n/|u|}
+
11 \cdot 5 \frac{\sqrt{n|u_{i_0}u_{i_0+1}|}}{|u|}
= 90
\frac{\sqrt{|u_{i_0}u_{i_0+1}|n}}{|u|} \enspace .
\end{equation*}
Summing up:
\begin{equation*}
(70+70+90)\frac{\sqrt{|u_{i_0}u_{i_0+1}|n}}{|u|} = 
230 \frac{\sqrt{|u_{i_0}u_{i_0+1}|n}}{|u|} \enspace .
\end{equation*}

Let us consider the running time bounds.
Note that the parametric powers are explicitly given in a parametric word that is $\Ocomp(m)$-represented,
so we can get the appropriate list in $\Ocomp(m)$ time.
To sort and remove the duplicates it is enough to observe
that the constant size in each of the expression is $\Ocomp(n/|u|)$:
note that $\beta, u_h, \alpha \leq n$ and also each trivial fragment is of length at most $n$,
hence the constant is at most $4 n/|u|$.
Hence we can use counting-sort:
create a bit table to represent each constant between $-8n/|u|$ and $8n/|u|$
(separately for integer expression with $-I$ and with $I$)
and mark each length in the set and then gather the set of obtained constants,
in time $\Ocomp(n/|u|)$.

If the $v$-reduction were done as first,
then the analysis for $v$ would be the same,
with the only difference being the division by $|v|$ and not $|u|$.
However, in case of $v^J\beta u_h \alpha u^I$ ($\ov u^I \ov \alpha u_h \ov \beta \ov v ^J$)
it could be that $v$-power prefix ($v$-power suffix, respectively) overlaps the 
$u$-power suffix ($u$-power prefix, respectively) that was already $u$-reduced;
the same can happen for trivial fragment.
It was already described in Lemma~\ref{lem:how_many_different_powers_simple}
that the difference between the $v$-power prefix and the word used for $v$-reduction has length at most $2|v|$ (from each side),
which increases the number of possible expressions by a factor of maximally $9$.
The running time bound is shown as for $u$.
\end{proof}

We can use Lemma~\ref{lem:pseudosolution_simple}
together with bounds on the number of different exponents in parametric powers from 
Lemma~\ref{lem:how_many_different_powers} to limit the possible candidates
$(i,j)$ for a solution.
However, these bounds are either on $i$ or on $j$.
And as soon as we fix, say, $J = j$ and substitute it to $W$,
the obtained parametric word $W(I,j)$ (or $W(i,J)$) is more complex
than $W$, in particular, we do not have the bounds of Lemma~\ref{lem:how_many_different_powers} for it,
so the set of possible candidates for $i$ for a given $W(I,j)$ is linear,
which is too much for the desired running time.

Instead, we analyze (as a mental experiment) $W(I, j)$:
Fix $j \in \mathbb Z$ such that  $W(i,j) \eqg \varepsilon$ for some $i$.
Compute $W(I, j)$, $u$-reduce it, call the resulting parametric word $W_{J = j}$.
If~$W_{J = j} = \varepsilon$, then clearly for each $i$ the $(i,j)$ is a solution of~\eqref{eq:parametric_equation} (and vice-versa, see~Lemma~\ref{lem:parametric_solution_almost_0}). 
It can be shown that in this case for some $v^\psi$ in $W_{J = j}$ it holds that $|\psi(j)| < 6$:
at least some two $u$-parametric powers in $W$ should be merged in $W_{J=j}$,
in $W$ they are separated by a~$v$-parametric power, say $v^\psi$.
All letters of $v^{\psi(j)}$ are $u$-reduced, then standard arguments using periodicity show that $|\psi(j)| < 6$
so we can compute all candidates for such $j$s
and test for each one whether indeed $W_{J = j} = \varepsilon$,
this is formally stated in Lemma~\ref{lem:dissapearing_v_powers}.

If $W_{J=j}$ depends on $I$ then from Lemma~\ref{lem:parametric_solution_almost_0} for
some of the (new) $u$-parametric powers $u^\phi$ it holds that $|\phi(i)| < 6$.
Consider, how this $\phi$ was created.
It could be that it is (almost) unaffected by the second $u$-reduction
and so it is (almost) one of the $u$-parametric powers in $W$,
see Lemma~\ref{lem:how_many_different_powers_second_reduction}
for precise formulation and sketch of proof,
in which case we can use Lemma~\ref{lem:how_many_different_powers}.
Intuitively, $u^\phi$ is affected if the whole two parametric powers in $W$ were used to create $u^\phi$.
Then it can be shown that some $v$-parametric power $v^\psi$ from $W$
turned into $v$-power $v^{\psi(j)}$ satisfies $|\psi(j)| < 6$
and is $u$-reduced to $u^\phi$,
the argument is as before, when $W_{J = j} \eqg \varepsilon$.
Moreover, this occurrence of $v^\psi$ also determines $u^\phi$;
hence the choice of $\psi$ determines $\Ocomp(1)$ candidates for $j$,
uniquely identifies $\phi$ and $i$ satisfies $|\phi(i)| < 6$,
i.e.~there are $\Ocomp(1)$ candidates for $(i,j)$.
Then Lemma~\ref{lem:how_many_different_powers_simple} is applied to this $v^\psi$:
if it is one of $n/|v|$ occurrences of $v$-parametric powers
then we get $\Ocomp(1)$ candidates for $(i,j)$ (for this $\psi$),
so~$\Ocomp(n/|v|)$ in total, over all choices of such $\psi$.
Otherwise, $\psi$ it is one of $\Ocomp(1)$ integer expressions (Lemma~\ref{lem:how_many_different_powers_simple})
and so $j$ is from $\Ocomp(1)$-size set and we can compute and consider $W_{J = j}$ for each one of them separately.

A similar analysis applies also to $i \in \mathbb Z$ substituted for $I$.
The results are formalized in~the~Lemma~\ref{lem:indices_sets_description} below,
its proof is spread across a couple of Lemmata.

\begin{lemma}
\label{lem:indices_sets_description}
Given equation~\eqref{eq:parametric_equation} we can compute in $\Ocomp(m n/|u|)$ time sets
$S_I, S_J, S_{\mathbb Z,J} \subseteq \mathbb Z$ and 
$S_{I,J} \subseteq \mathbb Z^2$,
where $|S_I| = \Ocomp({\sqrt{n|u_{i_0}u_{i_0+1}|}}/{|u|})$,
$|S_J| = \Ocomp(1)$, $|S_{\mathbb Z,J}|, |S_{I,J}| = \Ocomp(n/|u|)$,
such that:
if $(i,j)$ is a solution of~\eqref{eq:parametric_equation} then
at least one of the following holds:
\begin{itemize}
\item $i \in S_I$ or
\item $j \in S_J$ or
\item $j \in S_{\mathbb Z,J}$ and for each $i'$ the $(i',j)$ is a solution or
\item $(i,j) \in S_{I,J}$.
\end{itemize}
Similarly,
given equation~\eqref{eq:parametric_equation} we can compute
in $\Ocomp(m n/|v|)$ time
sets $S_{I}', S_{J}', S_{I, \mathbb Z}' \subseteq \mathbb Z$ and 
$S_{I,J}' \subseteq \mathbb Z^2$,
where $|S_{I}'| = \Ocomp(1)$,
$|S_{J}'| = \Ocomp({\sqrt{n|u_{i_0}u_{i_0+1}|}}/{|v|})$ 
$|S_{I, \mathbb Z}'|, |S_{I,J}'| = \Ocomp(n/|v|)$
such that at if $(i,j)$ is a solution of~\eqref{eq:parametric_equation}
then least one of the following holds:
\begin{itemize}
\item $i \in S_{I}'$ or;
\item $i \in S_{I,\mathbb Z}'$ and for each $j' \in \mathbb Z$ the $(i,j')$ is a solution or;
\item $j \in S_{J}'$ or;
\item $(i,j) \in S_{I,J}'$.
\end{itemize}
\end{lemma}

As noted above, the main distinction is whether the $u^\phi$ in $W_{J=j}$ was ``affected'' or not during the second $u$-reduction.
Let us formalize this.
Given an occurrence of a parametric power $u^\phi$ in $W_{J = j}$
consider the largest subword $w$ of $W$ such that each letter in $w(I,j)$
is either reduced or $u$-reduced to this $u^\phi$; note that this may depend on the order of reductions, we fix an arbitrary order.
We say that parametric powers in $w$ are \emph{merged} to $u^\phi$.
We extend this notion also to the case when $W_{J = j} = \varepsilon$,
in which case $W = w$ and every parametric power is merged to the same parametric power $u^0$.
A similar notion is defined also for parametric powers of $v$.
Note that a parametric power is not merged to two different parametric powers $u^\phi$
and $u^{\phi'}$.

\begin{lemma}
\label{lem:merged_to_one}
For any parametric power in $W$ there is at most one parametric power in $W_{J=j}$ to which it was merged; the same holds for $W_{I = i}$.
\end{lemma}
\begin{proof}
Let $u^{\phi_1}, u^{\phi_2}$ be two different parametric powers in $W_{J = j}$.
For $u^{\phi_1}$ there is a unique maximal subword $w_1$ of $W(I,j)$ such that each letter in $w$ was either reduced are $u$-reduced to $u^{\phi_1}$
during the creation of $W_{J = j}$;
define $w_2$ similarly for $u^{\phi_2}$. We claim that $w_1, w_3$ are disjoint.
If they overlap, then together they form $w_1'ww_2''$, where $w_1 = w_1'w$ and $w_2 = ww_2''$.
Then $w \eqg \varepsilon$, as it cannot be that a letter is $u$-reduced to both $u^{\phi_1}$ and $u^{\phi_2}$,
so $w_1', w_2''$ are both equivalent to $u$-parametric powers,
hence also $w_1' ww_2''$ is equivalent to a $u$-parametric power,
and so $u^{\phi_1}$ and $u^{\phi_2}$ were $u$-reduced to one $u$-parametric power, contradiction.
\end{proof}

We say that a $u$-parametric power $u^{\phi}$ in $W_{J = j}$ \emph{was affected} by substitution $J = j$
if
\begin{itemize}
\item more than one parametric power was merged to $u^\phi$ \emph{or}
\item for the unique $u$-parametric power $u^{\phi'}$ merged to $u^\phi$
there is a $v$-parametric power $v^{\psi'}$ such that $|\psi'(j)| < 6$ and there is no $u$-parametric power between $u^{\phi'}$ and $v^{\psi'}$.
\end{itemize}
The intuition behind the first condition is that when we merge two $u$-powers then we create a completely new parametric power,
for the second condition, when $|\psi'(j)| < 6$ then $v^{\psi'(j)}$ no longer behaves like $v^{\psi'}$ and can either be wholly merged to a $u$-power
or be canceled by a~trivial fragment, which can also lead to a large modification of the neighbouring $u$-parametric power.
Note that the second condition could be made more restrictive, but the current formulation is good enough for our purposes.

We first investigate the case, when the parametric power was affected by a substitution.

\begin{lemma}
\label{lem:dissapearing_v_powers}
In $\Ocomp(m n/|v|)$ time we can compute and sort
sets $S_J, S_{E,J}$, where $|S_J| = \Ocomp(1)$ and $|S_{E,J}| = \Ocomp(n/|v|)$,
such that for each occurrence of a $u$-parametric power $u^\phi$ in $W_{J = j}$
affected by the substitution $J = j$ either $j \in S_J$ or $(\phi,j) \in S_{E, J}$.

Similarly, in time $\Ocomp(m n/|u|)$ we can compute and sort sets 
$S_I', S_{I,E}$, where $|S_I'| = \Ocomp(1)$ and $|S_{I,E}| = \Ocomp(n/|u|)$,
such that for each occurrence of a $v$-parametric power $v^\psi$ in $W_{I = i}$
affected by the substitution $I = i$ either $i \in S_I'$ or $(i,\psi) \in S_{I,E}$.
\end{lemma}
The sketch of the argument was given above Lemma~\ref{lem:indices_sets_description}.
Concerning the running time, the appropriate exponents are identified during the $u$-reduction and $v$-reduction,
which are performed in given times using the data structure.

\begin{proof}
We first give the proof of the first claim, the second in slightly more involved.

Consider the possible reasons why $u^\phi$ was affected.
Suppose that the second condition holds,
i.e.\ there is some $v^{\psi'}$ such that $|\psi'(j)| < 6$ and there is $u$-parametric power $u^{\phi'}$ merged to $u^\phi$
such that there is no $u$-parametric power between $v^{\psi'}$ and $u^{\phi'}$.
From Lemma~\ref{lem:how_many_different_powers_simple}
there is a constant-size set $E$ such that $\psi' \in E$
or this occurrence of $v^{\psi'}$ is one of $\Ocomp(n/|v|)$ occurrences of $v$-parametric powers in $W$.
In the first case this yields $11 \cdot |E| \in \Ocomp(1)$ many $j'$s such that $|\psi(j')| < 6$ for some $\psi \in E$
and this set can be computed in constant time given $E$;
these numbers are added to $S_J$.
In the other case, this occurrence of $v^{\psi'}$ is one of $\Ocomp(n/|v|)$ chosen occurrences of $v$-parametric powers.
Note that this occurrence of $v$-parametric power plus the value of $j$ plus the direction left/right uniquely defines the parametric power $u^\phi$ in $W_{J = j}$:
this is the unique power such that the first $u$-parametric power directly to the left/right of $v^{\psi'}$ was merged to
(note that $v^{\psi'}$ may be not merged to $u^\phi$).
So there are $\Ocomp(n/|v|)$ choices of $(\phi,j,\psi')$
so there are $\Ocomp(n/|v|)$ choices of $(\phi,j)$, we add them to $S_{E,J}$.
Those sets can be computed in $\Ocomp(m n/|v|)$ time:
after the $v$-reduction ($\Ocomp(m)$ time)
we choose one of the $\Ocomp(n/|v|)$ $v$-parametric powers $v^{\psi'}$
(they can be identified in $\Ocomp(m + n/|v|)$ time by Lemma~\ref{lem:how_many_different_powers_simple})
choose one of the value $j$ such that $|\psi'(j)| < 6$
and perform the second $u$-reduction for $J = j$ ($\Ocomp(m)$ time)
and identify $u^\phi$ to which this $v^\psi$ was merged to (nothing is done if it is not merged);
so we use $\Ocomp(m)$ time for each $\Ocomp(n/|v|)$ candidates.
When all candidates are computed,
we sort and remove the duplicates in $\Ocomp(m + n/|u|)$ time,
in the same way as in Lemma~\ref{lem:how_many_different_powers}.

The other reason why $u^\phi$ was affected is that more than one parametric power was merged to it.
We first show that also some $v$-parametric power was merged to $u^\phi$.
Suppose not, consider all $u$-parametric powers that are merged to $u^\phi$, there are at least two.
During the second $u$-reduction, consider the first moment,
when two $u$-parametric powers (from the chosen ones) are $u$-reduced (if this happens).
Then the word between them is a power of $u$ in $W_{J = j}$, but it was not in $W$.
Hence there was some $v$-parametric power $v$ inside.
Let the word between those two $u$-parametric powers (in $W$) be
$w = s_0 v^{\psi_1} s_1 v^{\psi_2} \cdots v^{\psi_k}s_k$,
then $w(j) \eqg u^\ell$ for some $\ell$.
If only one $u$-parametric power, say $u^{\phi'}$, was merged to $u^\phi$
then consider the maximal word to the left and right of $u^{\phi'}$ in $W$
that does not contain $u$-parametric power
and choose the one that contains a $v$-parametric power merged to $u^\phi$:
it has to exist, as at least two parametric powers were merged.
Suppose that it is to the left, the other case is symmetric.
Then this word is $s_0 v^{\psi_1} s_1 v^{\psi_2} \cdots v^{\psi_k}s_k$,
such that each letter in $v^{\psi_k}s_k)(j)$
was $u$-reduced to $v^{\psi_0}$ or reduced.
Note that this generalizes the previous case,
when the whole $(s_0 v^{\psi_1} s_1 v^{\psi_2} \cdots v^{\psi_k}s_k)(j)$
was $u$-merged to $v^\psi$.

The case when some $|\psi_\ell(j)| < 6$ for $\ell > 0$ was already covered,
so we may assume that $|\psi_\ell(j)| \geq 6$ for $\ell > 0$.
Suppose that some $v^{\psi_\ell}(j)$ for $\ell > 0$ was reduced to at most $2|v|$ letters
and consider the first such $v^{\psi_\ell}(j)$.
Then those reductions are within
$v^{\psi_\ell-1}(j) s_{\ell-1}v^{\psi_\ell}(j)s_{\ell} v^{\psi_\ell+1}(j)$
and from Lemma~\ref{lem:no_long_self_reduction_after_reduction}
the reduction in $v^{\psi_{\ell-1}}(j) s_{\ell-1} v^{\psi_\ell}(j)$ and $v^{\psi_\ell}(j) s_\ell v^{\psi_{\ell+1}}(j)$
have lengths smaller than $2|v|$.
Hence more than $2|v|$ letters remained from $v^{\psi_\ell}(j)$
and they are $u$-reduced to a $u$-parametric power.
But those $2|v|$ letters are both a $v$-run and an $u$-run, contradiction with Lemma~\ref{lem:different_runs_overlap}.
This ends the proof for $S_J, S_{E,J}$.

For the second claim (for $S_{I'}, S_{I,E}$) the analysis when some $u$-parametric power $u^{\phi'}$ satisfies $|\phi'(i)| < 6$ is similar to the one when $|\psi'(j)| < 6$ from the previous main case;
the case when two parametric powers were merged to $v^\psi$ is more involved.
Define $s_0 u^{\phi_1} s_1 u^{\phi_2} \cdots s_{k-1} u^{\phi_{k}} s_k$
similarly as in the previous case, i.e.\ as the maximal string to the left
of the $v$-parametric that is merged to $v^\psi$
and without another $v$-parametric power and such that the whole
$(u^{\phi_{k}} s_k)(i)$ is $v$-reduced to $v^\psi$
(and perhaps some other letters as well).
The case when some of $u^{\phi_1}(i), \ldots, u^{\phi_k}(i)$ 
has length less than $6|u|$ was already considered
above, so we assume that they all have length at least $6|u|$.
In particular, the reductions in $(s_0 u^{\phi_1} s_1 u^{\phi_2} \cdots s_{k-1} u^{\phi_{k}} s_k)(i)$
will reduce at most $2|u|$ letters in each $u^{\phi_1}(i), \ldots, u^{\phi_k}(i)$,
see Lemma~\ref{lem:no_long_self_reduction_after_reduction}.

In the following,
we consider the suffix of $\nf(s_0 u^{\phi_1}(i) s_1 \cdots s_{k-1} u^{\phi_k}(i)s_k)$ that was $v$-reduced to the $v$-power on the right.
We show that its $u$-power suffix is almost $u^{\phi_k}(i)$ on one hand
and a $u$-power suffix of one of $\ov v^2, \ov v, \varepsilon, v, v^2$.
This will give us a finite number of possible values of $i$
(as well as finite number of possible $\psi$).

Consider the parametric power $u^I$ (or $\ov u^I$) from which
$u^{\phi_k}$ was created and the neighboring $v^J$ (or $\ov v ^J$),
it could also be that there is no such parametric power.
There are the following cases:
\begin{enumerate}
\item this power is to the right of $u^{\phi_k}$; \label{case1}
\item $v^J$ is to the left of $u^I$; \label{case2}
\item there is no such $v^J$ nor $\ov v^J$, as this $\ov u^I$ is the first symbol in $W$ from~\eqref{eq:parametric_equation}; \label{case3}
\item this $v^J$ (or $\ov v^J$) was removed during the preprocessing,
i.e.\ it (and some other parametric powers as well as some letters)
were replaced with a trivial factor. \label{case4}
\end{enumerate}

In the first case let the $v$-parametric power to the right of $u^{\phi_k}$
be $v^{\psi_k}$, then is also merged to $v^\psi$.
Moreover, $s_{k} = \varepsilon$, as the $u^I$ and $v^J$,
which were $u$-reduced and $v$-reduced to $u^{\phi_k}$ and $v^{\psi_k}$, respectively, are next to each other.
Let $u^c$ be the $u$-power suffix of $\nf(s_0 u^{\phi_1}(i) s_1 \cdots s_{k-1} u^{\phi_k}(i))$,
we claim that $|\phi(i)-c| < 2$:
observe that the part of $u^{\phi_k}(i)$ that was not reduced has length at least $2|u|$,
so it is part of the $u$-power suffix.
If something was reduced from $u^{\phi_k}(i)$ then the $u$-power suffix is exactly
the $u$-power that is left from $u^{\phi_k}(i)$ after reducing:
we can make the reduction by first
computing $\nf(s_0 u^{\phi_1}(i) s_1 \cdots u^{\phi_{k-1}}(i) s_{k-1})$ and then reducing it
with $u^{\phi_k}(i)$ and less than $2|u|$ letters are reduced in $u^{\phi_k}(i)$,
see Lemma~\ref{lem:no_long_self_reduction_after_reduction}.
Moreover, if we reduce prefix $\gamma$ from $u^{\phi_k}(i)$ then
$\nf(s_0 u^{\phi_1}(i) s_1 \cdots u^{\phi_{k-1}}(i) s_{k-1})$ ends with $\ov \gamma$
and in order to extend $\nf(\ov \gamma u^{\phi_k}(i))$ we need to append letter from
$\gamma$, but this would imply that $\nf(s_0 u^{\phi_1}(i) s_1 \cdots u^{\phi_{k-1}}(i) s_{k-1})$
is not reduced.
When no letters in $u^{\phi_k}(i)$ are reduced on the left
then the $u$-power suffix (of $\nf(s_0 u^{\phi_1}(i) s_1 \cdots u^{\phi_{k-1}}(i) s_{k-1}u^{\phi_k(i)})$)
may be longer than $u^{\phi_k}(i)$.
As $s_{k-1}$ does not end with $u$ nor $\ov u$,
this $u$-power suffix uses less than $|u|$ letters of $s_{k-1}$.
It also use less than $|u|$ letters from $u^{\phi_{k-1}}(i)$:
consider the $u$-runs including the $u$-power suffix of
$\nf(s_0 u^{\phi_1}(i) s_1 \cdots u^{\phi_{k-1}}(i) s_{k-1} u^{\phi_k}(i))$
(which includes at least $2|u|$ letters from $u^{\phi_k}(i)$)
and the one of what is left after reducing of $u^{\phi_{k-1}}(i)$
(which also has length at least $2|u|$).
By Lemma~\ref{lem:different_runs_overlap} the overlap is of length smaller than $|u|$
or they are included in a longer $u$-run.
And they cannot be part of a longer run as $s_{k-1}$ is not a power of $u$.

Now let us look at the $v$-power that was $v$-merged to $v^{\psi}$
and its $u$-power suffix, call it $u^{c'}$.
Whatever is left from $u^{\phi_k}(i)$ after the reductions in
$\nf(s_0 u^{\phi_1}(i) s_1 \cdots u^{\phi_{k-1}}(i) s_{k-1}$ $u^{\phi_k}(i))$
is part of this $u$-power suffix, so $|\phi_k(i)| - |c'| \leq 1$.
On the other had, the length of this $u$-power suffix is not longer
than the length of the $u$-power suffix of the whole
$\nf(s_0 u^{\phi_1}(i) s_1 \cdots u^{\phi_{k-1}}(i) s_{k-1} u^{\phi_k}(i))$,
so $|c'| - |\phi_k(i)| \leq 1$.
Hence $|\phi_k(i) - c'| \leq 1$.
On the other hand, $u^{c'}$ is both a $v$-run and a $u$-run,
so it is of length less than $|u| + |v|\leq 2|v|$.
Hence $u^{c'}$ is the $u$-power suffix of one of $\ov v ^2, \ov v, \varepsilon, v$ or $v^2$.

Now, in order to compute the appropriate pairs $(i,\psi) \in S_{I,E}$ and $i \in S_I'$,
we proceed as follows. After the initial $u$-reduction and $v$-reduction
($\Ocomp()$),
we choose a $u$-parametric power $u^{\phi_k}$, which has the associated  $v^{\psi_k}$-parametric power next to it, say to the right.
Then for $\phi^{\bullet} \in \{\phi_k-1, \phi_k, \phi_k+1\}$
and $\ell$ as one of $5$ possible lengths of $u$-power suffixes
(of $\ov v ^2, \ov v, \varepsilon, v$ or $v^2$)
we compute $i$ such that $\phi^\bullet(i) = \ell$.
We can do this in time $\Ocomp(n/|u|)$:
separately for each $u$-parametric power that is one of $\Ocomp(n/|u|)$ from Lemma~\ref{lem:how_many_different_powers_simple}
and separately for $\Ocomp(1)$ parametric powers from the same Lemma.
In the latter case we add all such computed $i$s to $S_I'$,
clearly there are $\Ocomp(1)$ of them.
In the former case for each such computed $i$ we compute
$W_{I = i}$ ($\Ocomp(m)$ time) and identify the unique $v$-power $v^\psi$
to which $u^{\phi_k}$ was merged.
We add the pair $(i,\psi)$ to $S_{I, E}$.
The running time is $\Ocomp(mn/|u|)$, as claimed.

The second case, when the $v$ parametric power $\ov v ^J$ was directly to the left of $\ov u ^I$,
is analogous: note that in this case the considered word
$s_0 u^{\phi_1} s_1 \cdots u^{\phi_{k-1}} s_{k-1} u^{\phi_k} s_k$
is simply $u^{\phi_k} s_k$ and the whole $\nf(u^{\phi_k}(i) s_k)$ is $u$-reduced to the $v$-parametric power,
in particular it is a $v$-power.
We analyze its $u$-power prefix of a $v$-power, the analysis is similar (in fact: simpler)
as in the first case, where we analyzed the $u$-power suffix of a $v$-power.

In the third case, in which there is no neighboring $v$-parametric power of $u^I$ or $\ov u ^I$,
the case assumption implies that the power is $\ov u^I$ and it is the first symbol in $W$.
Hence the considered word is again $u^{\phi_k} s_k$
and the rest of the analysis is identical as in the second case.

The fourth case, in which the $v$-parametric power next to $u^I$ was removed
during the preprocessing, is a bit more general, though similar to the first one.
Observe that there are at most $2n/|v| \leq 2n/|u|$ such occurrences of $u$-parametric powers,
as a factor in which two $v$-parametric powers are removed has its word $u_k$ of length at least $|v|$.
Such $u$-parametric powers can be identified by the choice of appropriate $v$-power that is removed in the preprocessing.
Recall that $u^{\phi_k}(i)$ reduces less $2|u|$ letters to the left
and less than $|u|$ to the right,
as $s_k$ does not begin with $u$ nor $\ov u$.
Let $u^{c'}$ for $|c'| \geq 3$ be the $u$-power that remained from $u^{\phi_k}(i)$ after the reductions
(remove the beginning and ending letters that do not form a power of $u$)
and let $s_{k}'$ be the word to the right of this $u^c$ in $\nf(s_0 u^{\phi_1}(i) s_1 \cdots u^{\phi_{k-1}}(i) s_{k-1} u^{\phi_k}(i) s_{k})$.
The $u$-power suffix $u^{c''}$ after the removal of $s_k'$,
so formally of $\nf(s_0 u^{\phi_1}(i) s_1 \cdots u^{\phi_{k-1}}(i) s_{k-1} u^{\phi_k}(i) s_{k} \ov{s_{k}'})$
satisfies $|\phi_k(i) - c''| < 3$:
we loose one $u$ on the right, reduce less than $2|u|$ letters on the left and if nothing is reduced on the left
then we can extend by less than $2|u|$ letters.
As in the first case, we move to the $v$-power suffix, say $v^c$,
of $\nf(u^{\phi_1}(i) s_1 \cdots u^{\phi_{k-1}}(i) s_{k-1} u^\phi(i) s_{k})$.
Then the $u$-power suffix of $\nf(v^c \ov {s_k'})$ is not longer than
the length of the $u$-power suffix of $\nf(u^{\phi_1}(i) s_1 \cdots u^{\phi_{k-1}}(i) s_{k-1} u^\phi(i) s_{k})$
and contains at least $u^{c'}$.
Moreover, this $u$-power suffix is also a $v$-run,
so it is of length less than $|u| + |v| \leq 2|v|$.
As $s_k'$ is known (for a fixed occurrence $u^{\phi_k}$),
as in the first case there are at most $6$ different
$u$-power suffixes of $\nf(v^c \ov{s_k'})$ for different $c$s.
The rest of the analysis and the algorithm finding the actual pairs and $i$s
is an in first case:
the choice of $\phi_k$ (there are $\Ocomp(n/|u|)$ of them)
determines a constant-size set of $i$s
and each $\phi_k, i$ (plus the direction left/right)
determines $\psi$.
\end{proof}

We now consider the case when $u^\phi$ was not affected.
Essentially, we claim that $u^\phi$ is almost the same as some $u^{\phi'}$ in $W$.
The difference is that it can $u$-reduce letters from $v$-parametric powers
that become $v$-powers.
However, as such $v$-power is not wholly merged (as it is not affected),
only its proper suffix or prefix can be $u$-reduced
and by primitivity and by case assumption $u \not \shift v$ and $|v| \geq |u|$, this suffix is of length at most $|v| + |u|$.
Thus, while in principle there are infinitely many possibilities for $v^{\psi}(j)$
when $j \in \mathbb Z$, it is enough to consider a constant number of different candidates (roughly: $\ov v ^2, \ov v, \varepsilon, v, v^2$)
and we can procure all of them so that an analysis similar to the one in Lemma~\ref{lem:how_many_different_powers}
can be carried out:
essentially we replace a fragment $v^J f_h u^I$ with $5$ ``fragments''
$v^{c} f_h u^I$ for $c \in \{-2,-1,0,1,2\}$.
In this argument, we used the assumption that $|v|\geq |u|$
(the $u$-reduction is of length at most $|v| + |u| \leq 2|v|$),
but it turns out that in the case $v$-parametric powers the argument is even simpler:
the $v$-reduced prefix of $u$-parametric power is of length at~most~$2|v|$,
so~the~$v$-parametric power is modified by an additive $\Ocomp(1)$ summand.

\begin{lemma}
\label{lem:how_many_different_powers_second_reduction}
We can compute and sort in $\Ocomp(m + n/|u|)$ time a set of $\Ocomp({\sqrt{n|u_{i_0}u_{i_0+1}|}}/{|u|})$
integer expressions $E$
such that for every $j$ if $u^\phi$ is a parametric power in $W_{J = j}$ not affected by substitution $J = j$ then $\phi \in E$.

A similar set of $\Ocomp({\sqrt{n|u_{i_0}u_{i_0+1}|}}/{|v|})$ integer expressions
can be computed for the not affected $v$-parametric powers after the second $v$-reduction
in $\Ocomp(m + n/|v|)$ time.
\end{lemma}

Note that the second $u$ (or $v$) reduction is performed only for some chosen values of $i$ and $j$,
and not for each possible one.

\begin{proof}
Fix some not affected $u^\phi$ and consider (the unique) $u^{\phi_2}$ that is merged to $u^{\phi'}$ that is not affected by the second $u$-reduction.
Consider the maximal parametric word $s_0v^{\psi_1}s_1 \ldots v^{\psi_k}s_k$
to the left of $u^{\phi_2}$ without a $v$-parametric power.
By the assumption $|\psi_\ell(j)|\geq 6$ and by Lemma~\ref{lem:no_long_self_reduction_after_reduction}
the reduction in each $v^{\psi_{\ell-1}}(j)s_\ell v^{\psi_\ell}(j)$ has length at most $2|v|$.
Hence none of those powers is reduced.

Consider $\nf(s_0v^{\psi_1}s_1 \ldots v^{\psi_k}s_k)$ and its $u$-power suffix.
The reduction in  $  v^{\psi_{k-1}}s_{k-1} v^{\psi_k}$
has length less than $2|v|$ and there are no further reduction on the left.
As $s_k$ does not have a prefix $v$ nor $\ov v$, the reduction on the right
has length less than $|v|$. Hence there are at least $3|v|$ letters left from $v^{\psi_k}$,
which are a $v$-run.
The part of them that are $u$-reduced into $u^\phi$ are also a $u$-run,
which means that less than $2|v|$ letters are $u$-reduced.
Hence the $u$-power suffix of $\nf(s_0v^{\psi_1}s_1 \ldots v^{\psi_k}s_k)$
is the same as the $u$-power suffix of $\nf(v^3s_k)$, when $\psi_k(j)\geq 6$,
or $\nf(\ov v^3s_k)$, when $\psi_k(j)\leq -6$ or $\varepsilon$, when $k=0$,
i.e.\ there are no $v$-parametric powers.

We make a similar analysis in the case of the $u$-parametric power directly to the right of $u^{\phi_2}$,
let $s_{k+1}$ be the word between $u^{\phi_2}$ and the following $v$-parametric power
(or the end of the word, if there is not such parametric power).
Hence $u^\phi$ is obtained by $u$-reduction of $u^{\phi_2}$ for a word
$v^\bullet s_k u^{\phi_2} s_{k+1} v^{\bullet \bullet}$,
where $v^\bullet, v^{\bullet \bullet} \in \{\ov v ^3, \varepsilon, v^3\}$.
We procure a word such that each such subword is in it,
and $u$-reduction of each such subword can be made separately.

Fix an occurrence of the $u$-parametric power $u^\phi$ in $W$ (from~\eqref{eq:parametric_equation})
consider the subword of the equation including the two neighboring parametric powers,
say $v^\psi u_h' u^\phi u_{h+1}' v^{\psi'}$.
Then we create nine words of the form
$v^{\bullet} u_h'$ $u^\phi u_{h+1}'$ $v^{\bullet \bullet}$,
for $v^\bullet, v^{\bullet \bullet} \in \{\ov v ^3, \varepsilon, v^3\}$.
The other options is that instead of $v$-power expressions
there are $u$-power expressions at the ends or,
then we introduce no $v$s, so for, say,
$v^\psi u_h' u^\phi u_{h+1}' u^{\phi'}$
we introduce only three words of the form
$v^{\bullet} u_h' u^\phi u_{h+1}'$ for $v^\bullet \in \{\ov v ^3, \varepsilon, v^3\}$;
similarly, when there is no power expression (as the equation ends),
we do not introduce $v^\bullet, v^{\bullet \bullet}$.
We concatenate all such words, separating them by some unused symbol, say $Y$.

Note that each occurrence of $u_h'$ is copied at most $9$ times
in the new word:
if the two neighboring parametric powers are $u$ and $v$-parametric powers,
then it is used only for the $9$ new words for $u$;
if by two power expressions of $u$ then it is used for both of them,
but at most thrice for each,
if by two $v$-power expressions then it is not used at all;
lack of power expression at one end behaves similarly as the $v$-power expression.

In order to estimate the number of different $u$-power expressions,
we are going to reuse the argument from the proof of Lemma~\ref{lem:how_many_different_powers}
and see how the estimations change.
Consider a fixed $u^I$ parametric power after the preprocessing and its fragment, 
say $v^J \beta u_h \alpha u^I$, consider the case when to the right of it there is $v^J$.
After the first $u$-reduction it introduced
$u^c$ to the parametric power expression,
where $u^c$ is the $u$-power prefix of $\nf(\beta u_h \alpha)$.

Let $f_h'$ be such that $v^{c'} f_h'u^c = \nf(\beta u_h \alpha)$,
where $u^c$ is the $u$-power suffix of $\nf(\beta u_h \alpha)$
and $v^{c'}$ is the $v$-power prefix of $v^{c'} f_h'$.
In the prepared word we replace the $v^c f_h'$ with $v^3$ or $\ov v^3$ or $\varepsilon$,
we consider $v^3$, the case of $\ov v^3$ is similar, the case of $\varepsilon$ is trivial.
We show that we can define $\beta' u_h' \alpha'$ such that
\begin{itemize}
\item the $u$-power suffix of $\nf(\beta' u_h' \alpha')$ is the same as of $v^3 f'_h u^c$;
\item $\beta' \in \{ v^3, v^3 \beta \}$, in particular, $|\beta'| \leq 4 |u_{i_0}u_{i_0+1}|$;
\item $|u_h'| \leq |u_h|$;
\item $\alpha \suff \alpha'$.
\end{itemize}
(Note that for $\varepsilon$ we simply take $\alpha' = \alpha, \beta' = \beta$ and $u_h' = u_h$ instead.)
Consider the shortest suffix of $\beta u_h \alpha$ whose normal form is $f'_h u^c$, remove the remaining prefix (of $\beta u_h \alpha$);
$u_h'$ and $\alpha'$ are what is left from $u_h, \alpha$ after this removal, clearly they satisfy the claim;
$\beta'$ is what is left from $\beta$ plus leading $v^3$.
There are only two possibilities for $\beta$ --- 
either it is $v^3$ or $v^3 \beta$ --- 
we remove a prefix whose normal form is a power of $v$ or $\ov v$
and $\beta$ does not begin with $v$ nor $\ov v$.
So this removed prefix cannot be a prefix of $\beta$ and so 
either nothing is removed at all (and so $\beta' = v^3 \beta$),
or the whole $\beta$ is removed (and so $\beta' = v^3$).
As $|v| \leq |u_{i_0}u_{i_0+1}|$, so $|\beta'| \leq 3 |u_{i_0}u_{i_0+1}| + |\beta| \leq 4 |u_{i_0}u_{i_0+1}|$.
Now, when we analyze the number of possible $u$-power suffixes of $\nf(\beta' u_h' \alpha')$
over all fragments of the prepared word,
the analysis from Lemma~\ref{lem:how_many_different_powers} applies with a couple of twitches:
\begin{description}
\item[$\alpha$] The only property of $\alpha$ used in Lemma~\ref{lem:how_many_different_powers} is that it does not end with $u$ nor $\ov u$
and this applies to $\alpha'$.
\item[$\beta$]
The $\beta$ was fixed for all fragments, now there are five variants ($v^3, \ov v ^3, v^3\beta, \ov v ^3\beta$, $\beta$),
and the length estimation is now $3 |u_{i_0}u_{i_0+1}|$ for $v^3$ and $\ov v^3$, 
$4 |u_{i_0}u_{i_0+1}|$ for $\ov v ^3\beta$ and $v^3 \beta$ and $|u_{i_0}u_{i_0+1}|$ for $\beta$.
Thus, using Lemma~\ref{lem:different_powers_in_a word}, the number of different $u$ maximal powers in all of them
is at most $\sqrt{5 \cdot 15|u_{i_0}u_{i_0+1}|/|u|}$, so $\sqrt {15}$ times larger than in case
Lemma~\ref{lem:how_many_different_powers}.
\item[$u_h'$] As $|u_h'| \leq |u_h|$ and each $u_h$ from the original equation is copied at most $9$ times in the prepared word,
the sum of lengths of all such words is $9n$.
Thus the number of possible lengths of maximal powers in all such $u_h'$s is
at most $\sqrt{5 \cdot 9n|/|u|} = 3\sqrt{5 \cdot n|/|u|}$,
so $3$ times more than in case of Lemma~\ref{lem:how_many_different_powers}.
\end{description}

In original equation after preprocessing to the right of the $u^I$ there is either $v^J$ or a trivial fragment.
In the first case in the prepared word there is either $v^3$ or $\ov v^3$ or fresh symbol to the right of $u^I$
and so during the second $u$-reduction we extend by either $u$-power prefix of $v^3$ or $\ov v^3$ or $\varepsilon$,
so one of three fixed numbers,
i.e.\ in comparison with Lemma~\ref{lem:how_many_different_powers}
this increases by $3$ the number of possible $u$-power expressions.

If to the right there is a trivial fragment then 
by the way the trivial fragments are constructed the parametric power to the right
is either $\ov u^I$ or $v^J$.
In the first case, the trivial fragment stays the same
and the analysis is as in the case of Lemma~\ref{lem:how_many_different_powers},
the only difference is that now each trivial fragment is copied at most $9$ times,
so the estimation on the sum of their lengths is $9$ times larger,
and so the number of different $u$-powers in them is $3$ times the one in
Lemma~\ref{lem:how_many_different_powers} (as we take the square root according to Lemma~\ref{lem:different_powers_in_a word}).
In the other case, let the trivial factor be $\nf(f_{h+1}\cdots f_{h+k})$,
let $u^d f'_{h+1}v^{d'} = \nf(f_{h+1}\cdots f_{h+k})$ be such that
$u^d$ is the $u$-power prefix of $\nf(f_{h+1}\cdots f_{h+k})$
and $v^{d'}$ is the $v$-power suffix of $f'_{h+1}v^{d'}$.
Then we replace $\nf(f_{h+1}\cdots f_{h+k})$ with $u^df'_{h+1}v^{3}$.
Also note that when this trivial fragment
is of the form $\nf(f_{h+1}\cdots f_{h+k})$
then we upper bound its length by $\sum_{\ell=1}^k |u_{h+\ell}|$.
Now we need to upper bound $|u^df'_{h+1}v^{3}| \leq |\nf(f_{h+1}\cdots f_{h+k})| + |v^3|$
but $u_{h+1}$ is a non-trivial power of $v$,
as the corresponding fragment trivialized,
so $|u_{h+1}| \geq |v|$ and so
\begin{align*}
|u^df'_{h+1}v^3|
	&\leq
|u^df'_{h+1}v^{d'}| + 3|v|\\
	&\leq
|\nf(f_{h+1}\cdots f_{h+k})| + 3|v|\\
	&\leq
\sum_{\ell=1}^k |u_{h+\ell}| + 3|u_{h+1}|\\
	&\leq
4 \sum_{\ell=1}^k |u_{h+\ell}| \enspace .
\end{align*}
Furthermore, there are at most nine copies of such fragments.
Thus when we estimate the number of different $u$-parametric powers in all of them using 
 Lemma~\ref{lem:different_powers_in_a word} the sum of their lengths is at most $36n$
 instead of $n$ in Lemma~\ref{lem:how_many_different_powers},
 so the estimation on the number of different $u$-powers increases $6$ times.

Thus, in the end, the number of possible $u$-parametric power is $\Ocomp(1)$ times larger than in case of Lemma~\ref{lem:how_many_different_powers}.

If the $u^\phi$ comes from a power in a fragment $\ov u^I \ov \alpha u_{h} \ov \beta \ov v^J$
then the analysis is symmetrical.
If it comes form a fragment $\ov u^I \ov \alpha u_{h} \alpha u^I$
then the analysis is only simpler: nothing changes in this fragment.
And for a fixed $u$-parametric power, say the left one,
we need only to consider what happens to the right: is there $\ov v^J$ or a trivial fragment there.

Note that the construction of the prepared word is explicit
and the word itself is $\Ocomp(m)$-represented.
The reduction takes $\Ocomp(m)$ time
and it is then still $\Ocomp(m)$-represented.
Finally, to compute the $u$-parametric powers it is enough for each $u$-parametric power
compute the $u$-power prefix of the word to the right and $u$-power suffix of the word to the left,
the computation takes $\Ocomp(m)$ time in total.
This yields all integer expressions, but perhaps with duplicates.
They can be removed in $\Ocomp(m + n/|u|)$ as in Lemma~\ref{lem:how_many_different_powers}.

In case of $v$-powers the argument is much simpler, as $|v| \geq |u|$.
Similarly, consider the unique $v^{\psi_2}$ that is merged to $v^{\psi}$ that is not affected by the second $v$-reduction,
let $s_0u^{\phi_1}s_1 \ldots u^{\phi_k}s_k$ be the maximal parametric word to the left of $v^{\psi_2}$ that does not contain a $v$-parametric power.
As in the case of $u$, none of
$u^{\phi_1}(i), \ldots u^{\phi_k}(i)$ is reduced.
Consider the $v$-power suffix of $\nf(s_0u^{\phi_1}s_1 \ldots u^{\phi_k}s_k)$.
As $u^{\phi_k}$ is not merged to $v^{\psi}$ (and it is not reduced),
the $v$-power suffix of $\nf(s_0u^{\phi_1}s_1 \ldots u^{\phi_k}s_k)$
is within $\nf(u^{\phi_k}s_k)$.
The part of $u^{\phi_k}$ that is $v$-reduced in $v^\psi$ is both an $u$-run and $v$-run,
so has less than $2|v|$ letters.
Moreover, as $s_k$ does not have an $v$ nor $\ov v$ as a suffix,
less than $|v|$ of its letters from it are used in $v$-reduction.
A similar analysis applies on the right-hand side.
Hence $|\psi - \psi_2| \leq 4$.
Hence we can take the set from Lemma~\ref{lem:how_many_different_powers} and replace each integer expression $\psi'$ with five copies
$\psi' - 2, \psi' - 1, \psi', \psi' +1, \psi'+2$,
which yields the desired set.
Note that we can eliminate the duplicates, as the set of candidates in Lemma~\ref{lem:how_many_different_powers} is sorted.
\end{proof}

Lemmata~\ref{lem:how_many_different_powers_simple}, \ref{lem:how_many_different_powers},
\ref{lem:dissapearing_v_powers} and \ref{lem:how_many_different_powers_second_reduction}
are enough to prove Lemma~\ref{lem:indices_sets_description}, by a simple case distinction,
as described in text preceding Lemma~\ref{lem:how_many_different_powers_simple}.

\begin{proof}[proof of Lemma~\ref{lem:indices_sets_description}]
First, let us consider some degenerate cases. If $W$ in~\eqref{eq:parametric_equation}
depends on one variable, say on $I$, and $(i,j)$ is a solution
then by Lemma~\ref{lem:parametric_solution_almost_0}
there is a $u$-parametric power $u^\phi$ such that $|\phi(i)|\leq 3$.
From Lemma~\ref{lem:how_many_different_powers_simple} the set of exponents has size $\Ocomp\left(\frac{\sqrt{n|u_{i_0}u_{i_0+1}|}}{|u|}\right) \leq
\Ocomp(n/|u|)$
and can be computed in $\Ocomp(m + n/|u|)$ time.
So the superset of possible $i$s also has this size and can be computed in the same time,
we set this as $S_{I,\mathbb Z}$ from the statement.
A similar analysis is carried out when the parametric powers do not depend on $I$.

So suppose $W$ in~\eqref{eq:parametric_equation} depends on both $I, J$.
Let $(i,j)$ be a solution, substitute $J = j$ in $W$ from~\eqref{eq:parametric_equation}
and compute (as a mental experiment) the second $u$-reduction of $W(J = j)$, obtaining $W_{J = j}$.
If $(i,j)$ is a solution, then $W_{J = j}(i) \eqg \varepsilon$
(this includes the case when $W_{J = j} = \varepsilon$).
By Lemma~\ref{lem:parametric_solution_almost_0}
there is a $u$-parametric power $u^\phi$ such that $|\phi(i)| \leq 3$
(if $W_{J = j} = \varepsilon$ then simply take $\phi = 0$).

If $\phi$ was not affected by $j$,
then by Lemma~\ref{lem:how_many_different_powers_second_reduction}
we can compute in $\Ocomp(m + n/|u|)$ time a set $E$ of integer expressions
of size $\Ocomp\left(\frac{\sqrt{n|u_{i_0}u_{i_0+1}|}}{|u|}\right)$
such that $\phi \in E$.
Define $S_I$ as the set of numbers $i'$ such that $|\phi'(i')| \leq 3$ for some $\phi' \in E$,
clearly it can be computed and sorted in $\Ocomp(m + n/|u|)$ time
and it has size $\Ocomp\left(\frac{\sqrt{n|u_{i_0}u_{i_0+1}|}}{|u|}\right)$.
And $i \in S_I$, as claimed.

So consider the case when $\phi$ was affected by substitution $J = j$
(this includes $\phi = 0$, as in this case at least one parametric power depending on $I$
and one on $J$ are merged).
By Lemma~\ref{lem:dissapearing_v_powers}
either $j$ is from a constant-size set, its elements are then added to $S_{J}$,
or $(\phi,j)$ is from a set of $\Ocomp(n/|v|)$ elements.
If $\phi = 0$ then we add the set of second components to $S_{\mathbb Z, J}$.
If $\phi \neq 0$ then from the fact that $|\phi(i)| \leq 3$
for each pair $(\phi,j)$ we can create at most $7$ pairs $(i,j)$ such that
$|\phi(i)| \leq 3$, we add them to $S_{I,J}$.
Clearly all those pairs can be computed and sorted in $\Ocomp(m n/|v|)$ time,
see Lemma~\ref{lem:dissapearing_v_powers}.

The proof for the second claim is symmetric.
\end{proof}

What is left to show is how to compute candidate solutions, when one of $I, J$, say $J$, is already fixed,
as in the claim of Lemma~\ref{lem:indices_sets_description}.
The analysis is similar as in the case of two parameters,
however, we cannot guarantee that after the $u$-reduction
the coefficient at the $u$-parametric powers are $\pm 1$.
On the positive side, as there is only one integer variable, we can apply Lemma~\ref{lem:parametric_solution_almost_0} directly.
The additional logarithmic in the running time is due to sorting,
which now cannot be done using counting sort, as the involved numbers may be large.
\begin{lemma}
\label{lem:one_parameter_fixed}
For any given $j$ in $\Ocomp(m)$ time
we can decide, whether for each $i \in \mathbb Z$ the~$\alpha u^iv^j\beta$
is a solution of~\eqref{eq:main} and if not then in
$\Ocomp(m + n \log m/|u|)$ time
compute a superset (of size $\Ocomp(n/|u|)$) of $i$s
such that $\alpha u^iv^j\beta$ is a solution.

A similar claim holds for any fixed $i$
(with superset size $\Ocomp(n/|v|)$ and running time $\Ocomp(m + n\log m/|v|)$).
\end{lemma}

\begin{proof}
We show the first claim, the second follows in the same way.
We proceed as follows: consider $\alpha u^i v^j \beta$ and compute
the $u$-power prefix $u^c$ of $\nf(v^j \beta)$, this can be done in $\Ocomp(1)$ time,
set $\beta' = \nf(u^{-c}v^j\beta)$.
Then $\alpha u^i v^j \beta \eqg \alpha u^{i+c} \beta'$.
Thus in the following we consider the problem of finding $i$ such that
$\alpha u^i \beta'$ is a solution of~\eqref{eq:main}.
Note that we do not have any bound on $|\beta'|$,
we do know that it is reduced and does not have a prefix $u$ nor $\ov u$.
Moreover, $\beta'$ is represented as a concatenation of a $v$-run and $2$-represented word,
so is of the form in Lemma~\ref{lem:solution_testing_2}.

We substitute $\alpha u^I \beta'$ for $X$ and perform the $u$-reduction.
During and after the reduction the equation is represented
as a concatenation words,
each is either $1$-represented, a $v$-run, a $u$-run or $u$-parametric power.
We store $w_1 \cdots w_{i}$, which is $u$-reduced and equivalent
to the read parameterized word,
moreover, when $w_j$ is a $u$-parametric power, then $w_{j+1}$
is the longest common prefix of $w_{j+1}\cdots w_i$ and $u$ or $\ov u$
(as $u$ is cyclically reduced, this is well defined).
Note, that as the word is $u$-reduced,
there are no two consecutive $u$-parametric powers
and a word before (after) a $u$-parametric power does not end (begin) with $u$ nor $\ov u$.

Suppose that we processed $w_1 \cdots w_{i-1}$ and the stored string is $w_1'\cdots w_{i'-1}'$.
We read $w$. If it is a $1$-parametric word or a run then we first reduce
it with previous words one by one.
By Lemma~\ref{lem:solution_testing}--\ref{lem:solution_testing_2}
this takes $\Ocomp(1)$ time (this is assigned to $w$)
plus $\Ocomp(1)$ time per words removed from $w_1'\ldots w_{i'-1}'$.
Let $w'$ be what remained from $w$ after the reduction
and $w_{i'-1}$ again denote the last stored word.
If the word after the $u$-parametric power is a prefix of $u$
(or $\ov u$, the case are symmetric, we consider only the former)
after the $u$-parametric power or it is the $u$-parametric power,
we should update the $u$-parametric power and the prefix after it.
We describe the case when $u_{i'-1}' = u' \pref u$, the other are done in the same way, let also $u = u'u''$.
We check compute the longest common prefix of $u''u'$ and $w'$,
we can deduce it from the length of $\nf(\ov{u'} \, \ov{u''} w')$,
which can be computed in $\Ocomp(1)$ by Lemma~\ref{lem:solution_testing_2}.
If the prefix is of length $|u|$, we also compute the longest prefix of $w'$
that has period $|u|$:
if it has length $p$ then $|\nf(\ov{w'[|u|+1\twodots |w'|]}, w')| = 2|w'|- |u| - 2p$,
and this can be computed in $\Ocomp(1)$ time by Lemma~\ref{lem:solution_testing_2}.
Thus we can compute the $u$-power prefix of $u'w'$ in $\Ocomp(1)$ time,
we remove it from the words, add it to the exponent in the $u$-power prefix
and also the longest common prefix of $u$ and what remains from $u'w'$ after the removal of the $u$-power prefix.

So suppose that the next $w$ is a $u$-parametric power $u^\phi$
(in fact it can be only $u^I$ or $\ov u ^I$, but this has no effect on the algorithm).
If $w_{i'-1}$ is also a $u$-parametric power, then we replace them by one (by adding the exponents).
If the exponent depends on the variable then we are done.
If it does not then the parametric power is equal to a $u$-run,
and we proceed as if we read this $u$-run.
The $\Ocomp(1)$ cost that should be charge to this $u$-run is charged to the removed $u$-parametric power instead.
If $w_{i'-1}$ is a word, then let $w_j$ be the previous $u$-parametric power.
By the assumption that $w_{j+1}\cdots w_i$ is a $u$-reduced,
this word is not a power of $u$.
We compute the $u$-power suffix of $w_{j+1}\cdots w_i$,
this is done similarly as in the previous case.
Let $w_{j+1}\cdots w_i = w^\bullet u^p$.
Then we replace $w_{j+1}\cdots w_i$ with $w^\bullet$ (appropriately represented),
and put $w_{i'} = u^{\phi+k}$ on the list.

The processing time is clearly $\Ocomp(1)$ per read symbol and $\Ocomp(1)$ per symbol removed.
As we add symbols only once, this yields total $\Ocomp(m)$ processing time.

The obtained parameterized word $w$ is a concatenation of $\Ocomp(m)$ $1$-represented words and $u$-runs and $v$-runs or $u$-parametric powers.
If $w$ does not depend on $I$ then either each $i$ yields a solution (when $w = \varepsilon$) or none (when $w \neq \varepsilon$).
In the other case, i.e.\ when $w$ does depend on $I$,
if $w(i) \eqg \varepsilon$ then by Lemma~\ref{lem:parametric_solution_almost_0} 
there is a $u$-parametric power $u^\phi$ in $w$ such that $|\phi(i)| \leq 3$.
We want to show that by removing some (though not necessarily all)
duplicates the set of exponents of $u$-parametric powers in $w$
we can construct a multi-set, i.e.\ possibly with duplicates, $E$ of $\Ocomp(n/|u|)$ integer expressions.
This will yield the claim:
we consider the $u$-parametric powers one by one and for each $u^{\phi}$
compute all $i$s such that $|u^{\phi}(i)|\leq 3$,
clearly there are at most $7$ such $i$s for one integer expression.
We can eliminate the duplicates by simple sorting in $\Ocomp(n/|u| \log m)$ time,
(as there are at most $m$ exponents).

Among words $u_1, \ldots, u_m$ there are at most $n/|u|$ that have length at least $|u|$.
Thus there are at most $2 n / |u|$ $u$-parametric powers $u^\phi$
that include parametric power that was neighboring one of such long $u_\ell$s
and $u^\phi$ is uniquely determined by the choice of $u_\ell$ and the direction left/right.
Choosing one of possible $7$ values of $\phi(i)$ yields that there are $14 n/|u|$ possible values of $i$.
Note that the proof above gives and effective procedure to compute them
in $\Ocomp(m)$ time.

So consider the case that $u^\phi$ is obtained from
\begin{equation*}\gamma u_{\ell-1} (\alpha u^I \beta')^{p_\ell} u_{\ell} \cdots u_{h+\ell-1} (\alpha u^I \beta')^{p_{\ell+h}} u_{\ell+h} \delta \enspace,
\end{equation*}
where $\gamma \in \{\beta', \ov \alpha \}$ and $\delta \in \{\ov {\beta'}, \alpha\}$;
note that the $u$-reduction does not necessarily consume the whole $\gamma u_{\ell-1}$ nor $u_{\ell+h} \delta$.
Let us first consider a degenerate case, when $h=0$,
i.e.\ there is only one $u$-parametric power, take $p_\ell = 1$,
the other case is symmetric.
Consider the $u$-power prefix of $\nf(\beta' u_\ell \delta)$.
If $\delta = \ov {\beta '}$ then from Lemma~\ref{lem:u_prefix_beta_w_ovbeta}, first point,
the $u$-power prefix is of length less than $2|u| + |u_\ell| < 3|u|$.
If $\delta = \alpha$ then from again from Lemma~\ref{lem:u_prefix_beta_w_ovbeta}
there are only $\Ocomp(1)$ possible lengths of its $u$-power prefix.
A symmetric analysis applies to the $u$-power suffix of $\nf (\gamma u_{\ell-1} \alpha)$
and so there are $\Ocomp(1)$ many possible $u$-parametric powers obtained 
as a $u$-reduction of $\gamma u_{\ell-1} \alpha u^I \beta' u_\ell \delta$.
Note that those integer expressions can be effectively computed:
during the $u$-reduction we can identify the expressions that were obtained
in this way and simply add them to a set, removing the duplicates in constant time (as the set is of constant size).
For each such expression $\phi$ we compute $i$s such $|\phi(i)|\leq 3$,
clearly there are at most $7$ for one $\phi$.
Then we add them to the superset,
as there are $\Ocomp(1)$ many of them, we can remove the duplicates made by this union in $\Ocomp(n/|u|)$ time.

Now, if $h>0$, i.e.\ there are at least two parametric powers that are $u$-reduced,
then $p _\ell = p _{\ell+1} = \cdots = p_{\ell+h}$:
suppose not and consider the first moment, when the $u$-reduction merges two integer expression with different signs at $I$:
then the word between them is $\nf (\beta' u_k \ov{\beta'})$  or $\nf(\alpha u_k \ov \alpha)$ for some $k$
and it is a $u$-run.
Then by Lemma~\ref{lem:u_prefix_beta_w_ovbeta}, third point,
we obtain that $|u_k| \geq |u|$, which we excluded.
By symmetry, assume that $p _\ell = 1$.
Then each $\nf(\beta' u_k \alpha)$ is a power of $u$
and so from Lemma~\ref{lem:sum_of_powers} we have 
$\nf(\beta' u_k \alpha) \in \{\ov u ^2, \ov u, 1, u, u^2\}$,
as the $u$-power prefix of $\beta'$ and $u$-power suffix of $\alpha$ are trivial and $|u_k|<|u|$.
It is left to consider, what are the possible $u$-power suffix
of $\gamma u_{\ell - 1} \alpha$ and $u$-power prefix of $\beta' u_{\ell+h} \delta$,
we analyse the latter, the analysis of the former is symmetric.
If $\delta  = \ov {\beta'}$ then we already have shown (in the case when $h=0$) that
the $u$-power prefix of $\nf(\beta' u_{\ell+1}\ov {\beta'})$ is of length at most $2|u|$.
If $\delta  = \alpha$ then we use that fact that $h>0$: we know that $\nf(\beta' u_\ell \alpha) = u^c$
for some $|c| \leq 2$.
Let $u^{c'}$ be the $u$-power prefix of $\nf(\beta' u_{\ell+1} \alpha)$,
then $u^{c'-c}$ is a $u$-power prefix of
$\nf(\ov \alpha \, \ov u_\ell \ov {\beta'} \beta' u_{\ell+1} \alpha)
=
\nf(\ov \alpha \, \ov u_\ell u_{\ell+1} \alpha)$.
From Lemma~\ref{lem:u_prefix_beta_w_ovbeta}, first case,
we get that the $u$-power prefix of $\nf(\ov \alpha \, \ov {u_\ell} u_{\ell+1} \alpha)$
is of length at most $2|u| + |\ov {u_\ell} u_{\ell+1}| < 4|u|$.
So $|c' - c| < 4$ and so $|c'| < 6$.
A similar analysis applies to the $u$-power suffix of $\nf(\gamma u_\ell \alpha)$,
which is also of length at most $5|u|$.
So in total the obtained parametric expression is of the form
$\phi = (h+1)J + c''$, where $|c''|\leq 2h+10$,
as $c''$ is a a sum of lengths of $u$-power suffix of $\nf(\gamma u_\ell \alpha)$
(at most $5|u|$), $h$ powers obtained from $\beta' u_{\ell+h'} \alpha$ (each of length at most $2|u|$)
and $u$-power prefix of $\nf(\beta' u_{\ell+h} \alpha)$ (at most $5|u|$).

Now, if $\phi(i) = c_0$ for some $|c_0| \leq 3$
then $(h+1)i + c'' = c_0$ and so
$
i = \frac{c_0 - c''}{h+1}
$, let us estimate its absolute value:
\begin{align*}
|i|
	&=
\frac{|c_0 - c''|}{h+1}\\
	&\leq
\frac{|c_0|}{h+1} + \frac{|c''|}{h+1}\\
	&\leq
\frac{3}{2} + \frac{2h+10}{h+1}\\
	&\leq
\frac{3}{2} + 6\\
	&=
7\frac 1 2 \enspace ,
\end{align*}
and so $|i| \leq 7$.
So instead of actually computing all the values, we simply add numbers $-7, \ldots, 7$
to the superset;
as there are $\Ocomp(1)$ many of them, we can remove the duplicates made by this union in $\Ocomp(n/|u|)$ time.
\end{proof}

\subsubsection{$u \in \{v, \ov v\}$}\label{sec:u-in-v-ov-v}
When $u \in \{v, \ov v\}$ then $u^Iv^J \eqg u^{I+J}$ or $u^Iv^J \eqg u^{I-J}$
and we can replace the parameter $I+J$ (or $I-J$) with a single $I$.
This case is subsumed by the case when we fix one of the parameters (i.e.\ $I$ or $J$), see Lemma~\ref{lem:one_parameter_fixed}.

\subsubsection{$u \shift v$}
\label{sec:u-shift-v}

In this case either $u = u'u''$ and $v = u''u'$ or $v  =\ov{u'} \, \ov{u''}$,
for some $u', u''$.
By substituting $v = \ov v$ we reduce the latter case to the former.
We consider the parametric solution $\alpha u^Iv^J\beta$,
note that $v \eqg u'' u \ov {u''}$ and so
$\alpha u^Iv^J \beta \eqg \alpha u^I u'' u^J \ov{u''}\beta$.
From now on the approach is similar as when $u \not \shift v$.
Most of the arguments are simpler, however, the extra technicality is that after the $u$-reduction we can have $u$-parametric power
of the form $u^{\pm(I+J) + c}$.
As a result, we consider not only substitutions $I = i$ and $J = j$,
but also $I + J = k$, i.e.\ we substitute $\phi$ with $\phi(I,k-I)$,
which depends only on $I$.
This requires some additional cases to consider and makes some formulations longer, but everything follows in a similar way.

We substitute $X = \alpha u^I u'' u^J \beta$ to~\eqref{eq:main} and proceed similarly as in Section~\ref{sec:u-not-sim-v}:
we rotate the equation, so that if it is of the form $\alpha u^I u'' u^J w \eqg \varepsilon$,
we replace it with $u'' u^J w \alpha u^I \eqg \varepsilon$
and replace $\ov \beta \, \ov u^J \, \ov{u''} \ov u^I \, \ov \alpha w \eqg \varepsilon$
with $\ov u^I \, \ov \alpha w \ov \beta \, \ov u^J \, \ov{u''} \eqg \varepsilon$.
The $u''$ and $\ov{u''}$ are not assigned to any fragment,
otherwise each subword between two consecutive parametric powers is a fragment.
Then we preprocess the equation: we reduce the words between the parametric powers
and then for each fragment of the form $\ov u^I f_h u^I$ or $u^J f_h \ov u^J$,
such that $f_h$ is a power of $u$
we replace it with $f_h$. Finally we concatenate such neighboring words,
as well as all neighboring $u''$ or $u''$,
obtaining trivial fragments.

\begin{lemma}
\label{lem:preprocessing_2}
The preprocessing can be performed in $\Ocomp(m)$ time.
The obtained parametric word is a concatenation of fragments
and trivial fragments as well as some $u''$ and $\ov{u''}$
and it is $\Ocomp(m)$-represented.
A trivial fragment obtained from fragments with words $f_{h}, \ldots, f_{h+k}$ has length less than
$2|u| + \sum_{i = h}^{h+k} |f_i| \leq 3\sum_{i = h}^{h+k} |f_i|$
and if $k>0$ then such trivial fragment is not a power of $u$.
\end{lemma}
\begin{proof}
The proof follows as in the case of Lemma~\ref{lem:preprocessing},
with the additional twist: observe that the trivial fragment is formed not only as the concatenation
of powers of $u$ but could also include $u''$ and $\ov{u''}$ from between the parametric powers.
Represent $u = u' u''$ and note that $u' u''$ is cyclically reduced.
Suppose that a trivial fragment is a power of $u$.
Then from Lemma~\ref{lem:concatenation_of_powers_is_not_a_power}
we could transform such an equality to the form that concatenation
of powers of $u$ and $u''$ is equivalent to $\varepsilon$,
replacing $u$ with $u'u''$ would yield a similar representation for powers of 
$u', u''$, and the Lemma~\ref{lem:concatenation_of_powers_is_not_a_power} 
yields that $u', u''$
are powers of the same word, which means that $u$ is not primitive.
\end{proof}

We $u$-reduce the equation as before, this time though it can happen that
a $u$-parametric power depending on $I + J$ is created
(when the whole fragment becomes one parametric power),
though parametric powers from different fragments are not $u$-reduced,
see Lemma~\ref{lem:reduction_2}.
In the end we obtain equation of the form
\begin{equation}
\label{eq:parametric_equation_2}
W \eqg \varepsilon
\end{equation}
where $W$ is a parameterized word.
We call $(i,j)$ a solution of~\eqref{eq:parametric_equation_2} if indeed $W(i, j) \eqg \varepsilon$.

\begin{lemma}
\label{lem:reduction_2}
\label{lem:different_u_v_reduction_2}
Every $u$-parametric power in~\eqref{eq:parametric_equation_2}
has an exponent of the form $n_II + n_JJ + c$,
where
\begin{itemize}
\item $n_J = 0$ and $n_I \in \{-1,+1\}$ or
\item $n_I = 0$ and $n_J \in \{-1,+1\}$ or
\item $n_I = n_J \in \{-1,+1\}$.
\end{itemize}
\end{lemma}
\begin{proof}
Consider two consecutive $u$-parametric powers that are $u$-reduced together as first.
If they are not from the same fragment, then between them is either $u''$ or $\ov {u''}$,
which cannot be $u$-reduced,
or a trivial fragment.
By Lemma~\ref{lem:preprocessing_2} if a trivial fragment is a $u$-power then it is formed
from exactly one fragment.
But then the word is of the form
$u'' u^\ell \ov{u''}$ or $\ov{u''} u^\ell u''$
(or $u'' u^\ell$ or $u^\ell u''$ when it is on the beginning or end of $W$)
for $\ell \neq 0$,
in all cases it is not a power of $u$.

If they are from the same fragment, then this fragment is 
$u^I f_h u^J$ or $\ov u^J f_h \ov u^I$
(as in the preprocessing we removed $u$-parametric powers $u^I f_h \ov u^I$ and $\ov u^j f_h u^J$ such that $f_h$ is a power of $u$),
the analysis is symmetric, so we consider $u^I f_h u^J$.
Then indeed coefficients at $I$ and $J$ in the exponent are the same.
By the argument as above, this $u$-parametric power cannot be $u$-reduced with another $u$-parametric power (as they are separated by $u''$ or $\ov{u''}$).
\end{proof}

As in Section~\ref{sec:u-not-sim-v}, we can estimate the number of different $u$-parametric powers in $W$ from~\eqref{eq:parametric_equation_2}.

\begin{lemma}
\label{lem:how_many_different_powers_2_simple}
\label{lem:how_many_different_powers_2}

There is a set $S$ of $\Ocomp(1)$ size of integer expressions
such that there are $\Ocomp(n/|u|)$ occurrences of $u$-parametric powers
in~$W$ from~\eqref{eq:parametric_equation_2} whose exponents are not in $S$.
The set $S$ can be computed and the parametric powers identified in $\Ocomp(m + n /|u|)$ time.

Furthermore, there are $\Ocomp({\sqrt{n|u_{i_0}u_{i_0+1}|}}/{|u|})$
different integer expressions as exponents in parametric powers of $u$ in~\eqref{eq:parametric_equation_2}.
The (sorted) lists of such expressions can be computed in $\Ocomp(m + n/|u|)$ time.
\end{lemma}
\begin{proof}
For the first claim we proceed similarly as in Lemma~\ref{lem:how_many_different_powers_simple}:
suppose that the parametric power $u^\phi$ is the result of $u$-reduction of $u^J$ from fragment $u^J \beta u_h \alpha u^I$.
If $|u_h| \geq |u|$ or this parametric power $u^J$ is next to a trivial fragment,
then this $u^\phi$ is one of the chosen parametric powers;
other powers and types of fragments are handled similarly.
Clearly there are at most $4n/|u|$ such chosen powers:
there are at most $n/|u|$ such $u_h$s and trivial fragments (as a trivial fragment also corresponds to $u_h$ of length at least $|u|$)
and one can be counted for two parametric powers.

Otherwise its exponent is from a constant-size set:
if the exponent depends on only one variable,
say in $u^J \beta u_h \alpha u^I$,
then $\nf(\beta u_h \alpha)$ is not a power of $u$
and the $u$-power prefix and suffix overlap at less than $u$ letters,
see Lemma~\ref{lem:different_runs_overlap}.
Thus the constant in the exponent in the parametric power is (almost) the exponent of the $u$-power prefix or suffix of $\nf(\beta u_h \alpha)$,
where $|u_h| < |u|$,
which is analyzed exactly as in Lemma~\ref{lem:how_many_different_powers_simple}.

For the parametric powers depending on $J + J$ say created from a fragment $u^J \beta u_h \alpha u^I$,
first note that they do not include letters
from the neighboring trivial fragments, 
as by case assumption there are no such trivial fragments.
So the constant in the exponent
is $k$, where $\nf (\beta u_h \alpha) = u^k$.
However, as $\beta$ does not begin and $\alpha$ does not end with $u$ nor $\ov u$
and $|u_k| < |u|$ from Lemma~\ref{lem:sum_of_powers}
we  get that $|k| \leq 2$.

Consider the second claim of the Lemma: note that the argument for the $u$-parametric powers with exponents depending on one variable
is exactly the same as in Lemma~\ref{lem:how_many_different_powers}:
suppose that we compute the $u$-parametric power for $u^J$
in $u^J \beta u_h \alpha u^I$.
The difference between the $u$-power prefix of $\nf(\beta u_h \alpha)$ and what was $u$-reduced is of length at most $|u|$:
it could be that we first $u$-reduced the $u^I$ on the right,
but as $\nf(\beta u_h \alpha)$ is not a power of $u$,
the $u$-power prefix and suffix can overlap by less than $|u|$ letters, see Lemma~\ref{lem:different_runs_overlap}.
Otherwise, we estimate the number of possible $u$-power prefixes (and suffixes) of fragments in exactly the same way
as in Lemma~\ref{lem:how_many_different_powers},
the difference is the estimation on the length of $\beta$,
which is at most $|u_{i_0}u_{i_0+1}|$ in Lemma~\ref{lem:how_many_different_powers}
and at most $2|u_{i_0}u_{i_0+1}|$ now,
this increase the constant by $\sqrt{2}$.
To the right there could be a trivial fragment,
we also estimate the number of possible $u$-power prefixes and suffixes in the same way,
with the only difference that we can have extra $u''$ or $\ov {u''}$ at one or both ends.
This increases the estimation of the length of the trivial fragment at most $3$ times,
so affects the constant by a factor of $\sqrt 3$.

Lastly, consider the $u$-parametric powers that depend on $I + J$,
say it was created from a fragment $u^J \beta u_h \alpha u^I$
(the other case is $\ov u^I \ov \alpha u_h \ov \beta \ov u^J$
and it is analyzed in a similar way).
If there are no trivial fragments to the left and right,
then the $u$-parametric powers is $u$-reduced exactly from
$u^I \beta u_h \alpha u^J$ and its exponent
is $I+J + k$, where $\nf(\beta u_h \alpha) = u^k$.
As $\beta$ does not begin and $\alpha$ does not end with $u, \ov u$,
from Lemma~\ref{lem:sum_of_powers} we get that there is a 
maximal $u$ power $u^{k'}$ in $u_h$ such that $|k - k'| \leq 2$.
The rest of the analysis is then as in Lemma~\ref{lem:how_many_different_powers}.
If there are trivial fragments $f,f'$ to the left and right
(to streamline the argument, one of them can be $1$),
then we are looking at a maximal $u$-power in
$f \nf(\beta u_h \alpha) f'$, where $|f| + |f'| \geq |u|$.
We estimate the sum of lengths of all such words, over possible $h$.
Note that $|\nf(\beta u_h \alpha)| < |u_h| + 2|u|$
and so $|\nf(f \beta u_h \alpha f')| < 3 |ff'| + |u_h|$
and so the sum of all those lengths over possible $h$ is at most
$19n$:
as $\sum |u_h| \leq n$, each trivial fragment is counted at most twice, the sum of lengths of all trivial fragments is less than $3n$,
by Lemma~\ref{lem:preprocessing_2}.
Hence by Lemma~\ref{lem:different_powers_in_a word} there are at most
$\sqrt{133 n}$ such different powers.
\end{proof}

Consider $W$ in equation~\eqref{eq:parametric_equation_2},
from Lemma~\ref{lem:parametric_solution_almost_0}
if $(i,j)$ is a solution of~\eqref{eq:parametric_equation_2}
then some exponent $\phi$ in the $u$-parametric power $u^\phi$ satisfies
$|\phi(i, j)| \leq 3$.
In principle, we could choose this exponent, the value of $I, J$
(so of $I$ or $J$, when it depends on one variable, and $I + J$ if on $I + J$)
substitute to $W$, and reuse the same argument.
However, the estimations on the possible exponents after the substitution are much weaker than the one in Lemma~\ref{lem:how_many_different_powers_2},
so instead we proceed as in Section~\ref{sec:u-not-sim-v},
but the argument is more subtle, as for the exponents that depend on $I + J$
the substitution is of the form $I + J = k$.
We say that a substitution depends on $I$ (on $J$, on $I + J$)
if it is of the form $I = i$ ($J = j$, $I+J = k$, respectively).
Note that the substitution $I + J = k$ is equivalent to $J = - I +k$
(and $I = -J + k$, we choose the former arbitrarily).

The argument follows the same line as in Section~\ref{sec:u-not-sim-v}:
when we substitute $I = i$ (or $J = j$ or $I + J = k$) and $u$-reduce the equation
in the obtained parametric word $W_{I = i}$ ($W_{J=j}$, $W_{I + J = k}$, respectively)
some new $u$-parametric power is (almost) $0$ for appropriate substitution.
This parametric power is almost the same as in~\eqref{eq:parametric_equation_2}
or it was affected by substitution, in which case we can link it with some other $u$-parametric power that was (almost) reduced due to the substitution.

Formally, we say that an occurrence of a $u$-parametric power $u^{\phi}$ in $W_{I + J = k}$
was affected by a substitution $I + J = k$
if more than one $u$-parametric power was merged to $u^{\phi}$ or
there are $u$-parametric powers $u^{\phi_1}$, $u^{\phi_2}$
in $W$ such that $|\phi_1(I, k - I)| < 6$, $u^\phi_2$ depends on $I$ or on $J$
and all $u$-parametric powers of $u$ between $u^{\phi_1}$ and $u^{\phi_2}$
depend on $I + J$.
Being affected by substitution $I = i$ and $J = j$ are defined similarly.
Note that we do allow a special case, when $W_{I+J = k} = \varepsilon$,
in which case $\phi = 0$ and all $u$-parametric powers were merged to it.

\begin{lemma}
\label{lem:dissapearing_v_powers_2}
There are sets $S_{I+J}, S_{E,I+J}$, $|S_{I + J}| = \Ocomp(1)$ and $|S_{E,I + J}| = \Ocomp(n/|v|)$
such that for each occurrence of $u$-parametric power $u^\phi$ in $W_{I + J = k}$
affected by substitution $I + J = k$ either $k \in S_{I + J}$ or $(\phi,k) \in S_{E,I + J}$.
These sets can be computed and sorted in $\Ocomp(m n/|u|)$ time.

Analogous claims hold for substitutions $I = i$ and $J = j$.
\end{lemma}
\begin{proof}
The proof is similar to the first (simpler) case of Lemma~\ref{lem:dissapearing_v_powers}.
Consider, why $u^\phi$ is affected by the substitution $I + J = k$.
It could be that there are two $u$-parametric powers $u^{\phi_1}$
and $u^\phi_2$ such that $\phi_1$ depends on $I + J$,
$|\phi_1(I, k - I)| < 6$,
each parametric power between $u^{\phi_1}$ and $u^{\phi_2}$
depends on $I + J$ and $\phi_2$ depends on $I$ or $J$.
Using Lemma~\ref{lem:how_many_different_powers_2_simple}
either $\phi_1$ is from a constant-size set $E$,
or it is one of $\Ocomp(n/|u|)$ occurrences
of parametric powers in $W$.
In the first case for each $\phi' \in E$ we add each $k$ such that $|\phi'(k)| < 6$ to $S_{I+J}$,
in the second case the choice of $\phi_1$ plus the direction (left or right)
determines $u^\phi$,
we add each such pair $(\phi,k)$ to $S_{E,I + J}$, clearly there are $\Ocomp(n/|u|)$ many of them.

The other reason for $u^{\phi}$ being affected is that there is more than one $u$-parametric power merged to $u^{\phi}$.
Observe that this means that some $u$-parametric power depending on $I + J$ was also merged to $u^\phi$:
suppose not and consider the first two parametric powers (that do not depend on $I + J$)
that are merged to $u^{\phi}$, there are two by the case assumption.
Then the word between them is a $u$-power in $W_{I + J = k}$
but it is not in $W$ (as they were not $u$-reduced), so it has a parametric power.
It cannot be a one depending on $I$ or $J$, it depends on $I + J$.

If there are two $u$-parametric powers depending on $I$ or $J$
merged to $u^\phi$ then consider the two such powers $u$-reduced first and the word
$w = s_0 u^{\phi_1} s_1 u^{\phi_2} \cdots u^{\phi_\ell}s_\ell$,
between them, where each $u^{\phi_\ell}$ depends on $I + J$,
note that $w(I, k - I)$ is equivalent to power of $u$.
If only one $u$-parametric depending on $I$ or $J$ was merged to $u^\phi$,
call it $u^{\phi'}$,
then consider the maximal word to the left and right of $u^{\phi'}$ in $W$ without a $u$-parametric power not depending on $I+J$..
Choose the one that contains a $u$-parametric depending on $I+J$:
it has to exist, as at least two parametric powers were merged.
Suppose that it is to the left, the other case is symmetric.
Let the word be $w = s_0 u^{\phi_1} s_1 u^{\phi_2} \cdots u^{\phi_\ell}s_\ell$,
where all letters in $u^{\phi_\ell}(i)s_\ell$ are $u$-reduced to $u^\phi$ (and perhaps some other as well).
This generalizes the previous case.

Then, as in Lemma~\ref{lem:dissapearing_v_powers},
the case when some $|\psi_\ell(j)| < 6$ for $\ell > 0$ was already covered and in the other case
it can be shown that from each $u^{\phi_\ell}(I, k - I)$
at least $2|u|$ letters remained.
In particular, there are $2|u|$ letters from $u^{\phi_\ell}(I, k - I)$
and they should also be a part of $u$-power suffix of
the whole word $\nf(w(I, \ell - I))$.
Then Lemma~\ref{lem:runs_of_different_words} implies
that those are the same run, which contradicts the fact that $s_\ell$ is not a power of $u$.

The running time analysis is as in Lemma~\ref{lem:dissapearing_v_powers}.

The analysis of other cases (of substitution $I = i$ or $J = j$) is identical.
\end{proof}

We now consider the $u$-parametric powers that were not affected by the substitution.
Note that the formulation is more involved then in case of Lemma~\ref{lem:how_many_different_powers_second_reduction},
there is a reason though.
Intuitively, we want to say that $u^{\phi'}$ in $W_{I + J = k}$ was not affected
(by the substitution $I + J = k$),
when it is (almost) the same as some $\phi$ in $W$.
However, if $\phi$ depends on $J$, say it is $\phi = J +1$ then
after the substitution we get $\phi' = - I + k +1$.
Moreover, varying $k$ yields infinitely many possible $\phi'$s.
Thus our set should include $\phi$ and we characterize $\phi'$ via $\phi$, i.e.\ as $\phi' = \phi(I,-I+k)$.

\begin{lemma}
\label{lem:how_many_different_powers_second_reduction_2}
We can compute in $\Ocomp(m + n/|u|)$ time a set of $\Ocomp({\sqrt{n|u_{i_0}u_{i_0+1}|}}/{|u|})$
integer expressions $E$ each depending on $I$ or $J$
such that for every $k$ if $u^{\phi'}$ is a parametric power in $W_{I + J = l}$ not affected by the substitution $I + J = k$
then there is $\phi \in E$ such that $\phi' = \phi(I, k - I)$.

Similar bounds hold also for parametric powers not affected
by substitutions $I = i$ or $J = j$.
\end{lemma}
\begin{proof}
The argument follows the lines of the simpler (second) case in Lemma~\ref{lem:how_many_different_powers_second_reduction},
with the extra care due to substitution $\phi(I, k - I)$.
We want to show that if $u^{\phi'}$ was not affected by the substitution $I + J = k$ then there is $\phi$ in $W$
such that $|\phi(I, k - I) - \phi'|\leq 2$.
Using the bounds from Lemma~\ref{lem:how_many_different_powers_2}
we get the claim of the Lemma.

Consider the unique $u^{\phi_{\ell+1}}$ that is merged to $u^{\phi'}$ that is not affected by the substitution $I + J = k$
and the maximal subword to the left of $u^{\phi_{\ell+1}}$ in $W$
(the one to the right is analyzed in the same way)
that does not include the $u$-parametric power depending on
$I$ or depending on $J$,
let this word be  $w = s_0u^{\phi_1}s_1 \ldots u^{\phi_\ell}s_\ell$,
where each $\phi_1, \ldots \phi_\ell$ depends on $I + J$.
Then it can be argued as in Lemma~\ref{lem:how_many_different_powers_second_reduction}
that none of $u^{\phi_\ell}(I, k - I)$ is reduced to less than $2|u|$ letters,
moreover, as $u^{\phi_\ell}$ and $u^{\phi_{\ell+1}}$ were not $u$-reduced together,
as $s_\ell$ is not a power of $u$.
Then the $u$-power suffix of $\nf(w(I, k - I))$ 
and what remained from $u^{\phi_\ell}(I, k - I)$ after the reductions
are different $u$-runs (as $s_\ell$ is not a power of $u$) and so they overlap for less than $|u|$ letters.
As $s_\ell$ does not begin nor end with $u$ nor $\ov u$
we conclude that the length of the $u$-power suffix of
$\nf(w(I, k - I))$ has length at most $|u|$.
The same applies to a similar word to the right of $u^{\phi_{\ell+1}}(I, k - I)$.
Hence the second $u$-reduction adds at most $|u|$ letter to the right and $u$ letters to the left of
$u^{\phi_{\ell+1}}(I, k - I)$ and so
$|\phi_{\ell+1}(I, k - I) - \phi'| \leq 2$, as claimed.

The analysis of the running time is the same as in Lemma~\ref{lem:how_many_different_powers_second_reduction_2}.

Lastly, the argument for other substitutions is the same.
\end{proof}

Overall, the main characterization is

\begin{lemma}
\label{lem:indices_sets_description_2}
Given the equation~\eqref{eq:parametric_equation}
we can compute in $\Ocomp(m n/|u|)$ time
sets $S_I, S_J,
S_{I + J}, 
S_{I+J, \mathbb Z} \subseteq \mathbb Z$
and 
$S_{I,J} \subseteq \mathbb Z^2$,
where $|S_I|, |S_J| = \Ocomp({\sqrt{n|u_{i_0}u_{i_0+1}|}}/{|u|})$,
$|S_{I + J}| = \Ocomp(1)$, $|S_{I+J, \mathbb Z}| = |S_{I,J}| = \Ocomp(n/|u|)$
such that
if $(i,j)$ is a solution of~\eqref{eq:parametric_equation_2} then
at least one of the following holds:
\begin{itemize}
\item $i \in S_I$ or
\item $j \in S_J$ or
\item $i+j \in S_{I+J}$ or
\item $i+j \in S_{I+J,\mathbb Z}$ and for each $i'$ the $(i',(i+j)-i')$ is a solution or
\item $(i,j) \in S_{I,J}$.
\end{itemize}

Similarly,
we can compute in $\Ocomp(m n/|u|)$ time
sets $S_I', S_J',
S_{I + J}', 
S_{I, \mathbb Z}' \subseteq \mathbb Z$
and 
$S_{I,J}' \subseteq \mathbb Z^2$,
where $|S_{I+J}'|, |S_J'| = \Ocomp({\sqrt{n|u_{i_0}u_{i_0+1}|}}/{|u|})$,
$|S_I'| = \Ocomp(1)$, $|S_{I, \mathbb Z}'| = |S_{I,J}'| = \Ocomp(n/|u|)$
such that
if $(i,j)$ is a solution of~\eqref{eq:parametric_equation_2} then
at least one of the following holds:
\begin{itemize}
\item $i \in S_I'$ or
\item $j \in S_J'$ or
\item $i+j \in S_{I+J}'$ or
\item $i \in S_{I,\mathbb Z}'$ and for each $j'$ the $(i,j')$ is a solution or
\item $(i,j) \in S_{I,J}'$
\end{itemize}
and
we can compute in $\Ocomp(m n/|u|)$ time
sets $S_I'', S_J'',
S_{I + J}'', 
S_{\mathbb Z, J}'' \subseteq \mathbb Z$
and 
$S_{I,J}'' \subseteq \mathbb Z^2$,
where $|S_{I+J}''|, |S_I''| = \Ocomp({\sqrt{n|u_{i_0}u_{i_0+1}|}}/{|u|})$,
$|S_J''| = \Ocomp(1)$, $|S_{\mathbb Z,J}''| = |S_{I,J}''| = \Ocomp(n/|u|)$
such that
if $(i,j)$ is a solution of~\eqref{eq:parametric_equation_2} then
at least one of the following holds:
\begin{itemize}
\item $i \in S_I''$ or
\item $j \in S_J''$ or
\item $i+j \in S_{I+J}''$ or
\item $j \in S_{\mathbb Z, J}''$ and for each $i'$ the $(i',j)$ is a solution or
\item $(i,j) \in S_{I,J}''$.
\end{itemize}
\end{lemma}

\begin{proof}
We proceed similarly as in Lemma~\ref{lem:indices_sets_description}.
First, let us consider some degenerate cases.
If all $u$-parametric powers in $W$ in~\eqref{eq:parametric_equation_2}
depend on the same set of variables,
say on $I + J$, and $(i,j)$ is a solution
then by Lemma~\ref{lem:parametric_solution_almost_0}
there is a $u$-parametric power $u^\phi$ such that $|\phi(i+j)|\leq 3$.
The rest of the analysis is an in Lemma~\ref{lem:indices_sets_description}:
from Lemma~\ref{lem:how_many_different_powers_2_simple}
either $\phi$ is one of $\Ocomp(1)$ parametric expressions
and so there are only $\Ocomp(1)$ possible values of $i+j$
and those are added to $S_{I+J}$;
or $u^\phi$ is one of $\Ocomp(n/u)$ occurrences of $u$-parametric powers in $W$.
Then there are at most $\Ocomp(n/|u|)$ different values of
$i+j$ such that
$\phi(I, i + j - I) = 0$, and for each such $i+j$
each $(i', (i+j) - i')$ is a solution,
each such $i+j$ is added to $S_{I+J, \mathbb Z}$

The running time analysis is as in Lemma~\ref{lem:indices_sets_description}.

So suppose $W$ in~\eqref{eq:parametric_equation_2}
has $u$-parametric powers depending on two different sets of variables,
so in particular it has a parametric power that does not depend on $I + J$.
Let $(i,j)$ be a solution, substitute $I + J = i + j$ in $W$ from~\eqref{eq:parametric_equation_2}
and compute $W_{I + J = i + j}$.
Recall that by the convention, the substitution means that we replace $J$ by
$-I +i+j$.

As in proof of Lemma~\ref{lem:indices_sets_description}
we conclude that $W_{I + J = i+j}(i) \eqg \varepsilon$,
this includes the case when $W_{I + J = i + j} = \varepsilon$.
Then there is a $u$-parametric power $u^\phi$ in $W_{I + J = i+j}$
such that $|\phi(i)| \leq 3$ (if $W_{I + J = i + j} = \varepsilon$ then simply we take $\phi = 0$).

If $\phi$ was not affected by substitution $I + J = i + j$,
then by  Lemma~\ref{lem:how_many_different_powers_second_reduction_2}
we can compute in $\Ocomp(m + n/|u|)$ time a set $E$ of integer expressions (depending on $I$ or $J$)
of size $\Ocomp\left(\frac{\sqrt{n|u_{i_0}u_{i_0+1}|}}{|u|}\right)$
such that there is $\phi' \in E$ such that $\phi = \phi'(I, i + j - I)$.
Define $S_I$ ($S_J$)as the set of numbers $i'$ ($j'$) such that $|\phi''(i')| \leq 3$ for some $\phi'' \in E$ depending on $I$
($|\phi''(j')| \leq 3$ for some $\phi'' \in E$ depending on $J$).
Observe that if $\phi'$ depends on $I$ then $i \in S_I$:
in this case $\phi = \phi'$ and so
$|\phi'(i)| = |\phi(i)| \leq 3$ and so $i \in S_I$.
If $\phi'$ depends on $J$ then let
$\phi = n_I I + c$ then $\phi' = n_I(-J + i+j) + c$
and $\phi'(j) = n_I(-j + i+j) + c = n_Ii + c=\phi(i)$
and thus $|\phi'(j)| = |\phi(i)| \leq 3$
and so $j \in S_J$.

The running time analysis is as in Lemma~\ref{lem:indices_sets_description}.

So consider the case when $\phi$ was affected by the substitution $I + J = i + j$
(this includes $\phi = 0$, as in this case at least two parametric powers depending on different sets of variables are merged).
By Lemma~\ref{lem:dissapearing_v_powers_2}
either $i+j$ is from a constant-size set, whose elements are added to $S_{I+J}$;
or $(\phi,i+j)$ is from a set of $\Ocomp(n/|v|)$ elements.
If $\phi = 0$ then we add the set of second components to $S_{(I + J), \mathbb Z}$.
Note that if $W_{I + J = i +j} \eqg \varepsilon$ then for every $i'$
the pair $(i', i+j-i')$ is a solution, as the sum $i' + i +j -i' = i+j$, and so $W(i',i +j -i') \eqg W_{I + J = i +j} (i) \eqg \varepsilon$.
If $\phi \neq 0$ then it is not a constant,
i.e.\ it depends on a variable ($I$) by definition.
From the fact that $|\phi(i)| \leq 3$ for each $\phi$
there are at most $7$ possible values of $i$
and so for each pair $(\phi,i+j)$ we can create at most $7$ pairs $(i,j)$ such that
$|\phi(i)| \leq 3$, we add them to $S_{I,J}$.
Clearly there are $\Ocomp(n/|u|)$ such pairs added.
The running time analysis is as in
Lemma~\ref{lem:indices_sets_description}.

The proof for other claims is symmetric.
\end{proof}

The fourth possibility in Lemma~\ref{lem:indices_sets_description_2} means that $W(I,k-I) \eqg \varepsilon$,
which would yield an infinite family of solutions
$\{ \alpha u^i v^{k-i}\beta \: : \: i \in \mathbb Z\}$.
Additional combinatorial analysis yields that this cannot happen
(and we know this from the earlier characterisation of the solution set).

\begin{lemma}
	\label{lem:no_i+j_trivilizes}
	Consider a parametric word $\alpha u^I v^J \beta$ for $u \shift v$
	and the corresponding $W \neq \varepsilon$ obtained after the substitution of $X = \alpha u^I v^J \beta$,
	as in~\eqref{eq:parametric_equation}.
	Then for every $k$ it holds that $W(I, k-I) \not \eqg \varepsilon$.
\end{lemma}
\begin{proof}
	Suppose first that in $W$ there is a $u$-parametric power depending on $I$ or $J$ (so not only ones depending on $I + J$).
	We show that in $W(I,k-I)$ no two $u$-parametric powers are $u$-reduced;
	note that those have to depend on $I$,
	as after the substitution $I + J = k$ the $u$-parametric powers depending on $J$
	become $u$-parametric powers depending on $I$ and $u$-parametric powers
	depending on $I + J$ are turned to $u$-powers.

	Suppose not, consider an arbitrary sequence of $u$-reductions and the first moment when two $u$-parametric powers are $u$-reduced,
	call them $u^\phi$ and $u^{\phi'}$ and let $u^\phi$ be the left one.
	First, observe that in between them (in $W(I,k-I)$) there cannot be a $u$-parametric power,
	as this would contradict the fact that they were the first  parametric powers to be $u$-reduced.
	So in $W$ between $u^\phi$ and $u^{\phi'}$ there only $u$-parametric powers depending on $I+J$.
	So one of $u^\phi, u^{\phi'}$ was the first parametric power in its fragment and the other the second in its fragment:
	otherwise the number of parametric powers between them before the preprocessing would be odd, and during the preprocessing we remove two parametric powers in one step
	and also creation of $u$-parametric power depending on $I + J$ uses two parametric powers.

	Suppose that $u^\phi$ was created from the left parametric power in its fragment
	(and $u^{\phi'}$ from the right-parametric power in its fragment).
	Consider the parametric power $u^\psi$ that was created from the right parametric power
	in the same fragment as $u^\phi$ and similarly $u^{\psi'}$ that was created from the left parametric power in the fragment in which $u^{\phi'}$ was created.
	Then in preprocessing we do not remove the $u^{\psi}, u^{\psi'}$:
	the other powers in their fragments, i.e.\ $u^\phi, u^{\phi'}$, are not removed in the preprocessing.
	By the choice of $u^\phi, u^{\phi'}$ they are not $u$-reduced, contradiction, as they are between $u^\phi, u^{\phi'}$.

	So $u^\phi$ is obtained from the $u$-parametric power on the right of the fragment and $u^{\phi'}$ from the parametric power on the left of the fragment.
	Consider the word between those two parametric powers, before the preprocessing.
	It is a concatenation of fragments and $u''$ or $\ov {u''}$ in between,
	it also begins with $u''$ or $\ov{u''}$ and ends with $u''$ or $\ov{u''}$.
	There is no $u$-parametric power in $W(I,k-I)$ left from the powers between $u^{\phi}$ and $u^{\phi'}$,
	so each such fragment either was replaced by a power of $u$ during the preprocessing
	(this power is not $\varepsilon$)
	or was turned into a parametric power depending on $I+J$ and turned into a power of $u$
	by substitution $I + J = k$ (those ones could be equal to $\varepsilon$).
	So the word between $u^{\phi}$ and $u^{\phi'}$ is of the form
	(each power of $u$ comes from one fragment, it could be that $p_i = 0$)
	\begin{equation}
	\label{eq:powers_are_powers}
	(u'')^{p_1}(u'u'')^{p_2}\cdots (u'u'')^{p_{2\ell}} (u'')^{p_{2\ell+1}} \eqg (u'u'')^{p}
	\end{equation}
	where each $p_k \in \{-1,1\}$ for odd $k$.
	
	By Lemma~\ref{eq:powers_to_powers}, the mapping $S \mapsto u'$, $T \mapsto u''$
	is an isomorphism between the group freely generated by $S, T$ and the group generated by $u', u''$.
	Hence the rewriting procedure that takes a concatenation of powers of $u', u''$
	and reduces $u' \ov{u'}$, $\ov{u'}u'$, $u'' \ov{u''}$, $\ov{u''}u''$ leads to a~unique normal form.
	In~\eqref{eq:powers_are_powers} the right-hand side is already in such a normal form.
	Consider any reduction that leads to the normal form and the corresponding partial pairing
	of the left-hand side.
	The right hand-side begins and ends with a power of $u''$,
	so at least one of $(u'')^{p_1}$ or $(u'')^{p_{2\ell_1}}$ is paired.
	Consider paired $u'', \ov {u''}$ such that between them there is no $u''$ nor $\ov {u''}$.
	Clearly there are such paired $u'', \ov{u''}$:
	we have a paired $u'', \ov{u''}$ and when they are paired,
	each $u', u'', \ldots$ between them is also paired.
	Then the whole word between them is also paired or $\varepsilon$,
	so between them there is $(u'u'')^0 = \varepsilon$.
	We obtain a $u$-power with exponent $0$ only from
	a $u$-parametric power depending on $I + J$, as $u$-powers obtained from
	$u$-reducing a fragment into a trivial trivial fragment are not equal to $\varepsilon$
	(as, say, it replaces $u^I u_h \ov u ^I$ with $u_h$).
	But a $u$-parametric power depending on $I+J$ has to its sides either $u''$ and $u''$
	or $\ov{u''}, \ov {u''}$, contradiction, as the one that we found has paired $u''$ and $\ov{u''}$.
	
	The other case is that there is no $u$-power depending on $I$ or $J$ in $W$
	(and $W(I,k-I) \eqg \varepsilon$).
	When we represent $W$ as a concatenation of fragments and $u''$ or $\ov{u''}$,
	then it either begins with $u''$ or it ends with $\ov{u''}$
	and the same argument as in the main case shows that this $u''$ or $\ov{u''}$
	should reduce with some other $\ov{u''}$ or $u''$ inside (as the whole $W(I,k-I)$ reduces to $\varepsilon$).
	Then the rest of the argument is the same.	
\end{proof}

\clearpage

\subsection{Algorithm and running time}\label{sec:running-time}
\begin{lemma}
\label{lem:running_time}
All candidate solutions can be tested in total $\Ocomp(n^2m)$ time.
\end{lemma}
\begin{proof}
Lemma~\ref{lem:solutions_superset} yields that there
are $\Ocomp(n^2)$ solutions that satisfy the statement (the ones from set $S$).

We show that we test $\Ocomp(n^2)$ solutions, each in time $\Ocomp(m)$.
We separately estimate the time needed to test the infinite solution sets
as well as the time spent on other subprocedures.

By Lemma~\ref{lem:testing_a_single_solution}
one solution of the form $\alpha u^i v^j \beta$, for fixed $\alpha, u,i,v,j,\beta$ which are $\Ocomp(1)$-represented,
can be tested in $\Ocomp(m)$ time.
So it is enough to show that there are at most
$\Ocomp(n^2)$ different candidates tested
(we estimate other computation times as well).
Lemma~\ref{lem:solutions_superset} yields that there are $\Ocomp(n^2)$ candidate solutions
(from the set $S$).
Other solutions are obtained in the following way:
for two consecutive words $u_{i_0}, u_{i_0+1}$ from the equation
we have a~family of~$\ell_{i_0} \cdot \ell_{i_0}'$ candidates of the form
$\alpha u^Iv^J \beta$, see Lemma~\ref{lem:solutions_superset},
where $\ell_{i_0} = |u|, \ell_{i_0}' = |v|$ and $\ell_{i_0},\ell_{i_0}' \leq |u_{i_0}|+|u_{i_0+1}|$;
by Lemma~\ref{lem:solutions_superset} the total time, over all $i_0$, spent on computing words 
$\alpha, \beta, u, v$ is $\Ocomp(n^2)$ time.
We will often use the estimation (a similar one hold for $\ell_{i_0}'$):
\begin{equation}
\label{eq:_l_i_sum}
\sum_{i_0=1}^{m} \ell_{i_0} \leq \sum_{i_0=1}^m |u_{i_0}|+|u_{i_0+1}| \leq 2n \enspace .
\end{equation}

Suppose first that $u \not \shift v$,
then by Lemma~\ref{lem:indices_sets_description}
we can compute in time $\Ocomp(m n/\ell_{i_0})$
sets  $S_I$, $S_J$, $S_{J,\mathbb Z}$, $S_{I,J}$,
where $|S_I| = \Ocomp\left({\sqrt{n|u_{i_0}u_{i_0+1}|}}/{l_{i_0}} \right)$,
$|S_J| = \Ocomp(1)$, $|S_{\mathbb Z,J}|, |S_{I,J}| = \Ocomp(n/\ell_{i_0})$,
such that for each solution $(i,j)$ at least one of the following holds:
\begin{enumerate}[{i}1.]
\item $i \in S_I$ or \label{i1}
\item $j \in S_J$ or \label{i2}
\item $j \in S_{\mathbb Z, J}$ and $(i',j)$ is a solution for each $i'$ or \label{i3}
\item $(i,j) \in S_{I,J}$. \label{i4}
\end{enumerate}
and in time $\Ocomp(m n/\ell_{i_0}')$
sets  $S_{J}'$, $S_I'$, $S_{I,\mathbb Z}'$, $S_{I,J}'$,
where $|S_{I}'| = \Ocomp(1)$,
$|S_{J}'| = \Ocomp({\sqrt{n|u_{i_0}u_{i_0+1}|}}/{\ell_{i_0}'})$ 
$|S_{I, \mathbb Z}'|, |S_{I,J}'| = \Ocomp(n/\ell_{i_0}')$,
such that for each solution $(i,j)$ at least one of the following holds:
\begin{enumerate}[{j}1.]
\item $j \in S_{J}'$ or \label{j1}
\item $i \in S_{I}'$ or \label{j2}
\item $i \in S_{I,\mathbb Z}$ and $(i,j')$ is a solution for each $j'$ or \label{j3}
\item $(i,j) \in S_{I,J}'$. \label{j4}
\end{enumerate}
As both of those characterization hold, we should describe how do we treat each of the 16 cases.
Fortunately, for most of the cases the further action and analysis depends on one of the cases alone.

If we are in the case~\iref{4} or \jref{4}
then we test each pair $(i,j) \in S_{I,J} \cup S_{I,J}'$ separately.
There are (over all $0 \leq i_0 \leq m-1$) at most
(note that some of those solutions have $u \shift v$,
we will estimate their running time separately, so now we overestimate the running time)
\begin{equation}
\label{eq:n_u_quadratic}
\sum_{i_0=0}^{m-1} \ell_{i_0} \ell_{i_0}' \left( \frac n {\ell_{i_0}} + \frac n {\ell_{i_0}'} \right)
	=
n \sum_{i_0=0}^{m-1} \ell_{i_0}' + \ell_{i_0} \leq 4n^2 \quad \quad \text{ by~\eqref{eq:_l_i_sum}}
\end{equation}
such solutions.

Concerning the time of establishing those sets,
the largest is from Lemma~\ref{lem:dissapearing_v_powers}
and it is $\Ocomp(mn/\ell_{i_0})$ (for $S_{I,J}$) or $\Ocomp(mn/\ell_{i_0}')$ (for $S_{I,J}'$).
So up to a constant it is:
\begin{equation*}
\sum_{i_0=0}^{m-1}\ell_{i_0}\ell_{i_0}' \left( \frac{mn}{\ell_{i_0}} + \frac{mn}{\ell_{i_0}'}\right) = 
mn \sum_{i_0=0}^{m-1} (\ell_{i_0} +\ell_{i_0}') \leq 2mn^2 \quad \quad  \text{by~\eqref{eq:_l_i_sum}} \enspace .
\end{equation*}

If we are in the case~\iref{3}
then for each $j \in S_{J,\mathbb Z}$ we substitute $J = j$
and test, whether $W(I,j) \eqg \varepsilon$;
by Lemma~\ref{lem:parametric_solution_almost_0} this is equivalent
to $(i',j)$ being a solution for each $i' \in \mathbb Z$.
Each such $j$ yields a family of solutions
of the required form
$\{ \underbrace{\alpha}_{\text{fixed}} u^i \underbrace{v^j \beta}_{\text{fixed}} \: : \: i \in \mathbb Z\}
$
and there are at most $| S_{J,\mathbb Z}| = \Ocomp(n/\ell_{i_0})$ such families. 
Over all $0 \leq i_0 \leq m-1$ this yields at most (up to a~constant)
\begin{equation*}
\sum_{i_0 = 0}^{m-1}\ell_{i_0}\ell_{i_0}' \frac n {\ell_{i_0}'}
	=
n \sum_{i_0 = 0}^{m-1} \ell_{i_0} \leq 2n^2 \quad \quad \text{ by~\eqref{eq:_l_i_sum}} \enspace .
\end{equation*}
Concerning the running time,
note that testing whether $W(I,j) \eqg \varepsilon$
takes $\Ocomp(m)$ time, see Lemma~\ref{lem:one_parameter_fixed},
so it is enough to show that
we test $\Ocomp(n^2)$ such $j$s.
As $|S_{\mathbb Z, J}| = \Ocomp(n/\ell_{i_0})$,
the calculations are as in~\eqref{eq:n_u_quadratic}.
A similar analysis applies to $S_{I,\mathbb Z}$,
i.e.\ case \jref{3}.

If we are in case \iref{2} then for each $j \in S_J$
we can compute, by Lemma~\ref{lem:one_parameter_fixed}, in time $\Ocomp(m)$
whether each $(i',j)$ is a solution, note that the set is of the required form,
as in the case of $j \in S_{\mathbb Z, J}$,
moreover the estimation on the number of such solution sets is not larger than in the case of $j \in S_{\mathbb Z, J}$, as $|S_J| = \Ocomp(1)$ and
$|S_{\mathbb Z, J}| = \Ocomp(n/\ell_{i_0})$.
Otherwise, again by Lemma~\ref{lem:one_parameter_fixed},
we compute in time $\Ocomp(m + n/\ell_{i_0} \log m))$ a set $S$ of size $|S| = \Ocomp(n/\ell_{i_0})$
such that if $(i,j)$ is a~solution then $i \in S$.
This yields $|S_J| \times |S| = \Ocomp(n/\ell_{i_0})$ candidate pairs,
which are individually tested, so the running time is $\Ocomp(m n / \ell_{i})$, note that this dominates $\Ocomp(m + n \log m/\ell_{i_0})$ from Lemma~\ref{lem:one_parameter_fixed}.
The estimation in~\eqref{eq:n_u_quadratic} yields that there are at most $\Ocomp(n^2)$ such candidate pairs and the whole running time is
$\Ocomp(m n^2)$.
A similar analysis applies to $i \in S_{I}'$.

The only remaining option is that we are simultaneously in case \iref{1} and \jref{1},
i.e.\ $i \in S_I$ and $j \in S_{J}'$.
As
$|S_I| = \Ocomp\left({\sqrt{n|u_{i_0}u_{i_0+1}|}}/{\ell_{i}} \right)$
and
$|S_{J}'| = \Ocomp\left({\sqrt{n|u_{i_0}u_{i_0+1}|}}/{\ell_{i}'} \right)$
there are (over all $i_0$ and up to a constant) at most
\begin{equation*}
\sum_{i_0=1}^m \ell_{i_0} \ell_{i_0}' \frac{\sqrt{n|u_{i_0}u_{i_0+1}|}}{\ell_{i_0}} \cdot 
\frac{\sqrt{n|u_{i_0}u_{i_0+1}|}}{\ell_{i_0}'}
	=
\sum_{i_0=1}^m n|u_{i_0}u_{i_0+1}|
	\leq
2n^2
\end{equation*}
such solutions tested.

The cases of $u = v$ or $u = \ov v$
are done using Lemma~\ref{lem:one_parameter_fixed},
the bounds are the same as in case of $u \not \shift v$.

The case of $u \shift v$ is a bit more involved,
let $\ell_{i_0} = |u|$.
By Lemma~\ref{lem:indices_sets_description_2}
we can compute in time $\Ocomp(m n/|u|)$
sets $S_I$, $S_J$, $S_{I+J}$, $S_{I+J,\mathbb Z}$, $S_{I,J}$
and
such that for each solution $(i,j)$ either
\begin{enumerate}
\item $i \in S_I$ or
\item $j \in S_J$ or
\item $i+j \in S_{I+J}$ or
\item $i+j \in S_{I+J,\mathbb Z}$ and for each $i'$ the $(i',(i+j)-i')$ is a solution or
\item $(i,j) \in S_{I,J}$.
\end{enumerate}
The third case is dealt with as previously (for each $i+j$ we make a substitution $I + J = i + j$, check, whether the obtained equation is trivial
and solve the corresponding equation),
similarly fourth (for each $i+j$ we substitute $I + J = i+ j$ and check whether the obtained word is $\varepsilon$;
note that it can be shown that this never holds,
see Lemma~\ref{lem:no_i+j_trivilizes})
and fifth (we~substitute $I = i, J = j$ and test).
So we are left only with the first two cases.
Moreover, Lemma~\ref{lem:indices_sets_description_2} also gives us a similar characterization resulting from a substitution
$I = i$, again there are 5 cases and the last three of them are dealt with similarly, the first two give that there are sets
$S_J', S_{I+J}'$ such that
\begin{enumerate}
\item $j \in S_J'$ or
\item $i+j \in S_{I+J}'$
\end{enumerate}
and applied to substitution $J = j$ again gives 5 cases, the last three of which are dealt with and the first two yield
that there are sets 
$S_I'', S_{I+J}''$ such that
\begin{enumerate}
\item $i \in S_I''$ or
\item $i+j \in S_{I+J}''$.
\end{enumerate}
There are in total $8$ cases (we choose one of two options for three substitutions),
in each such a case from the three choices some two (though not each two) allow to give
$\Ocomp\left({n|u_{i_0}u_{i_0+1}|}/{\ell_{i_0}^2}\right)$ 
candidates for $(i,j)$ :
say if $i \in S_I$, $j \in S_J'$ and $i+j \in S_{I+J}''$ then any two determine $(i,j)$
and when $j \in S_J$, $j \in S_J'$ and $i+j \in S_{I+J}''$ then
$j \in S_J \cap S_J'$ and $i+j \in S_{I+J}''$.
The rest of the calculations is the same as in the case of $u \not \shift v$.
\end{proof}

\end{document}